%% file: FHNpaper.tex
\documentclass[10pt]{articleHJ}
\usepackage[centertags]{amsmath}
\usepackage{amsfonts,amssymb,amsthm}
\usepackage{caption,graphicx}
\usepackage{subcaption}
\usepackage[text={6.0in,8.6in},centering,letterpaper]{geometry}
\usepackage[numbers,comma,square,sort&compress]{natbib}

\usepackage{tabularx}

\usepackage{mathtools}

\newlength{\defbaselineskip}
\newcommand{\setlinespacing}[1]%
           {\setlength{\baselineskip}{#1 \defbaselineskip}}

\newcommand{\singlespacing}{\setlength{\baselineskip}{\defbaselineskip}}

\newcommand{\R}{\ensuremath{\mathbb{R}}}
\newcommand{\C}{\ensuremath{\mathbb{C}}}
\renewcommand{\b}{\ensuremath{\beta}}
\newcommand{\e}{\ensuremath{\varepsilon}}
\newcommand{\p}{\ensuremath{\partial}}
\renewcommand{\k}{\ensuremath{\kappa}}
\renewcommand{\d}{\ensuremath{\delta}}
\newcommand{\g}{\ensuremath{\gamma}}
\newcommand{\G}{\ensuremath{\Gamma}}
\newcommand{\s}{\ensuremath{\sigma}}

\renewcommand{\l}{\ensuremath{\lambda}}
\renewcommand{\a}{\ensuremath{\alpha}}

\renewcommand{\L}{\ensuremath{\mathcal{L}}}
\renewcommand{\O}{\ensuremath{\mathcal{O}}}

\newcommand{\oV}{\ensuremath{\overline{V}}}

\DeclarePairedDelimiter{\ip}\langle\rangle
\DeclarePairedDelimiter{\nrm}\lVert\rVert

\def\Re{\mathop\mathrm{Re}\nolimits}			
\newcommand{\Real}{\mathbb{R}}							
\newcommand{\abs}[1]{\left\vert#1\right\vert}			
\newcommand{\norm}[1]{\left\Vert#1\right\Vert}		
\newcommand{\sref}[1]{(\ref{#1})}                       

\newtheorem{thm}{Theorem}[section]
\newtheorem{cor}[thm]{Corollary}
\newtheorem{lem}[thm]{Lemma}
\newtheorem{prop}[thm]{Proposition}

\newtheorem*{defn*}{Definition}

 {\begin{trivlist}\item[]\textbf{Acknowledgments }}{\end{trivlist}}

\begin{document}

\numberwithin{equation}{section}
\renewcommand{\theequation}{\thesection.\arabic{equation}}

\begin{frontmatter}
\title{Stability of Traveling Waves \\ for Systems of Reaction-Diffusion Equations \\
  with Multiplicative Noise }
\journal{SIMA}

\author[LD1]{C. H. S. Hamster\corauthref{coraut}},
\corauth[coraut]{Corresponding author. }
\author[LD2]{H. J. Hupkes},
\address[LD1]{
  Mathematisch Instituut - Universiteit Leiden \\
  P.O. Box 9512; 2300 RA Leiden; The Netherlands \\
  Email:  {\normalfont{\texttt{c.h.s.hamster@math.leidenuniv.nl}}}
}
\address[LD2]{
  Mathematisch Instituut - Universiteit Leiden \\
  P.O. Box 9512; 2300 RA Leiden; The Netherlands \\ Email:  {\normalfont{\texttt{hhupkes@math.leidenuniv.nl}}}
}

\date{Version of \today}

\begin{abstract}
\singlespacing

We consider reaction-diffusion equations
that are stochastically forced by a small multiplicative noise term.
We show that spectrally stable traveling wave solutions
to the deterministic system retain their orbital stability
if the amplitude of the noise is sufficiently small.

By applying a stochastic phase-shift
together with a time-transform, we obtain a
quasi-linear SPDE that describes the fluctuations from the
primary wave. We subsequently follow the semigroup approach developed in \cite{Hamster2017} to handle the nonlinear stability question. 
The main novel feature is that we no longer require
the diffusion coefficients to be equal.
\end{abstract}



\end{frontmatter}

\section{Introduction}

In this paper we consider stochastically perturbed versions
of a class of reaction-diffusion equations
that includes the FitzHugh-Nagumo equation
\begin{equation}
\label{eq:int:fzhnag:pde}
\begin{array}{lcl}
u_t & = & u_{xx} + f_{\mathrm{cub}}(u) - w
\\[0.2cm]
w_t & = & \varrho w_{xx} + \e [ u - \gamma w ] .
\end{array}
\end{equation}
Here we take $\e,\varrho,\gamma > 0$ and consider
the standard bistable nonlinearity
\begin{equation}
f_{\mathrm{cub}}(u) = u ( 1 - u )( u -a ).
\end{equation}
It has been known for quite some time that this system admits spectrally (and nonlinearly)
stable traveling pulse solutions when
$(\varrho, \gamma, \e)$ are all small \cite{alexander1990topological}. Recently, such results have
also become available for the equal-diffusion setting $\varrho = 1$
by using variational techniques
together with the Maslov index \cite{chen2015traveling,cornwell2017opening,cornwell2017existence}.

Our goal here is to show that these spectrally
stable wave solutions survive in a suitable
sense upon adding a small pointwise
multiplicative noise term to the underlying PDE. In particular,
we generalize previous results in \cite{Hamster2017} where we were only able to consider the special case $\varrho = 1$.
For example, we are now able to cover the Stochastic Partial Differential Equation (SPDE)
\begin{align}
\label{eq:int:SPDE:fhn}\begin{split}
dU &= \big[ U_{xx} + f_{\mathrm{cub}}(U) -W ] dt
+ \sigma \chi(U)  U ( 1 - U) d \beta_t\\
dW &= \big[ \varrho W_{xx} + \e(U-\gamma W)] dt\end{split}
\end{align}
for small $\abs{\sigma}$, in which $(\beta_t)$
is a Brownian motion and $\chi(U)$
is a cut-off function with $\chi(U) = 1$
for $\abs{U} \le 2$. The presence
of this cut-off is required to
enforce the global Lipschitz-smoothness of the noise term.
In this regime, one can think of \sref{eq:int:SPDE:fhn}
as a version of the FitzHugh-Nagumo PDE \sref{eq:int:fzhnag:pde}
where the parameter $a$ is replaced by
$a + \sigma  \dot{\beta}_t$. Notice that
the noise vanishes at the asymptotic state $U = 0$ of the pulse.

\paragraph{Phase tracking}
Although the ability to include noise in models is becoming an essential tool
in many disciplines \cite{Bressloff,bressloff2015nonlinear,
Zhang,NunnoAdvMathFinance2011,Climate},
our understanding of the impact that such distortions have
on basic patterns such as stripes, spots and waves
is still in a preliminary stage
\cite{brassesco1995brownian,funaki1995scaling,Kuske2017,gowda2015early,vinals1991numerical,Lord2012,Shardlow,Bloemker}.
As explained in detail in \cite[{\S}1]{Hamster2017},
several approaches are being developed \cite{Stannat,stannat2014stability,Lang,Inglis}
to analyze stochastically forced waves
that each require a different set of conditions on the noise and structure of the system. The first main 
issue that often limits the application range of the results
is that the underlying linear flow is required to be immediately
contractive, which is (probably) not true for multi-component systems
such as \sref{eq:int:fzhnag:pde}.
The second main issue 
is that an appropriate phase needs to be defined for the wave.
Various ad-hoc choices have been made for this purpose,
which typically rely on geometric intuition of some kind.

Inspired by the agnostic viewpoint described in the expository paper \cite{Zumbrun2009},
we initiated a program in \cite{Hamster2017} that aims to define the phase, shape
and speed of a stochastic wave purely by the technical considerations
that arise when mimicking a deterministic nonlinear stability argument. In particular,
the phase is constantly updated in such a way that the neutral part of the linearized flow
is not felt by the nonlinear terms. 
Together with our novel semigroup approach, this allows us to significantly extend the class of systems
for which it is possible to obtain 
stability results. 

Although defined a priori by technical considerations, we emphasize that in practice our stochastic phase tracking gives us
high quality a posteriori information concerning the position
of the stochastic wave; see Figure \ref{fig:UDifRef} in the next section. 
We remark that the formal approach  recently developed in \cite{cartwright2019} also touches upon several of the ideas
underlying our approach.

\paragraph{Stochastic wave speed}
In principle, once the notion of a stochastic phase
has been defined, one can introduce the stochastic speed 
by taking a derivative in a suitable sense.
In \S\ref{sec:mr:od} and \S\ref{sec:mr:ex} we discuss
two key effects that influence the stochastic speed.
On the one hand, we show how to construct a 
`frozen wave' $\Phi_\sigma$ that feels only
instantaneous stochastic forcing and travels
at an instantaneous speed $c_\sigma$. The interaction
between the nonlinear terms and the stochastic forcing
causes the shape of the wave profile to fluctuate around
$\Phi_\sigma$, which 
introduces an extra net effect in the speed that we refer to as `orbital drift'. 

Our framework provides a mechanism by which both effects can be explicitly described in a perturbative fashion. Indeed, the example in \S\ref{sec:mr:ex} illustrates how this expansion
can be performed up to order $\O(\s^2)$.
Our computations show that the results are in good agreement with numerical simulations of the full SPDE. We emphasize
here that our explicit expressions crucially
and non-trivially
involve the long-term behavior of the underlying semigroup, 
a clear indication of the benefits that can be gained by pursuing a semigroup-based approach.

\paragraph{Obstructions}
Applying the phase tracking procedure sketched above
to the FitzHugh-Nagumo SPDE \sref{eq:int:SPDE:fhn},
one can show that
the
deviation
$(\tilde{U}, \tilde{W})$ from the phase-shifted stochastic wave
satisfies a SPDE of the general form
\begin{equation}
\label{eq:int:fhn:spde:deviation}
\begin{array}{lcl}
d \tilde{U} & = & \Big[
   \big( 1 + \frac{1}{2} \sigma^2 b( \tilde{U}, \tilde{W} )^2  \big) \tilde{U}_{xx}
  + \mathcal{R}_U(\tilde{U}, \tilde{W},  \tilde{U}_x, \tilde{W}_x )  \Big] dt + \mathcal{S}_U(\tilde{U}, \tilde{W},  \tilde{U}_x, \tilde{W}_x ) \, d \beta_t ,
\\[0.2cm]
d \tilde{W} & = & \Big[
   \big( \varrho + \frac{1}{2} \sigma^2 b( \tilde{U}, \tilde{W} )^2  \big) \tilde{W}_{xx}
    + \mathcal{R}_W(\tilde{U}, \tilde{W},  \tilde{U}_x, \tilde{W}_x )  \Big] dt
    + \mathcal{S}_W(\tilde{U}, \tilde{W},  \tilde{U}_x, \tilde{W}_x ) \, d \beta_t
\end{array}
\end{equation}
in which $b$ is a bounded scalar function. For $\sigma \neq 0$
this is a quasi-linear system, but the coefficients in front of the second-order derivatives
are constant with respect to the spatial variable $x$.
These extra second-order terms are a direct consequence of It\^o's formula,
which shows that second derivatives need to be included when applying the chain rule
in a stochastic setting. In particular, deterministic phase-shifts lead to
extra convective terms, while stochastic phase-shifts lead to
extra diffusive terms.

These extra nonlinear diffusive terms cause short-term
regularity issues that
prevent a direct analysis of \sref{eq:int:fhn:spde:deviation}
in a semigroup framework. However, in the special case $\varrho = 1$
they can be transformed away by introducing a new time variable $\tau$
that satisfies
\begin{equation}
\tau'(t) =  1 + \frac{1}{2} \sigma^2 b( \tilde{U}, \tilde{W} )^2 .
\end{equation}
This approach was taken in \cite{Hamster2017},
where we studied reaction-diffusion systems with equal diffusion strengths.

In this paper we concentrate on the case $\varrho \neq 1$ and develop
a more subtle version of this argument. In fact,
we use a similar procedure to
scale out the first of the two nonlinear diffusion terms.
The remaining nonlinear second-order term is only present in the equation
for $\tilde{W}$, which allows us to measure its effect on $\tilde{U}$
via the off-diagonal elements of the associated semigroup. The key point
is that these off-diagonal elements have better regularity properties than their on-diagonal
counterparts, which allows us to side-step the regularity issues outlined above.
Indeed, by commuting $\partial_x$ with the semigroup,
one can obtain an integral expression for $\tilde{U}$
that only involves $(\tilde{U}, \tilde{W}, \partial_x \tilde{U}, \partial_x \tilde{W})$
and that converges in $L^2(\Real)$. A second time-transform
can be used to obtain similar results for $\tilde{W}$.

A second major complication in our stochastic setting is that
$(\partial_x \tilde{U}, \partial_x \tilde{W})$
cannot be directly estimated in $L^2(\Real)$. Indeed,
in order to handle the stochastic integrals
we need tools such as the It\^o Isometry, which
requires square integrability in time.
However, squaring the natural $\O(t^{-1/2})$ short-term behavior
of the semigroup as measured in $\mathcal{L}( L^2; H^1 )$
leads to integrals involving $t^{-1}$ which diverge.

This difficulty was addressed in \cite{Hamster2017} by controlling
temporal integrals of the $H^1$-norm. By performing a delicate
integration-by-parts procedure one can explicitly isolate the
troublesome terms and show that the divergence is in fact
`integrated out'. A similar approach works for our setting here,
but the interaction between the separate time-transforms used for $\tilde{U}$ and $\tilde{W}$
requires a careful analysis with some non-trivial modifications.

\paragraph{Outlook}
Although this paper relaxes the severe equal-diffusion
requirement in \cite{Hamster2017}, we wish to emphasize
that our technical phase-tracking approach is still in
a proof-of-concept state. For example, we rely heavily
on the diffusive smoothening of the deterministic flow
to handle the extra diffusive effects introduced by the stochastic
phase shifts. Taking $\varrho= 0$ removes the former but keeps
the latter, which makes it unclear at present how
to handle such a situation. This is particularly relevant
for many neural field models where the diffusion is modeled
by convolution kernels rather than the standard Laplacian.

It is also unclear at present if our framework can be generalized
to deal with branches of essential spectrum that touch the imaginary axis.
This occurs when analyzing planar
waves in two or more dimensions \cite{BHM, KAP1997, HJHSTB2D, HJHOBST2D}
or when studying viscous shocks in the context of conservation laws
\cite{beck2010nonlinear,MasciaZumbrun02,HJHNLS}. In the deterministic
case these settings require the use of pointwise estimates
on Green's functions, which give more refined control on the linear
flow than standard semigroup bounds.

We are more confident about the possibility of including more general types of noise
in our framework.  For instance, we believe that there is no fundamental obstruction
including noise that is colored in space\footnote{During the review 
process, results in this direction were published in \cite{Hamster2020}.}, which arises frequently in many applications
\cite{Lang,Garcia2001}. In addition, it should also be possible
to remove our dependence on the variational framework
developed by Liu and R\"ockner \cite{LiuRockner}. Indeed, our estimates on the mild solutions
appear to be strong enough to allow short-term existence results to be obtained for the original
SPDE in the vicinity of the wave.

\paragraph{Organization}
This paper is reasonably self-contained and the main narrative
can be read independently of \cite{Hamster2017}. However, we do borrow
some results from \cite{Hamster2017} that do not depend on the structure of the diffusion matrix.
This allows us to focus our attention on the parts that are essentially different.

We formulate our phase-tracking mechanism and state our
main results in \S\ref{sec:mr}. In addition, we
illustrate these results in the same section 
by numerically analyzing an example system of FitzHugh-Nagumo type.
In \S\ref{sec:SplitSem} we decompose the semigroup associated to the linearization
of the deterministic wave into its diagonal and off-diagonal parts.
We focus specifically on the short-time behavior of the off-diagonal elements
and show that the commutator of $\partial_x$ and the semigroup extends to a bounded
operator on $L^2$.
In \S\ref{sec:STT}  we describe the stochastic phase-shifts and time-shifts
that are required to eliminate the problematic terms from our equations.
We apply the results from \S\ref{sec:SplitSem} to recast the resulting SPDE into a mild
formulation and establish bounds for the final nonlinearities.
This allows us to close a nonlinear stability argument in \S\ref{sec:nls}
by carefully estimating each of the mild integrals.

\paragraph{Acknowledgements.}
Hupkes acknowledges support from the Netherlands Organization
for Scientific Research (NWO) (grant 639.032.612).

\section{Main results}
\label{sec:mr}

In this paper we are interested in the stability of traveling
wave solutions to SPDEs of the form
\begin{equation}
\label{eq:mr:main:spde}
dU =
  \big[ \rho \partial_{xx} U + f(U) \big] dt
  + \sigma g(U) d \beta_t.
\end{equation}
Here we take $U = U(x,t) \in \Real^n$ with $x \in \Real$ and $t \ge 0$.
We start in \S\ref{sec:mr:fs} by formulating
precise conditions on the system above
and stating our main theorem.
In \S\ref{sec:mr:od}, we subsequently discuss how our formalism gives us explicit expressions for the stochastic corrections
to the deterministic wave speed. We actually compute these corrections up to $\O(\sigma^2)$ for the FitzHugh-Nagumo equation in \S\ref{sec:mr:ex}
and show that the results are in good agreement with numerical simulations of the full SPDE.  

\subsection{Formal setup}
\label{sec:mr:fs}
We start by formulating two structural conditions on the deterministic
and stochastic part of \sref{eq:mr:main:spde}.
Together these imply that our system
has a variational structure with a nonlinearity $f$ that grows at most cubically.
In particular, it is covered by the variational framework developed in \cite{LiuRockner} with $\alpha = 2$.
The crucial difference between assumption (HDt) below and assumption (HA) in  \cite{Hamster2017}
is that the diagonal elements of $\rho$ no longer have to be equal.

\begin{itemize}
\item[(HDt)]{
  The matrix $\rho \in \Real^{n\times n}$ is a diagonal matrix with
  strictly positive diagonal elements $\{\rho_i\}_{i=1}^n$.
  In addition, we have $f \in C^3(\Real^n; \Real^n)$
 and there exist $u_\pm \in \Real^n$
 for which $f(u_-) = f(u_+) = 0 $.
 Finally, 
 $D^3 f$ is bounded
 and there exists
 a constant $K_{\mathrm{var}} > 0$ so that the one-sided inequality
 \begin{equation}
    \langle f( u_A ) - f(u_B) , u_A - u_B \rangle_{\Real^n} \le K_{\mathrm{var}} \abs{ u_A - u_B}^2
 \end{equation}
 holds for all pairs $(u_A,u_B) \in \Real^n \times \Real^n$.
}
\item[(HSt)]{
  The function $g \in C^2(\Real^n; \Real^n)$
  is globally Lipschitz with $g(u_-) = g(u_+) = 0$.
  In addition, $Dg$ is bounded and globally Lipschitz.
  Finally, the process $(\beta_t)_{t \ge 0}$ is a Brownian motion
  with respect to the complete filtered probability space
  \begin{equation}
    \Big(\Omega, \mathcal{F}, ( \mathcal{F}_t)_{t \ge 0} , \mathbb{P} \Big).
  \end{equation}
}
\end{itemize}

We write $\rho_{\mathrm{min}}=\min\{\rho_i\}>0$,
together with $\rho_{\mathrm{max}} = \max\{\rho_i\}$.
In addition, we introduce the shorthands
\begin{equation}
L^2 = L^2(\Real;\Real^n),
\qquad
H^1 = H^1(\Real ;\Real^n),
\qquad
H^2 = H^2(\Real; \Real^n).
\end{equation}
Our final assumption states that the deterministic part
of \sref{eq:mr:main:spde}
has a spectrally stable traveling wave solution
that connects the two equilibria $u_\pm$
(which are allowed to be equal). This traveling
wave should approach these equilibria at an exponential rate.

\begin{itemize}
\item[(HTw)]{
  There exists a wavespeed $c_0 \in \Real$ and a waveprofile
  $\Phi_0 \in C^2(\Real ; \Real^n)$ that
  satisfies the traveling wave ODE
  \begin{equation}\label{eq:MR:TWODE}
    \rho \Phi_0''+c_0 \Phi_0' + f(\Phi_0)=0
  \end{equation}
  and approaches its limiting values
  $\Phi_0(\pm \infty) = u_\pm$ at an exponential rate.
  In addition, the associated linear operator
  $\mathcal{L}_{\mathrm{tw}}: H^2 \to L^2$ that acts as
  \begin{equation}
   \label{eq:mr:def:l:tw}
   [\mathcal{L}_{\mathrm{tw}} v](\xi) =
      \rho v''(\xi)+c_0 v'(\xi)  + Df\big(\Phi_0(\xi) \big) v(\xi)
  \end{equation}
  has a simple eigenvalue at $\lambda = 0$
  and has no other spectrum
  in the half-plane $\{\Re \lambda \ge -2\beta\} \subset \mathbb{C}$
  for some $ \beta > 0$.
}
\end{itemize}

The formal adjoint
\begin{equation}
\mathcal{L}_{\mathrm{tw}}^*: H^2 \to L^2
\end{equation}
of the operator \sref{eq:mr:def:l:tw}  acts as
\begin{equation}
[\mathcal{L}_{\mathrm{tw}}^* w](\xi) =
\rho w''(\xi) -c_0 w'(\xi) +  (Df\big(\Phi_0(\xi) \big))^* w(\xi) .
\end{equation}
Indeed, one easily verifies that
\begin{equation}
\langle \mathcal{L}_{\mathrm{tw}} v , w \rangle_{L^2}
= \langle v, \mathcal{L}_{\mathrm{tw}}^* w \rangle_{L^2}
\end{equation}
whenever $(v,w) \in H^2 \times H^2$. Here $\langle \cdot, \cdot \rangle_{L^2}$
denotes the standard inner-product on $L^2$.
The assumption that zero is a simple eigenvalue for $\mathcal{L}_{\mathrm{tw}}$
implies that $\mathcal{L}_{\mathrm{tw}}^* \psi_{\mathrm{tw}} = 0$
for some $\psi_{\mathrm{tw}} \in H^2$ that we normalize to get
  \begin{equation}
    \label{eq:mr:hs:norm:cnd:psitw}
     \langle \Phi_0' , \psi_{\mathrm{tw}} \rangle_{L^2} = 1.
  \end{equation}

We remark here that it is advantageous to view
SPDEs as evolutions on Hilbert spaces, since
powerful tools are available in this setting.
However, in the case where $u_- \neq u_+$,
the waveprofile $\Phi_0$ does not lie in the natural
statespace $L^2$. In order to circumvent this problem,
we use $\Phi_0$ as a reference function
that connects $u_-$ to $u_+$, allowing
us to measure deviations from this
function in the Hilbert spaces $H^1$ and $L^2$. In
 order to highlight this dual role and prevent
any confusion,
we introduce the duplicate notation
\begin{equation}
\Phi_{\mathrm{ref}} = \Phi_0.
\end{equation}
This allows us to introduce the sets
\begin{equation}
\mathcal{U}_{L^2} = \Phi_{\mathrm{ref}} + L^2,
\qquad
\mathcal{U}_{H^1} = \Phi_{\mathrm{ref}} + H^1,
\qquad
\mathcal{U}_{H^2} = \Phi_{\mathrm{ref}} + H^2,
\end{equation}
which we will use as the relevant state-spaces
to capture the solutions $U$ to \sref{eq:mr:main:spde}.

We now set out
to couple an extra phase-tracking\footnote{See \cite[\S 2.4]{Hamster2017} for a more intuitive explanation of this phase.} SDE
to our SPDE \sref{eq:mr:main:spde}.
As a preparation, we
pick a sufficiently large constant $K_{\mathrm{high}} > 0$
together with
two $C^\infty$-smooth non-decreasing cut-off functions
\begin{equation}
\chi_{\mathrm{low}}: \Real \to [\frac{1}{4}, \infty),
\qquad
\chi_{\mathrm{high}}: \Real \to
[- K_{\mathrm{high}} - 1, K_{\mathrm{high}} + 1]
\end{equation}
that satisfy the identities
\begin{equation}
\chi_{\mathrm{low}}(\vartheta) = \frac{1}{4} \hbox{ for } \vartheta \le \frac{1}{4},
\qquad
\chi_{\mathrm{low}}(\vartheta) = \vartheta \hbox{ for } \vartheta \ge \frac{1}{2},
\end{equation}
together with
\begin{equation}
\chi_{\mathrm{high}}(\vartheta) = \vartheta \hbox{ for } \abs{\vartheta} \le K_{\mathrm{high}},
\qquad
\chi_{\mathrm{high}}(\vartheta) = \mathrm{sign}(\vartheta) \big[K_{\mathrm{high}} + 1]
   \hbox{ for } \abs{\vartheta} \ge K_{\mathrm{high}} + 1.
\end{equation}

For any $u \in \mathcal{U}_{H^1}$
and $\psi \in H^1$,
this allows us to introduce
the function
\begin{equation}
\label{eq:mr:def:b}
\begin{array}{lcl}
b(u, \psi)
 & = &
  - \Big[
    \chi_{\mathrm{low}}\big(
    \langle \partial_\xi u ,
   \psi \rangle_{L^2} \big) \Big]^{-1}
     \chi_{\mathrm{high}}\big(
        \langle g(u) , \psi \rangle_{L^2}
     \big) ,
\\[0.2cm]
\end{array}
\end{equation}
together with the diagonal $n\times n$-matrix
\begin{equation}
\label{eq:mr:def:kappa}
\begin{array}{lcl}
\kappa_{\sigma}(u, \psi)
& = & \mathrm{diag} \{ \kappa_{\sigma;i}(u, \psi) \}_{i=1}^n
  := \mathrm{diag} \{ 1 + \frac{1}{2\rho_i} \sigma^2 b( u, \psi )^2 \}_{i=1}^n.
\end{array}
\end{equation}
In addition, for any
$u \in \mathcal{U}_{H^1}$, $c \in \Real$ and $\psi \in H^2$
we write
\begin{equation}
\begin{array}{lcl}
\label{eq:mr:def:a}
a_{\sigma}(u,  c, \psi )
 & = & -
    \Big[ \chi_{\mathrm{low}}\big( \langle \partial_\xi u, \psi \rangle_{L^2} \big) \Big]^{-1}
    \langle \kappa_{\sigma}(u,\psi)u , \rho \partial_{\xi \xi} \psi \rangle_{L^2}
\\[0.2cm]
& & \qquad
    -\Big[ \chi_{\mathrm{low}}\big( \langle \partial_\xi u, \psi \rangle_{L^2} \big) \Big]^{-1}
    \langle
        f(u )
    +c \partial_\xi u
    + \sigma^2 b( u, \psi)
      \partial_\xi [g(u )]
       ,
     \psi \rangle_{L^2} .
\end{array}
\end{equation}
The essential difference with the definitions of $\kappa_\sigma$
and $a_{\sigma}$ in \cite{Hamster2017} is that
$\kappa_{\sigma}$ is now a matrix instead of a constant.
However, this does not affect the ideas and results in {\S}3,4 and {\S}7 of \cite{Hamster2017},
which can be transferred to the current setting almost verbatim. Indeed,
one simply replaces $\rho$ by $\rho_\mathrm{min}$ or
$\rho_\mathrm{max}$ as necessary.

The traveling wave ODE \sref{eq:MR:TWODE} implies that
$a_0(\Phi_0, c_0, \psi_{\mathrm{tw}}) = 0$.
Following \cite[Prop. 2.2]{Hamster2017},
one can show that
there exists a branch
of profiles and speeds
$(\Phi_\s,c_\s)$ in $\mathcal{U}_{H^2} \times \R$ that is $\mathcal{O}(\s^2)$ close to $(\Phi_0,c_0)$,
for which
\begin{equation}
a_{\sigma}(\Phi_{\sigma}, c_{\sigma}, \psi_{\mathrm{tw}}) = 0.
\end{equation}
Upon introducing the
right-shift operators
\begin{equation}
[T_{\gamma} u](\xi)
= u(\xi - \gamma)
\end{equation}
we can now formally introduce the coupled
SPDE
\begin{equation}
\label{eq:mr:formal:spde}
\begin{array}{lcl}
dU &= &
  \big[ \rho\p_{xx} U + f(U) \big] dt
  + \sigma g(U) d \beta_t,
\\[0.2cm]
d\Gamma & = &
  \big[ c_{\sigma} + a_\sigma \big( U ,c_{\sigma}, T_{\Gamma} \psi_{\mathrm{tw}} \big)
  \big] dt
  + \sigma b\big(U , T_{\Gamma} \psi_{\mathrm{tw}} \big)
     \, d \beta_t ,
\end{array}
\end{equation}
which is the main focus in this paper.
Following the procedure used to establish
\cite[Prop. 2.1]{Hamster2017},
one can show
that this SPDE coupled with an initial condition
\begin{equation}
\label{eq:mr:init:cond}
(U, \Gamma)(0) = (u_0 , \gamma_0)  \in \mathcal{U}_{H^1} \times \Real
\end{equation}
has solutions\footnote{We refer to \cite[Prop. 2.1]{Hamster2017} for
the precise notion of a solution.}
$\big(U(t), \Gamma(t) \big) \in \mathcal{U}_{H^1} \times \Real$
that can be defined for all $t \ge 0$ and are almost-surely continuous
as maps into $\mathcal{U}_{L^2} \times \Real$.

For any initial condition $u_0 \in \mathcal{U}_{H^1}$
that is sufficiently close to $\Phi_{\sigma}$,
\cite[Prop. 2.3]{Hamster2017}
shows that it is possible to pick $\gamma_0$ in such a way that
\begin{equation}
\langle T_{-\g_{0}}u(0)-\Phi_\s, \psi_{\mathrm{tw} } \rangle_{L^2} = 0.
\end{equation}
This allows us to define the process
\begin{equation}
\label{eq:MR:DefV}
V_{u_0}(t) = T_{-\Gamma(t) }
  \big[ U(t)\big] - \Phi_{\sigma} ,
\end{equation}
which can be thought of as the
deviation of the solution
$U(t)$ of \sref{eq:mr:formal:spde}-\sref{eq:mr:init:cond}
from the stochastic wave
$\Phi_{\sigma}$ shifted to the position $\Gamma(t)$.

In order to measure the size
of this deviation
we pick $\e > 0$
and introduce the scalar function
\begin{equation}\label{eq:DefNeps}
N_{\e;u_0} (t) =
\norm{V_{u_0}(t)}_{L^2}^2
 + \int_0^t e^{- \e (t - s) }
    \norm{  V_{u_0}(s)}_{H^1}^2 \, ds .
\end{equation}
For each $T > 0$ and $\eta > 0$ we now define
the probability
\begin{equation}
p_{\e}(T, \eta, u_0) = P\Big(
 \sup_{0 \le t \le T}
 N_{\e;u_0}(t) > \eta
\Big) .
\end{equation}
Our  main result shows that
the probability that
$N_{\e;u_0}$ remains
small on timescales of order $\s^{-2}$ can be pushed arbitrarily close to one
by restricting the strength of the noise
and the size of the initial perturbation.
This extends \cite[Thm. 2.4]{Hamster2017}
to the current setting where the diffusion matrix $\rho$ need not be proportional to the identity.

\begin{thm}[{see \S\ref{sec:nls}}]
\label{thm:mr:orbital:stb}
Suppose that (HDt), (HSt) and
(HTw) are all satisfied
and pick sufficiently small constants $\e > 0$,
$\delta_0>0$, $\delta_{\eta} > 0$ and $\delta_{\sigma} > 0$.
Then there exists a constant $K > 0$
so that for every $0 \le \sigma \le \delta_{\sigma}T^{-1/2}$,
any $u_0 \in \mathcal{U}_{H^1}$
that satisfies $\nrm{u_0-\Phi_\s}_{L^2}<\delta_0$,
any $0 < \eta \le \delta_{\eta}$ and any $T > 0$,
we have the inequality
\begin{equation}
\label{eq:mr:bnd:for:p:eps}
p_{\e}(T, \eta, u_0)
  \le   \eta^{-1} K \Big[ \norm{u_0 - \Phi_{\sigma}}_{H^1}^2 + \sigma^2T \Big].
\end{equation}
\end{thm}

\subsection{Orbital drift}
\label{sec:mr:od}
On account of the theory developed in \cite[{\S}12]{leadbetter2012extremes}
to describe the suprema of finite-dimensional Gaussian processes,
we suspect that the $\sigma^2 T$ term appearing
in the bound \sref{eq:mr:bnd:for:p:eps} can
be replaced by $\sigma^2 \ln T$. This would allow us to consider
timescales of order  $\mathrm{exp}[\delta_\sigma /\sigma^2 ]$,
which are exponential in the noise-strength instead of merely
polynomial. The key limitation is that the theory of stochastic convolutions in Hilbert spaces is still in the early stages of development.

In order to track the evolution of the phase
over such long timescales,
we follow \cite{Hamster2017} and introduce the formal Ansatz
\begin{equation}
\Gamma(t) = c_{\sigma} t + \sigma \Gamma_{\sigma;1}(t)
+ \sigma^2 \Gamma_{\sigma;2}(t) + O (\s^3) .
\end{equation}
The first-order term is the scaled Brownian motion
\begin{equation}
\label{eq:gamma:sigma:1}
\Gamma_{\sigma;1}(t) = b(\Phi_{\sigma}, \psi_{\mathrm{tw}}) \beta_t,
\end{equation}
which naturally has zero mean and hence does not
contribute to any deviation of the average observed wavespeed.

In order to understand the second-order term,
we introduce the orbital drift coefficient
\begin{equation}
\label{eq:mr:def:c:orb:drift}
\begin{array}{lcl}
c^{\mathrm{od}}_{\sigma;2}
 &  = &
\frac{1}{2} \int_0^\infty D_1^2 a_{\sigma}\big(\Phi_{\sigma} ,
  c_{\sigma} , \psi_{\mathrm{tw}} \big)
     \Big[S(s) \big( g(\Phi_{\sigma} ) + b(\Phi_{\sigma}, \psi_{\mathrm{tw}}) \Phi'_{\sigma}   \big),
     S(s) \big( g(\Phi_{\sigma} ) + b(\Phi_{\sigma}, \psi_{\mathrm{tw}}) \Phi'_{\sigma}   \big)
     \Big] \, d s ,
\end{array}
\end{equation}
in which $\{S(s)\}_{s \ge 0}$ denotes the semigroup 
generated by $\mathcal{L}_{\mathrm{tw}}$.
In \cite[{\S}2.4]{Hamster2017} we gave
an explicit expression for $\Gamma_{\sigma;2}$
and showed that
\begin{equation}
\lim_{t \to \infty} t^{-1} E \Gamma_{\sigma;2}(t)
 = c^{\mathrm{od}}_{\sigma;2} .
\end{equation}
Note that we are keeping the $\sigma$-dependence
in these definitions for notational convenience,
but in {\S}\ref{sec:mr:ex} we show how the leading
order contribution can be determined.

The discussion above suggests that it is natural
to introduce the expression
\begin{equation}
\label{eq:mr:def:c:sigma:lim}
c_{\sigma;\lim}^{(2)} =
  c_{\sigma} + \sigma^2 c^{\mathrm{od}}_{\sigma;2},
\end{equation}
which satisfies $c_{\sigma;\lim}^{(2)} -c_0 = O(\sigma^2)$.
Our conjecture is that the expected value of the wavespeed
for large times
behaves as $c_{\sigma;\lim}^{(2)} + O(\s^3)$.
In order to interpret this,
we note that the profile $\Phi_\s$ travels at an instantaneous
velocity $c_{\sigma}$, but also experiences stochastic forcing.
As a consequence of this forcing, which is mean reverting
toward $\Phi_\s$, the profile 
fluctuates in the orbital vicinity of $\Phi_\sigma$.
At leading order, the underlying mechanism behind this behavior resembles
an Ornstein-Uhlenbeck process, which 
means that the amplitude
of these fluctuations can be expected to stabilize for large times.
This leads to an extra contribution to the observed
wavespeed, which we refer to as an orbital drift.
The second term in \sref{eq:mr:def:c:sigma:lim}
describes the leading order contribution to this orbital drift.

\subsection{Example}
\label{sec:mr:ex}

In order to illustrate our results, let us consider the
FitzHugh-Nagumo system
\begin{align}
\label{eq:example:SPDE:fhn}\begin{split}
dU &= \big[ U_{xx} + f_{\mathrm{cub}}(U) -W ] dt
+ \sigma  g^{(u)}( U)  d \beta_t,
\\[0.2cm]
dW &= \big[ \varrho V_{xx} + \e(U-\gamma W)] dt
\end{split}
\end{align}
in a parameter regime where (HDt), (HSt) and (HTw) all hold. 
We write $\Phi_0=(\Phi^{(u)}_0,\Phi^{(w)}_0)$ 
for the deterministic wave defined in (HTw)
and recall the associated linear operator
$\mathcal{L}_{\mathrm{tw}}: H^2(\Real ;\Real^2) \to L^2(\Real; \Real^2)$
that acts as
\begin{equation}
\begin{array}{lcl}
\mathcal{L}_{\mathrm{tw}}
 & = & \left(
     \begin{array}{cc}
         \partial_{\xi \xi} + c_0 \partial_\xi + f'_{\mathrm{cub}}(\Phi_0^{(u)}) & - 1
         \\[0.2cm]
         \e  & \varrho \partial_{\xi \xi} + c_0 \partial_\xi   - \e \gamma
         \\[0.2cm]
     \end{array}
 \right) .
\end{array}
\end{equation}
The adjoint operator acts as
\begin{equation}
\begin{array}{lcl}
\mathcal{L}^*_{\mathrm{tw}}
 & = & \left(
     \begin{array}{cc}
         \partial_{\xi \xi} - c_0 \partial_\xi + f'_{\mathrm{cub}}(\Phi_0^{(u)}) & \e
         \\[0.2cm]
         -1  & \varrho \partial_{\xi \xi} - c_0 \partial_\xi   - \e \gamma
         \\[0.2cm]
     \end{array}
 \right)
\end{array}
\end{equation}
and admits the eigenfunction
$\psi_{\mathrm{tw}}=(\psi^{(u)}_{\mathrm{tw}},\psi^{(w)}_{\mathrm{tw}})$
that can be normalized in such a way that
\begin{equation}
\label{eq:mr:ex:norm:adj}
\langle \partial_\xi \Phi_0 , \psi_{\mathrm{tw}} \rangle_{L^2(\Real; \Real^2)} = 1.
\end{equation}
To summarize, we have
\begin{equation}
\mathcal{L}_{\mathrm{tw}} \partial_\xi ( \Phi^{(u)}_0, \Phi^{(w)}_0 )^T = 0,
\qquad \qquad
\mathcal{L}^*_{\mathrm{tw}}  ( \psi^{(u)}_{\mathrm{tw}}, \psi^{(w)}_{\mathrm{tw}})^T = 0.
\end{equation}

\begin{figure}
    \centering
        \begin{subfigure}{0.5\textwidth}
 		\def\svgwidth{\columnwidth}
    		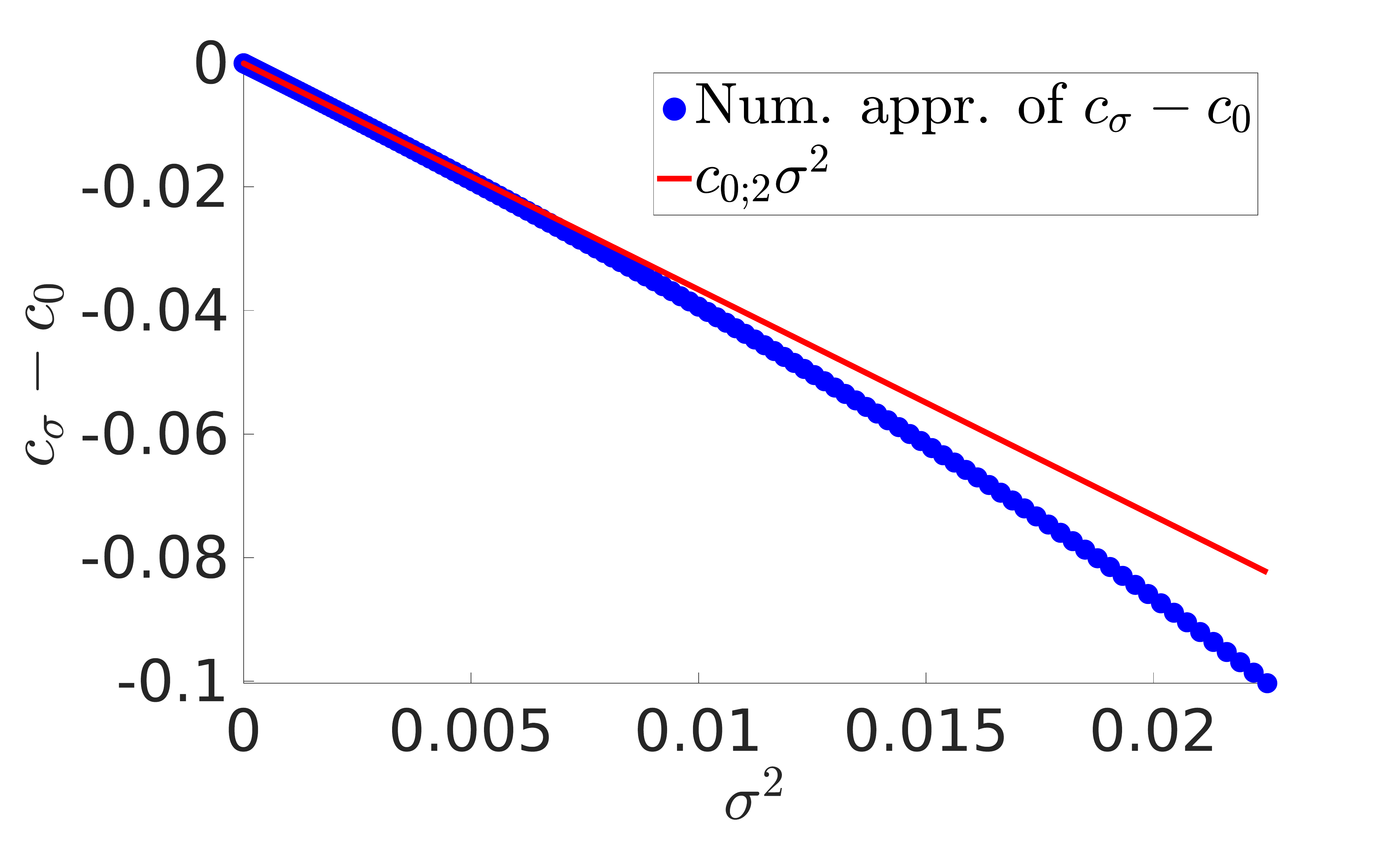
        \caption{}
    \end{subfigure}%
    \begin{subfigure}{0.5\textwidth}
 		\def\svgwidth{\columnwidth}
    		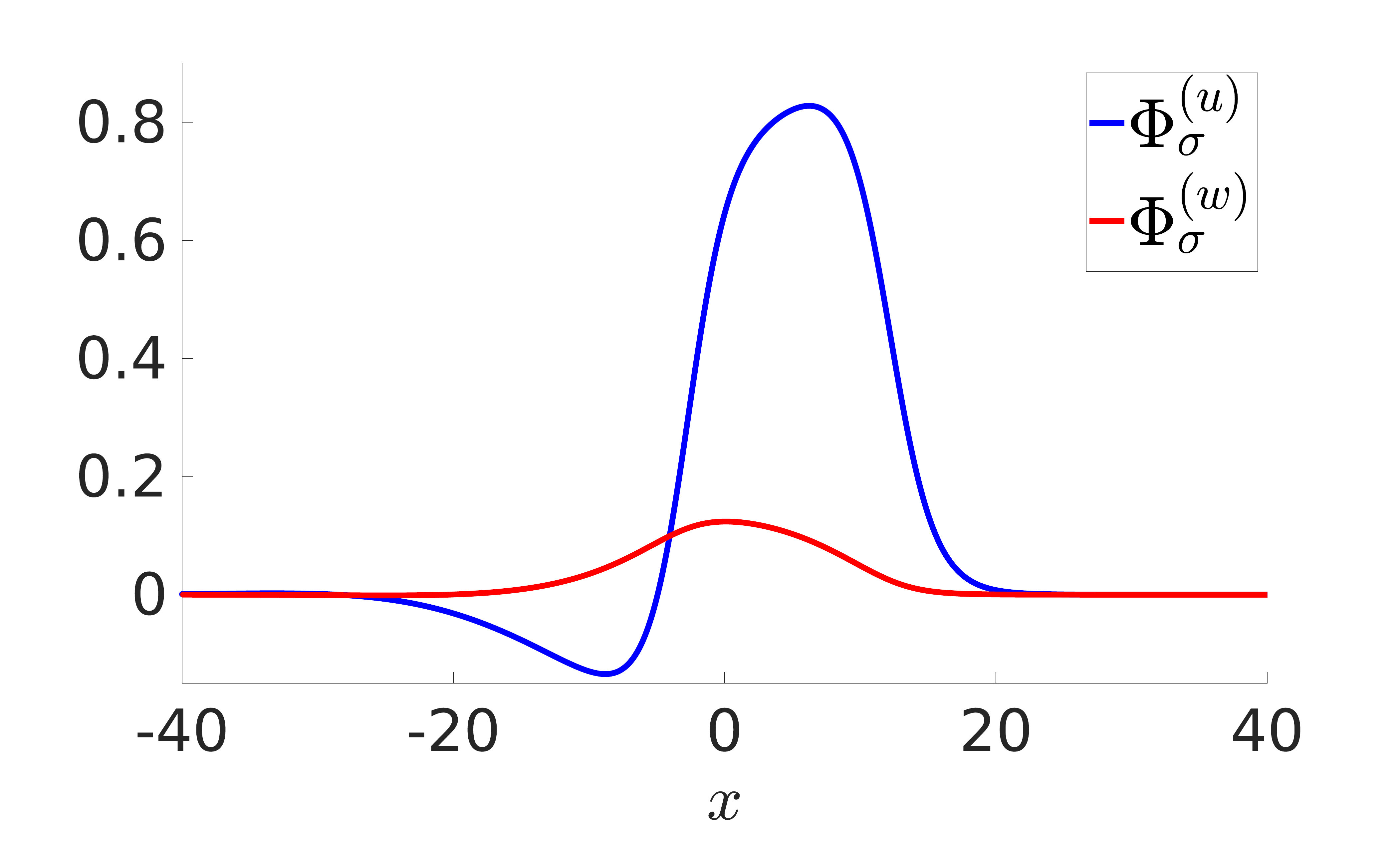
        \caption{}
    \end{subfigure}
    \caption{Numerical results for the solution $(\Phi_\s,c_\s)$ to equation \sref{eq:example:TWode}. Figure (a) shows the numerical approximation of $c_\s-c_0$ and the first order approximation of this difference. We chose $g^{(u)}(u)=u$ with parameters $a=0.1$, $\varrho=0.01$, $\e=0.01$, $\gamma=5$.
    Using \sref{eq:mr:melnikov:cop:c:zero:2}
    we numerically computed $c_{0;2}=-3.66$. Figure (b) shows the two components of $\Phi_\s$ for $\s=0.15$ for the same parameter values. On the scale of this figure they are almost identical to $\Phi_0$. }
    \label{fig:cs}
\end{figure}

Upon writing $\Phi_\s=(\Phi^{(u)}_\s,\Phi^{(w)}_\s)$,
the stochastic wave equation $a_{\sigma}(\Phi_\s, c_\s, \psi_{\mathrm{tw}} ) = 0$
can be written as
\begin{equation}
\begin{array}{lcl}
\label{eq:example:TWode}
 - c_\s \p_{x}\Phi^{(u)}_\s & = &
   \left(1+\frac{\s^2}{2}\tilde{b}(\Phi_\s)^2\right)\p_{xx}\Phi^{(u)}_\s  + f_{\mathrm{cub}}(\Phi^{(u)}_\s) -\Phi^{(w)}_\s
     +\s^2\tilde{b}(\Phi_\s)\p_x[ g^{(u)}(\Phi^{(u)}_\s)] ,
     \\[0.2cm]
 - c_\s \p_{x}\Phi^{(w)}_\s & = &
  \left(\varrho +\frac{\s^2}{2}\tilde{b}(\Phi_\s)^2\right)\p_{xx}\Phi^{(w)}_\s
      + \e(\Phi^{(u)}_\s-\gamma \Phi^{(w)}_\s) ,
\end{array}
\end{equation}
where $\tilde{b}$ is given by
\begin{align}
    \tilde{b}(\Phi_\s)=-\frac{\ip{  g^{(u)} (\Phi^{(u)}_\s ),\psi^{(u)}_{\mathrm{tw}}}_{L^2(\Real;\Real)}}
                     {\ip{\p_x\Phi_\s,\psi_{\mathrm{tw}}}_{L^2(\Real;\Real^2)}}.
\end{align}
We now introduce the expansions
\begin{equation}
\Phi_\s=\Phi_0+\s^2\Phi_{0;2}+\O(\s^4),
\qquad \qquad
c_\s = c_0 + \s^2 c_{0;2} + \O(\s^4)
\end{equation}
with $\Phi_{0;2} = \big( \Phi^{(u)}_{0; 2} , \Phi^{(w)}_{0; 2} \big)$.
Substituting these expressions into \sref{eq:example:TWode}
and balancing the second order terms, we find
\begin{equation}
\begin{array}{lcl}
 - c_{0;2} \p_{x}\Phi^{(u)}_0 - c_{0} \p_{x}\Phi^{(u)}_{0;2}& = &
   \p_{xx}\Phi^{(u)}_{0;2} + \frac{1}{2} \tilde{b}(\Phi_0)^2 \p_{xx} \Phi^{(u)}_0
     + f'_{\mathrm{cub}}(\Phi^{(u)}_0) \Phi^{(u)}_{0;2} -\Phi^{(w)}_{0;2}
    \\[0.2cm]
    & & \qquad
     +\tilde{b}(\Phi_0)\p_x g^{(u)}(\Phi^{(u)}_0 ) ,
     \\[0.2cm]
 - c_{0;2} \p_{x}\Phi^{(w)}_0  - c_{0} \p_{x}\Phi^{(w)}_{0;2}& = &
   \varrho \p_{xx} \Phi^{(w)}_{0;2}
   + \frac{1}{2} \tilde{b}(\Phi_0)^2 \p_{xx} \Phi^{(w)}_{0}
     + \e(\Phi^{(u)}_{0;2}-\gamma \Phi^{(w)}_{0;2}) ,
\end{array}
\end{equation}
which can be rephrased as
\begin{equation}
\begin{array}{lcl}
\mathcal{L}_{\mathrm{tw}} \Phi_{0;2}
& = &
  - c_{0;2} \p_{x} \Phi_0  - \frac{1}{2} \tilde{b}(\Phi_0)^2 \p_{\xi \xi} \Phi_0
   - \tilde{b} (\Phi_0) \big( \p_x g^{(u)}( \Phi^{(u)}_0 ) , 0 \big)^T .
\end{array}
\end{equation}
Using the normalization \sref{eq:mr:ex:norm:adj}
together with the fact that $\langle \psi_{\mathrm{tw}}, \mathcal{L}_{\mathrm{tw}} \Phi_{0;2} \rangle_{L^2(\Real;\Real^2)} = 0$,
we find the explicit expression
\begin{equation}
\label{eq:mr:melnikov:cop:c:zero:2}
c_{0;2} = - \frac{1}{2} \tilde{b}(\Phi_0)^2 \langle \p_{\xi \xi} \Phi_0 , \psi_{\mathrm{tw}} \rangle_{L^2(\Real ; \Real^2) }
 - \tilde{b}(\Phi_0) \langle \p_x g^{(u)} ( \Phi^{(u)}_0 ) , \psi^{(u)}_{\mathrm{tw}} \rangle_{L^2(\Real ;\Real)}
\end{equation}
for the coefficient that governs the leading order behavior of 
$c_{\s} - c_0$.
In Figure \ref{fig:cs} we show numerically
that $c_{0;2} \s^2$ indeed corresponds well with $c_{\s} - c_0$
for small values of $\sigma^2$.

In Figure \ref{fig:UDifRef} we illustrate the behavior 
of a representative sample solution to \sref{eq:example:SPDE:fhn}
by plotting it in three different moving frames.
Figure \ref{fig:UDifREfa} clearly shows that the deterministic
speed $c_0$ overestimates the actual speed as the wave 
moves to the left. The situation is improved
in Figure \ref{fig:UDifREfb}, where we use a frame
that travels with the stochastic speed $c_{\sigma}$.
However, the position of the wave now fluctuates
around a position that still moves slowly to the left
as a consequence of the orbital drift. This is remedied
in Figure \ref{fig:UDifREfc} where we use the full stochastic
phase $\Gamma(t)$. Indeed, the wave now appears to be at
a fixed position, but naturally still experiences fluctuations in its shape. This shows that $\Gamma(t)$ is indeed a powerful tool
to characterize the position of the wave.

\begin{figure}
\centering
\begin{subfigure}{.33\textwidth}
  \centering
 		\def\svgwidth{\columnwidth}
    		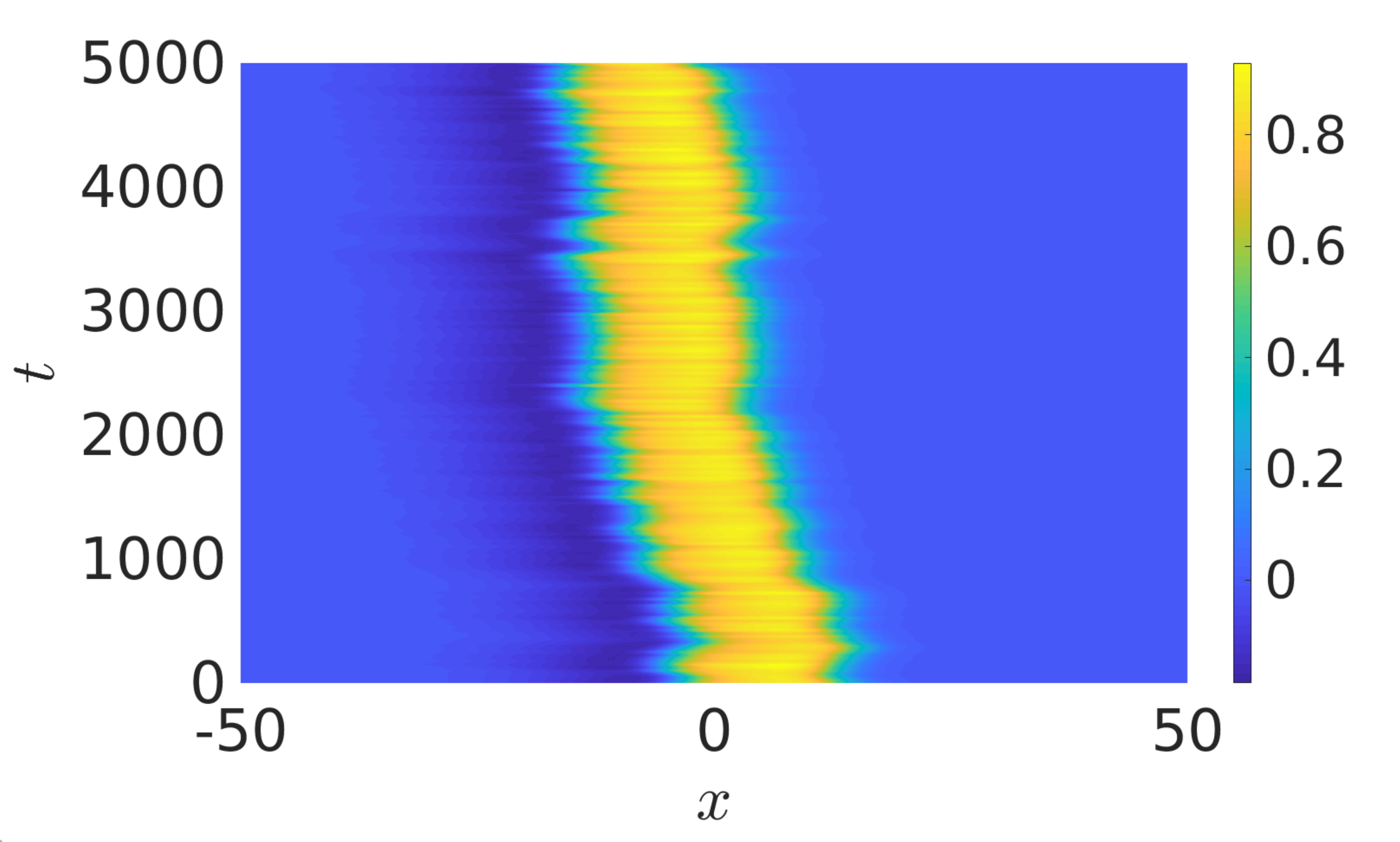
  \caption{$U(\cdot +c_0t,t)$}
    \label{fig:UDifREfa}
\end{subfigure}%
\begin{subfigure}{.33\textwidth}
  \centering
 		\def\svgwidth{\columnwidth}
    		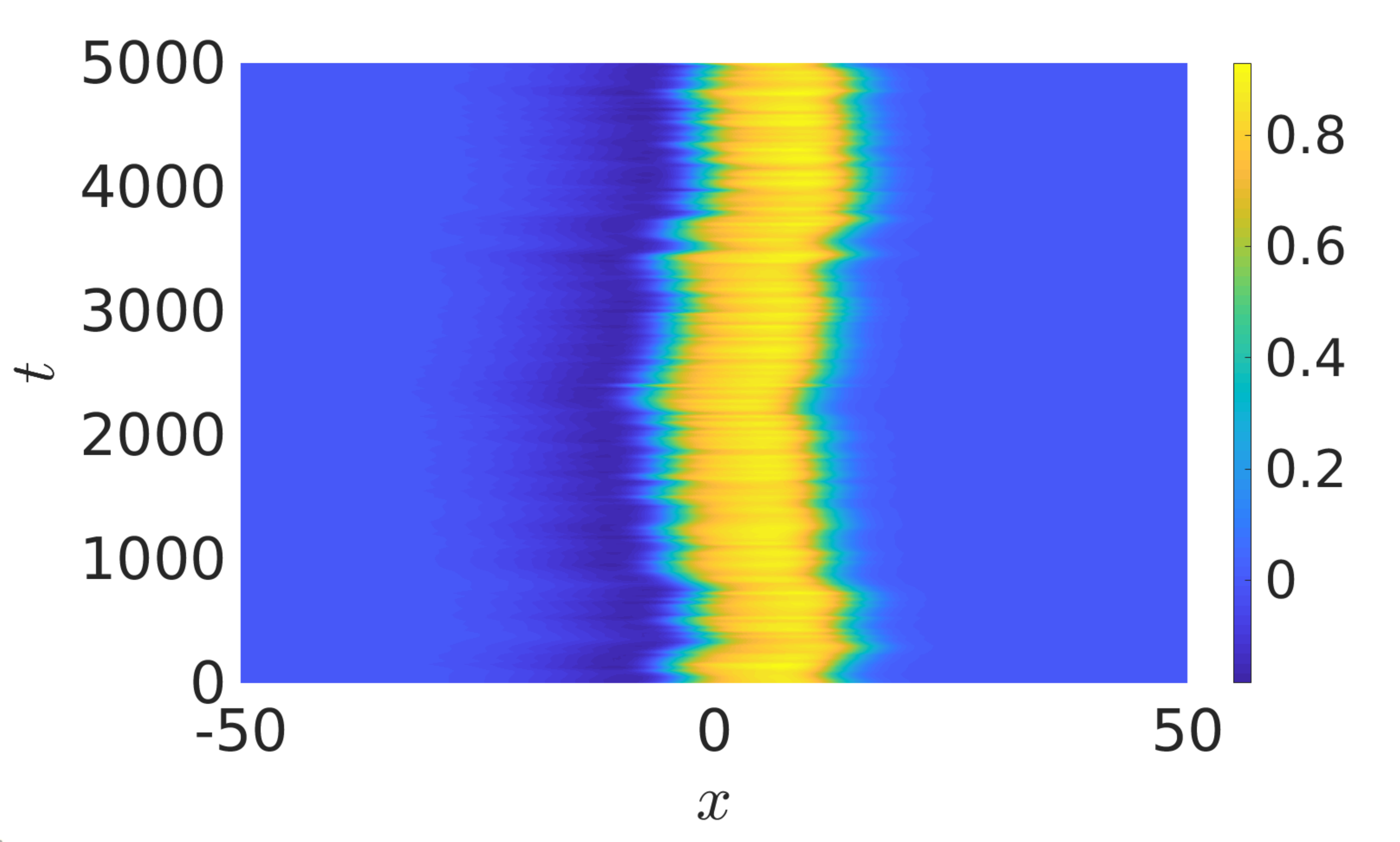
  \caption{$U(\cdot +c_\sigma t,t)$}
      \label{fig:UDifREfb}
\end{subfigure}%
\begin{subfigure}{.33\textwidth}
  \centering
 		\def\svgwidth{\columnwidth}
    		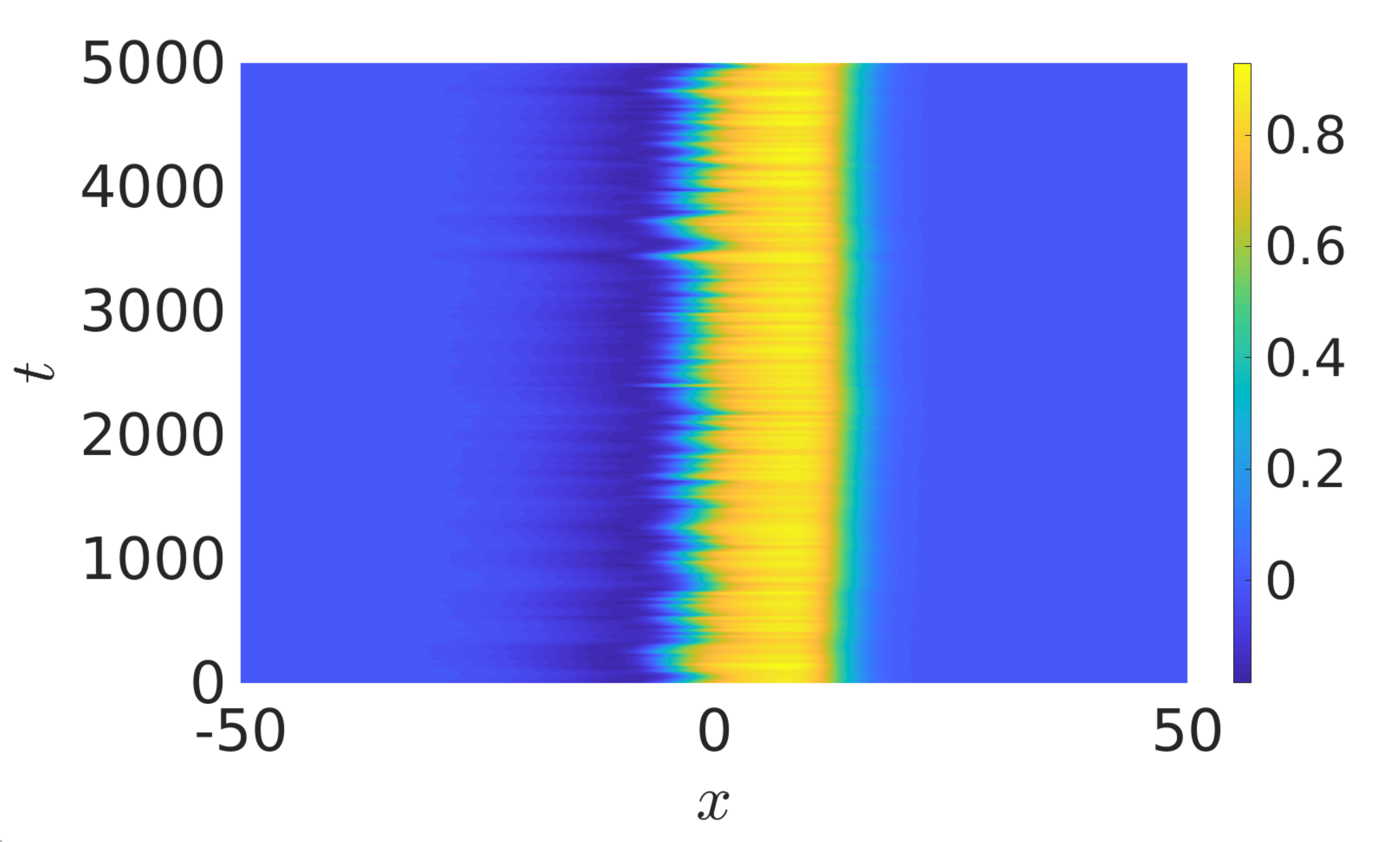
  \caption{$U(\cdot +\Gamma(t),t)$}
        \label{fig:UDifREfc}
\end{subfigure}
\caption{A single realization of the $U$-component of  \sref{eq:example:SPDE:fhn} with initial condition $\Phi_\s$ in 3 different reference frames. We chose $g^{(u)}(u)=u$ with parameters $a=0.1$, $\sigma=0.03$, $\varrho=0.01$, $\e=0.01$, $\gamma=5$.}
\label{fig:UDifRef}
\end{figure}

In order to study the orbital drift
mentioned above,
we split the semigroup $S(t)$ generated by $\mathcal{L}_{\mathrm{tw}}$ into its components
\begin{equation}
    S(t) = \left( \begin{array}{cc}
       S^{(uu)}(t) & S^{(uw)}(t) \\[0.2cm]
       S^{(wu)}(t) & S^{(ww)}(t)
    \end{array} \right)
\end{equation}
and introduce the expression
\begin{equation}
\begin{array}{lcl}
    \mathcal{I}(s) &  = & 
    S^{(uu)}(s) g^{(u)}(\Phi_0 ) 
    + \tilde{b}(\Phi_{0}) S^{(uu)} \partial_\xi \Phi^{(u)}_{0} 
    + \tilde{b}(\Phi_{0}) S^{(uw)} \partial_\xi \Phi^{(w)}_{0} ,
\end{array}
\end{equation}
together with
\begin{equation}
\label{eq:mr:c:zero:od}
c^{\mathrm{od}}_{0;2}
= -\frac{1}{2}\int_0^\infty
\ip{f''_{\mathrm{cub}}(\Phi_0^{(u)})
  \mathcal{I}(s)^2, \psi^{(u)}_{\mathrm{tw}}}_{L^2} \, ds.
\end{equation}
This last quantity is in fact the leading order term
in the Taylor expansion of 
\sref{eq:mr:def:c:orb:drift},
which means that
\begin{equation}
\begin{array}{lcl}
c^{\mathrm{od}}_{\sigma;2}
= c^{\mathrm{od}}_{0;2}
+ O(\sigma^2) .
\end{array}
\end{equation}
In particular, we see that
\begin{equation}
c^{(2)}_{\sigma;\lim} = c_0 + \sigma^2 \big[ c_{0;2} + c^{\mathrm{od}}_{0;2} \big]
+ O(\sigma^3) ,
\end{equation}
which means that we have explicitly identified 
the leading order correction to the full limiting wavespeed.

To validate our prediction for the size of the orbital drift,
we first approximated
$E[\Gamma(t)-c_\s t]$ numerically by performing an average over 
a set of numerical simulations. In fact, to speed
up the convergence rate, we first
subtracted the term $\Gamma_{\sigma;1}(t)$
defined in \sref{eq:gamma:sigma:1} from each simulation,
using the same realization of the
Brownian motion that was used to generate the path for $(U,W)$.
The results can be found in Figure \ref{fig:Difta}.

\begin{figure}
\centering
\begin{subfigure}{.5\textwidth}
  \centering
 		\def\svgwidth{\columnwidth}
    		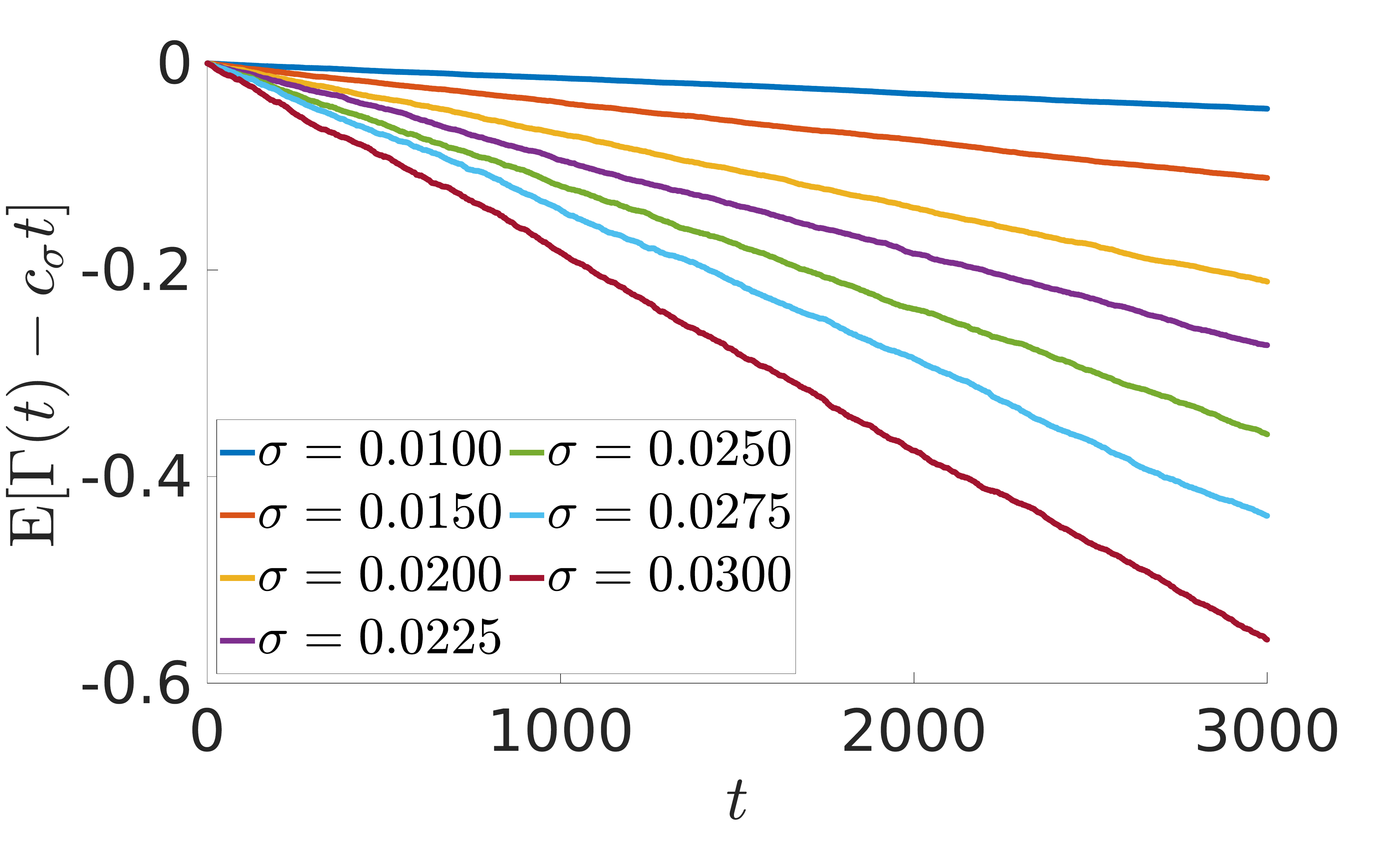
  \caption{}
  \label{fig:Difta}
\end{subfigure}%
\begin{subfigure}{.5\textwidth}
  \centering
 		\def\svgwidth{\columnwidth}
    		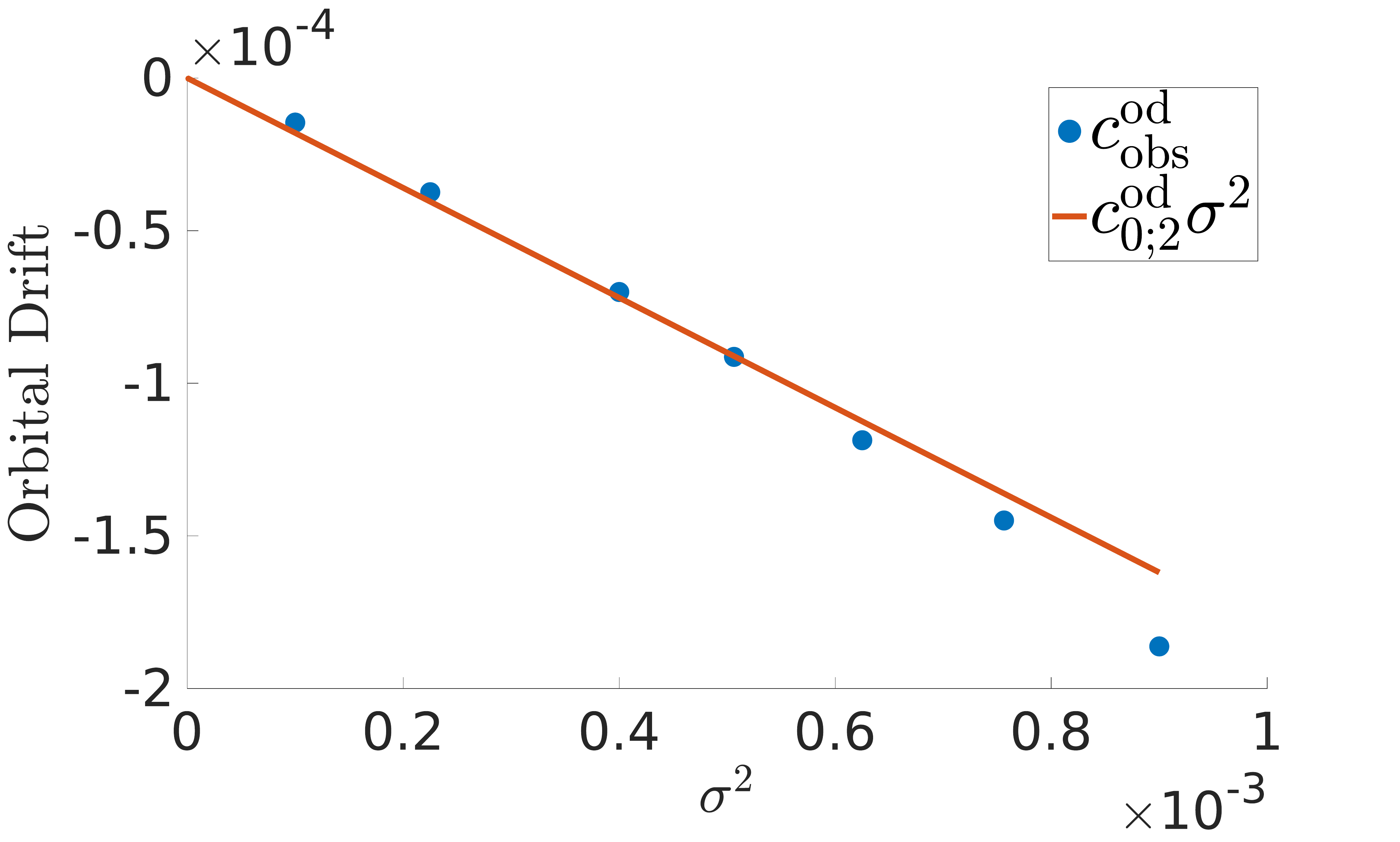
  \caption{}
    \label{fig:Diftb}
\end{subfigure}
\caption{In (a) we computed the average $E[\Gamma(t)-c_\s t]$
over 1000 simulations of \sref{eq:example:SPDE:fhn}, using
the procedure described in the main text for several values of $\s$.
Notice that a clear trend is visible. In (b) 
we computed the corresponding orbital drift
by evaluating the average \sref{eq:mr:def:c:od:obs}
for the data in (a).
Observe that there is a reasonable match with the predicted values $c_{0;2}^{\mathrm{od}} \sigma^2$. 
We chose $g^{(u)}(u)=u$ with parameters $a=0.1$, $\varrho=0.01$, $\e=0.01$, $\gamma=5$.
We used the value $c^{\mathrm{od}}_{0;2}=-0.18$,
which was found by evaluating \sref{eq:mr:c:zero:od} numerically.
}
\label{fig:Drift}
\end{figure}

In order to eliminate any transients from the data,
we subsequently numerically computed the
quantity
\begin{align}
\label{eq:mr:def:c:od:obs}
 c^{\mathrm{od}}_{\mathrm{obs}} 
 =\frac{2}{T} \int_{\frac{T}{2}}^T \frac{1}{t}E[\Gamma(t)-c_\s t]dt .
\end{align}
This corresponds with the average slope of
the data in  Figure \ref{fig:Difta} on the interval $[T/2, T]$,
which is a useful proxy for the observed orbital drift.
Figure \ref{fig:Diftb} shows that these 
quantities are well-approximated by 
our leading order expression $\sigma^2 c_{0;2}^{\mathrm{od}}$.


\section{Structure of the semigroup}
\label{sec:SplitSem}

In this section we analyze the analytic semigroup $S(t)$ generated by the linear operator
$\mathcal{L}_\mathrm{tw}$, focusing specially on its off-diagonal elements.
Assumption (HTw) implies that $\mathcal{L}_\mathrm{tw}$ has a spectral gap,
which is essential for our computations. In order to exploit this, we introduce the
maps $P: L^2 \to L^2$ and $Q: L^2 \to L^2$ that act as
\begin{equation}
P v = \langle v , \psi_{\mathrm{tw}} \rangle_{L^2} \Phi_0',
\qquad \qquad
Q v = v - P v  .
\end{equation}
We also introduce the suggestive notation $P_\xi \in \mathcal{L}(L^2; L^2)$
to refer to the map
\begin{equation}
\label{eq:spl:def:p:x}
P_\xi v = - \langle v , \partial_\xi \psi_{\mathrm{tw}} \rangle_{L^2} \Phi_0',
\end{equation}
noting that $P_\xi v = P \partial_\xi v $ whenever $v \in H^1$.
These projections
enable us to remove the simple eigenvalue at the origin
and obtain the following bounds.

\begin{lem}[see \cite{lorenzi2004analytic}]
\label{lem:nls:sem:group:decay}
Assume that (HDt) and (HTw) hold.
Then $\mathcal{L}_{\mathrm{tw}}$ generates an analytic
semigroup semigroup $S(t)$
and there exists a constant $M\geq 1$ for which
we have the bounds
\begin{equation}
\begin{array}{lclcl}
\nrm{S(t)Q}_{\L(L^2;L^2)}&\leq& M e^{-\b t},
  && 0<t<\infty ,
  \\[0.2cm]
\nrm{S(t)Q}_{\L(L^2;H^1)}&\leq& M t^{-\frac{1}{2}} ,
 &&  0<t\leq 2 ,
\\[0.2cm]
\nrm{S(t)P}_{\L(L^2;H^2)}
+
\nrm{S(t)P_\xi}_{\L(L^2;H^2)}
+ \nrm{S(t)\partial_\xi P }_{\mathcal{L}(L^2; H^2)}
&\leq& M , & & 0<t\leq 2 ,
\\[0.2cm]
\nrm{S(t)Q}_{\L(L^2;H^2)}&\leq& M e^{-\b t},
  & & t\geq 1,
\\[0.2cm]
\norm{[\mathcal{L}_\mathrm{tw} - \rho \partial_{\xi\xi}] S(t)Q}_{\L(L^2;L^2) }
  & \le & M t^{-\frac{1}{2}},
    & &  0 < t\leq 2,
\\[0.2cm]
\norm{[\mathcal{L}^*_\mathrm{tw} - \rho \partial_{\xi\xi}] S(t)Q }_{\L(L^2;L^2) }
  & \le & M t^{-\frac{1}{2}},
   & & 0 < t\leq 2 .
\end{array}
\end{equation}
\end{lem}
\begin{proof}
Since $\rho \partial_{\xi\xi}$ generates $n$ independent heat-semigroups,
the analyticity of the semigroup $S(t)$
can be obtained from \cite[Prop 4.1.4]{lorenzi2004analytic};
see also \cite[Prop 6.3.vi]{Hamster2017}. The desired bounds
follow from \cite[Prop 5.2.1]{lorenzi2004analytic}
together with the fact that $\Phi_0' \in H^3$.
\end{proof}

In {\S}\ref{sec:STT} we will show that the function $V(t)$ defined in \sref{eq:MR:DefV}
satisfies an SPDE that involves nonlinear terms containing second order derivatives.
The short-term bounds above are too crude to handle such terms as they lead to divergences
in the integrals governing short-time regularity. In addition, the variational
framework in \cite{LiuRockner} only provides control on the $H^1$-norm of $V$.

In order to circumvent the first issue, we
introduce the representation
\begin{equation}
S(t) v = \left( \begin{array}{ccc}
 S_{11}(t) & \ldots & S_{1n}(t) \\ \vdots & \ddots & \vdots \\
 S_{n1}(t) & \ldots  & S_{nn}(t)
 \end{array} \right) \left( \begin{array}{c} v_1 \\ \vdots \\ v_n \end{array} \right)
\end{equation}
with operators $S_{ij}(t) \in \mathcal{L}\big( L^2(\Real ; \Real) ; L^2(\Real; \Real) \big)$.
Upon writing
\begin{equation}
S_{\mathrm{d}}(t) = \mathrm{diag}\big( S_{11}(t) , \ldots, S_{nn}(t) \big)
\end{equation}
this allows us to make the splitting
\begin{equation}
S(t) = S_{\mathrm{d}}(t) + S_{\mathrm{od}}(t).
\end{equation}
Our main result below shows that the off-diagonal terms $S_{\mathrm{od}}(t)$
have better short-term bounds
than the original semigroup.


The second issue can be addressed by introducing the commutator
\begin{equation}
\Lambda(t) = [S(t) Q, \p_{\xi}] = S(t) Q \p_{\xi} - \p_{\xi} S(t) Q
\end{equation}
that initially acts on $H^1$. In fact, we show that this commutator
can be extended to $L^2$ in a natural fashion
and that it has better short-time bounds than $S(t)$.
Upon writing
\begin{equation}
\label{eq:spl:eqn:with:comm}
S(t) \p_{\xi} v = S(t) Q  \p_{\xi} v + S(t) P_\xi v = \p_{\xi} S(t) Q v + \Lambda(t) v + S(t) P_\xi v,
\end{equation}
we hence see that the right-hand side of this identity
is well-defined for $v \in L^2$.
In {\S}\ref{sec:STT} this observation will allow us to give a mild interpretation
of the SPDE satisfied by $V(t)$ posed on the space $H^1$.

\begin{prop}\label{prp:MainPropSplit}
Suppose that (HDt) and (HTw) are
satisfied. Then the operator $\Lambda(t)$ can be extended to $L^2$ for each $t \ge 0$.
In addition, there is a constant $M > 0$ so that the short-term bound
\begin{equation}
\label{eq:MainPropSpliti}
\nrm{\Lambda(t)}_{L^2\to H^2}
  +\nrm{S_{\mathrm{od}}(t)}_{L^2\to H^2}\leq M
\end{equation}
holds for $0 < t \le 1$,
while the long-term bound
\begin{equation}
\label{eq:MainPropSplitii}
\nrm{\Lambda(t)}_{L^2\to H^2}\leq Me^{-\b t}
\end{equation}
holds for $t \ge 1$.
\end{prop}

\subsection{Functional calculus}

For any linear operator $\mathcal{L}: H^2 \to L^2$ we introduce the notation
\begin{equation}
R(\mathcal{L},\lambda) = [\lambda - \mathcal{L}  ]^{-1}
\end{equation}
for any $\lambda$ in the resolvent set of $\mathcal{L}$.
On account of (HTw) and the sectoriality
of $\mathcal{L}_{\mathrm{tw}}$, we can find $\eta_+ \in (\frac{\pi}{2} , \pi)$ and $M > 0$
so that the sector
\begin{equation}
\Omega_{\mathrm{tw}} = \{ \lambda \in \mathbb{C}\setminus \{ 0 \}
  : \abs{\mathrm{arg}(\lambda) } < \eta_+ \}
\end{equation}
lies entirely in the resolvent set of $\mathcal{L}_{\mathrm{tw}}$,
with
\begin{equation}
\norm{ R(\mathcal{L}_{\mathrm{tw}} , \lambda) }_{L^2 \to L^2}
  \le \frac{M}{\abs{\lambda} }
\end{equation}
for all $\lambda \in \Omega_{\mathrm{tw}}$.
Since $\lambda = 0$ is a simple eigenvalue for $\mathcal{L}_{\mathrm{tw}}$, we have the
limit
\begin{equation}
\label{eq:spl:simple:pole}
\lambda R(\mathcal{L}_{\mathrm{tw}} , \lambda) \to P
\end{equation}
as $\lambda \to 0$.

For any $r > 0$ and any $\eta \in (\frac{\pi}{2}, \eta_+)$,
the curve given by
\begin{align}
\g_{r,\eta} = \{\l\in\C:|\mathrm{arg}\l| =\eta, |\l|>r\}
\cup\{\l\in\C:|\mathrm{arg}\l|\leq \eta, |\l|=r\}
\end{align}
lies entirely in $\Omega_{\mathrm{tw}}$.
This curve can be used \cite[(1.10)]{lorenzi2004analytic}
to represent the semigroup $S$
in the integral form
\begin{equation}\label{eq:IntDefS}
S(t) = \frac{1}{2 \pi i}\int_{\gamma_{r, \eta}} e^{t \lambda} R(\mathcal{L}_{\mathrm{tw}} , \lambda)
  \, d \lambda
\end{equation}
for any $ t  > 0$, where $\gamma_{r, \eta}$ is traversed in the upward direction.

We will analyze $\Lambda(t)$ and $S_{\mathrm{od}}(t)$ by manipulating this integral. As a preparation,
we state two technical results concerning the convergence
of contour integrals that are similar to \sref{eq:IntDefS}.
We note that our computations here are based rather directly on \cite[{\S}1.3]{lorenzi2004analytic}.


\begin{lem}
\label{lem:spl:est:from:k}
Suppose that (HDt) and (HTw) are satisfied
and pick  $r > 0$
together with $\eta \in (\frac{\pi}{2} , \eta_+)$.
Suppose furthermore that $\lambda \mapsto K(\lambda) \in \mathbb{C}$ is an analytic function
on the resolvent set of $\mathcal{L}_{\mathrm{tw}}$
and that there
exist constants $C > 0$ and $\vartheta \ge 1$
so that the estimate
\begin{equation}
\abs{K(\lambda) } \le \frac{C}{\abs{\lambda}^{\vartheta} }
\end{equation}
holds for all $\lambda \in \Omega_{\mathrm{tw}}$.
Then there exists $C_1 > 0$ so that
\begin{equation}
 \abs{\int_{\gamma_{r, \eta} }  e^{\lambda t}  K(\lambda)
 \, d \lambda } \le C_1 t^{ \vartheta -1 }
\end{equation}
for all $t > 0$.
\end{lem}
\begin{proof}
Writing
\begin{equation}
\mathcal{I}(t) = \int_{\gamma_{r, \eta}} e^{\lambda t}
  K(\lambda) \, d \lambda
\end{equation}
and substituting $\lambda t = \xi$,
the analyticity of $K$ on $\Omega_{\mathrm{tw}}$
implies
\begin{equation}
\mathcal{I}(t) = \int_{\gamma_{rt, \eta}} e^{\xi}
  K\left(\frac{\xi}{t}\right) \, \frac{1}{t} d \xi
  = \int_{\gamma_{r, \eta}} e^{\xi}
  K\left(\frac{\xi}{t}\right) \, \frac{1}{t} d \xi .
\end{equation}
Using the obvious parametrization for $\g_{r,\eta}$, we find
\begin{equation}
\label{eq:IntegrantDefT}
\begin{array}{lcl}
\mathcal{I}(t) &=&
 -\int_r^\infty
  e^{(\rho\cos(\eta)-i\rho\sin(\eta))}
    K\big(t^{-1} \rho e^{-i\eta} \big)
    e^{-i\eta} t^{-1} d\rho
    \\[0.2cm]
& & \qquad +\int_\eta^\eta e^{(r\cos(\a)-ir\sin(\a))}
  K\big( t^{-1} re^{i\a}) \big)
     ire^{i\a} t^{-1} d\a
     \\[0.2cm]
&& \qquad + \int_r^\infty e^{(\rho\cos(\eta)-i\rho\sin(\eta))}
  K\big( t^{-1} \rho e^{i\eta}\big)
      e^{i\eta} t^{-1} d\rho .
\end{array}
\end{equation}
We hence obtain the desired estimate
\begin{equation}
\label{eq:IntegrantDefTNorm}
\begin{array}{lcl}
\abs{\mathcal{I}(t)}
 &\leq&
  C t^{\vartheta - 1}
  \left(2\int_r^\infty e^{\rho\cos(\eta)} \rho^{-\vartheta}  \, d\rho
   +\int_\eta^\eta e^{r\cos(\a)}   r^{1 - \vartheta}
     d\a\right)\\[0.2cm]
 &:= & C_1 t^{\vartheta - 1} .
\end{array}
\end{equation}
\end{proof}

\begin{lem}
\label{lem:spl:exp:est:for:bnd:itg}
Suppose that (HDt) and (HTw) are satisfied
and pick  $r > 0$
together with $\eta \in (\frac{\pi}{2} , \eta_+)$.
Suppose furthermore that $\lambda \mapsto K(\lambda)$
is an analytic function
on the resolvent set of $\mathcal{L}_{\mathrm{tw}}$
and that there
exists a constant $C > 0$
so that the estimate
\begin{equation}
\abs{K(\lambda) } \le C
\end{equation}
holds for all $\lambda \in \Omega_{\mathrm{tw}}$.
Then there exists $C_2 > 0$ so that the bound
\begin{equation}
 \abs{ \int_{\gamma_{r, \eta} }  e^{\lambda t} K(\lambda)
 \, d \lambda } \le C_2 e^{ - \beta t }
\end{equation}
holds for all $ t \ge 1$.
\end{lem}
\begin{proof}
Since $K$ remains bounded for $\lambda \to 0$,
this function can be analytically extended to a neighborhood of $\lambda = 0$.
We can hence replace the curve $\gamma_{r, \eta}$ by the
two half-lines
\begin{equation}
  \tilde{\gamma}_{\eta'} =  - \beta +  \{ \lambda \in \mathbb{C}: \abs{\mathrm{arg} \lambda } =
    \eta'  \}
\end{equation}
for appropriate $\eta' \in (\frac{\pi}{2} , \eta_+)$. We can then compute
\begin{equation}
\begin{array}{lcl}
\abs{ \int_{\tilde{\gamma}_{\eta'} } e^{\lambda t} K(\lambda)
 \, d \lambda}
  & \le &
      2 C  e^{-  \beta t}
       \int_0^\infty e^{ \rho \cos(\eta') t } \, d \rho
  \\[0.2cm]
  & \le &
    2 C  e^{-  \beta t}
       \int_0^\infty e^{ \rho \cos(\eta')  } \, d \rho
  \\[0.2cm]
  & := & C_2 e^{-  \beta t} .
 \end{array}
\end{equation}
\end{proof}

\subsection{The commutator $\Lambda(t)$}
In this section we analyze $\Lambda(t)$
and establish the statements in
Proposition \ref{prp:MainPropSplit} that concern this commutator.
Based on the identity \sref{eq:IntDefS},
we first set out to compute
the commutator of $R(\mathcal{L}_{\mathrm{tw}}, \lambda)$ and $\p_{\xi}$.
As a preparation,
we introduce the commutator
\begin{equation}
\label{eq:spl:def:B}
B =  [\L_{\mathrm{tw}}  Q, \p_{\xi}]
 =   [\L_{\mathrm{tw}} , \p_{\xi}] ,
\end{equation}
which can easily be seen to act as
\begin{equation}
B v =  - D^2 f(\Phi_0)\Phi_0' v
\end{equation}
for any $v \in H^3$.

\begin{lem}\label{lem:split:res}
Suppose that (HDt) and (HTw) are satisfied and pick any $\lambda$ in the resolvent set of $\mathcal{L}_{\mathrm{tw}}$.
Then for any $g \in H^1$ we have the identity
\begin{equation}
\label{eq:split:comm:resolv}
\begin{array}{lcl}
[R(\mathcal{L}_{\mathrm{tw}},\lambda) Q, \partial_\xi ] g
& = & R(\mathcal{L}_{\mathrm{tw}},\lambda) Q \partial_\xi g - \partial_\xi R(\mathcal{L}_{\mathrm{tw}},\lambda)Q g
\\[0.2cm]
& = &
    R(\mathcal{L}_{\mathrm{tw}},\lambda)
     \Big[  B R(\mathcal{L}_{\mathrm{tw}},\lambda) Qg - [P, \partial_\xi] g \Big] .
\end{array}
\end{equation}
\end{lem}
\begin{proof}
Let us first write
\begin{equation}
v = [\lambda - \mathcal{L}_{\mathrm{tw}} ]^{-1} Q g.
\end{equation}
The definition \sref{eq:spl:def:B} implies that
\begin{equation}
\begin{array}{lcl}
[\lambda - \mathcal{L}_{\mathrm{tw}} ] Q\partial_\xi v
& = &
  \partial_\xi [\lambda - \mathcal{L}_{\mathrm{tw}} ] Q v
    - B v
    + \lambda [Q , \partial_\xi] v
\\[0.2cm]
& = &
\partial_\xi [\lambda - \mathcal{L}_{\mathrm{tw}}]  v
   - \partial_\xi \lambda (I -Q) v
    - B v
    + \lambda [Q, \partial_\xi] v
\\[0.2cm]
& = &
\partial_\xi [\lambda - \mathcal{L}_{\mathrm{tw}} ]  v
   - \lambda \partial_{\xi} P v
    - B v
    - \lambda [P, \partial_{\xi}] v
\\[0.2cm]
& = &
   \partial_{\xi} Q g -  B v - \lambda P \partial_{\xi} v
\\[0.2cm]
& = &  Q \partial_{\xi} g     - [P , \partial_{\xi}] g  -  B v - \lambda P \partial_{\xi} v .
\end{array}
\end{equation}
Using $(\lambda - \mathcal{L}_{\mathrm{tw}} )^{-1} P =  \lambda^{-1} P$
we obtain
\begin{equation}
\begin{array}{lcl}
[\lambda - \mathcal{L}_{\mathrm{tw}} ]^{-1}  Q  \partial_{\xi} g
 & = &
    Q \partial_{\xi}  v
    + [\lambda - \mathcal{L}_{\mathrm{tw}}]^{-1} B v
     + P \partial_{\xi} v
    + [\lambda - \mathcal{L}_{\mathrm{tw}} ]^{-1} [P , \partial_{\xi}] g
\\[0.2cm]
& = &
     \partial_{\xi}  [\lambda - \mathcal{L}_{\mathrm{tw}} ]^{-1} Q g
    + [\lambda -\mathcal{L}_{\mathrm{tw}} ]^{-1} B
      [\lambda - \mathcal{L}_{\mathrm{tw}} ]^{-1} Q g
\\[0.2cm]
& & \qquad
      +  [\lambda - \mathcal{L}_{\mathrm{tw}} ]^{-1} [P , \partial_{\xi}] g
       ,
\end{array}
\label{eq:ResIden}
\end{equation}
which can be reordered to yield \sref{eq:split:comm:resolv}.
\end{proof}

On account of
\sref{eq:split:comm:resolv}
we recall the definition \sref{eq:spl:def:p:x}
and introduce the operator $T_A \in \mathcal{L}(L^2 ; L^2)$
that acts as
\begin{equation}
T_A  = \partial_{\xi} P - P_\xi.
\end{equation}
In addition, we introduce the expression
\begin{equation}
T_B(\lambda)  = B R(\L_{\mathrm{tw}},\l)Q,
\end{equation}
which is well-behaved in the following sense.

\begin{lem}
\label{eq:spl:est:on:TB}
Suppose that (HDt) and (HTw) are satisfied.
Then there exists a constant $C > 0$
so that for any $\lambda$
in the resolvent set of $\mathcal{L}_{\mathrm{tw}}$
the operator $T_B(\lambda)$ satisfies the bound
\begin{equation}
\norm{T_B(\lambda)}_{L^2 \to L^2}
 \le  \frac{C}{1 + \abs{\lambda}}.
\end{equation}
In additions, the maps
\begin{equation}
\label{eq:spl:an:continuations:TAB}
\lambda \mapsto T_B(\lambda) \in \mathcal{L}(L^2;L^2),
\qquad
\lambda \mapsto \lambda^{-1} P \big[T_A + T_B(\lambda) \big] \in \mathcal{L}(L^2;L^2)
\end{equation}
can be continued analytically into the origin $\lambda = 0$.
\end{lem}
\begin{proof}
Since
$\Phi_0$ and $\Phi_0'$ are bounded functions,
we have
\begin{align}
\nrm{ BR(\L_{\mathrm{tw}},\l)}_{L^2\to L^2}&\leq
\frac{M}{\abs{\l}}\nrm{D^2 f(\Phi_0)\Phi_0'}_{\infty}.
\end{align}
Using $P \mathcal{L}_{\mathrm{tw}} = 0$
and the resolvent identity
\begin{equation}
\mathcal{L}_{\mathrm{tw}} R(\mathcal{L}_{\mathrm{tw}} , \lambda)
 = - I + \lambda R(\mathcal{L}_{\mathrm{tw}} , \lambda),
\end{equation}
we may compute
\begin{equation}
\begin{array}{lcl}
P [T_A + T_B(\lambda) \big]
& = & P_\xi P - P_\xi  +   P B R(\mathcal{L}_{\mathrm{tw}} , \lambda) Q
\\[0.2cm]
& = &
P_\xi P - P_\xi
+ P \mathcal{L}_{\mathrm{tw}} \partial_{\xi}
  R(\mathcal{L}_{\mathrm{tw}} , \lambda) Q
  - P \partial_{\xi} \mathcal{L}_{\mathrm{tw}}
       R(\mathcal{L}_{\mathrm{tw}} , \lambda) Q
\\[0.2cm]
& = &
P_\xi P - P_\xi
  + P_\xi Q
  - P \partial_{\xi} \lambda
       R(\mathcal{L}_{\mathrm{tw}} , \lambda) Q
\\[0.2cm]
& = &
  - P \partial_{\xi} \lambda R(\mathcal{L}_{\mathrm{tw}} , \lambda ) Q .
\end{array}
\end{equation}
Since $\lambda \mapsto R(\mathcal{L}_{\mathrm{tw}} , \lambda ) Q $
can be analytically continued to $\lambda = 0$
on account of \sref{eq:spl:simple:pole},
the same hence holds for
the functions \sref{eq:spl:an:continuations:TAB}.
\end{proof}

Upon fixing $r > 0$ and $\eta \in (\frac{\pi}{2}, \eta_+ )$,
we now introduce the expressions
\begin{equation}
\label{eq:spl:defs:lambda:a:b}
\begin{array}{lcl}
\Lambda_{\mathrm{ex};A}(t)
  & = & \frac{1}{2 \pi i} \int_{\g_{r,\eta}} e^{\l t}
    R(\L_{\mathrm{tw}},\l) T_A \, d\l ,
\\[0.2cm]
\Lambda_{\mathrm{ex};B}(t)
  & = & \frac{1}{2 \pi i} \int_{\g_{r,\eta}} e^{\l t}
     R(\L_{\mathrm{tw}},\l)   T_B(\lambda) \,  d\l
\end{array}
\end{equation}
and write
\begin{equation}
\Lambda_{\mathrm{ex}}(t) =
\Lambda_{\mathrm{ex};A}(t) + \Lambda_{\mathrm{ex};B}(t).
\end{equation}

We note that
\begin{equation}
\Lambda_{\mathrm{ex};A}(t)
 = S(t) T_A = S(t) \partial_{\xi} P  - S(t) P_\xi,
\end{equation}
which for $0 < t \le 1$ is covered
by the bounds in Lemma \ref{lem:nls:sem:group:decay}.
The results below show that
$\Lambda_{\mathrm{ex}}(t)$
is well-defined as an operator in
$\mathcal{L}(L^2 ; H^2)$ and that it
is indeed an extension of the commutator $\Lambda(t)$.

\begin{lem}\label{lem:LexWellDef}
Suppose that (HDt) and (HTw) are satisfied. Then
$\Lambda_{\mathrm{ex} }(t)$
is a well-defined operator in
$\mathcal{L}(L^2, H^2)$ for all $t > 0$
that does not depend on $r > 0$ and $\eta \in (\frac{\pi}{2} , \eta_+)$.
In addition, there exists a constant $C > 0$ so that
the bound
\begin{equation}
\nrm{ \Lambda_{\mathrm{ex}}(t) }_{L^2 \to H^2} \le C e^{-\beta t }
\end{equation}
holds for all $t > 0$.
\end{lem}
\begin{proof}
Note first that there exists a constant $C'_1 > 0$ for which
\begin{equation}
\norm{v}_{H^2} \le C'_1 \big[ \norm{\mathcal{L}_{\mathrm{tw}} v}_{L^2} + \norm{v}_{L^2} \big]
\end{equation}
holds for all $v \in H^2$.
On account of the identity
\begin{equation}
\mathcal{L}_{\mathrm{tw}} R(\mathcal{L}_{\mathrm{tw}} , \lambda)
\big[ T_A + T_B(\lambda)\big]
= - \big[ T_A + T_B(\lambda)\big]
+ \lambda R(\mathcal{L}_{\mathrm{tw}} , \lambda)
  \big[ T_A + T_B(\lambda) \big]
\end{equation}
and the analytic continuations \sref{eq:spl:an:continuations:TAB},
we see that there exist $C_2' > 0$ so that
\begin{equation}
\norm{ \mathcal{L}_{\mathrm{tw}} R(\mathcal{L}_{\mathrm{tw}} , \lambda)
\big[ T_A + T_B(\lambda)\big] }_{L^2 \to L^2}
+
\norm{ R(\mathcal{L}_{\mathrm{tw}} , \lambda)
\big[ T_A + T_B(\lambda)\big] }_{L^2 \to L^2}
\le C_2'
\end{equation}
for all $\lambda \in \Omega_{\mathrm{tw}}$.
We can now apply Lemma \ref{lem:spl:exp:est:for:bnd:itg}
to obtain the desired bound for $t \ge 1$.

The bounds in Lemma \ref{eq:spl:est:on:TB}
imply that there exists  $C_3' > 0$ for which
\begin{equation}
\begin{array}{lcl}
\norm{ \mathcal{L}_{\mathrm{tw}} R(\mathcal{L}_{\mathrm{tw}} , \lambda)
\big[ T_B(\lambda)\big] }_{L^2 \to L^2}
& \le & \frac{C_3'}{1 + \abs{\lambda} }
\\[0.2cm]
\norm{ R(\mathcal{L}_{\mathrm{tw}} , \lambda)
\big[  T_B(\lambda)\big] }_{L^2 \to L^2}
& \le & \frac{C_3'}{\abs{\lambda} }
\end{array}
\end{equation}
holds for all $\lambda \in \Omega_{\mathrm{tw}}$.
We can  hence use Lemma
\ref{lem:spl:est:from:k}
to find a constant $C_4' > 0$ for which
we have the bound
\begin{equation}
\norm{\Lambda_{\mathrm{ex}; B}(t) }_{L^2 \to H^2} \le C_4'
\end{equation}
for all $0 < t \le 1$. A direct application
of Lemma \ref{lem:nls:sem:group:decay}
shows that also
\begin{equation}
\norm{\Lambda_{\mathrm{ex}; A}(t) }_{L^2 \to H^2} \le M
\end{equation}
for all $0 < t \le 1$, which completes the proof.
\end{proof}

\begin{cor}\label{cor:LambdaExt}
Suppose that (HDt) and (HTw) are satisfied.
Then for any $g \in H^1$ we have
\begin{equation}
\Lambda_{\mathrm{ex}}(t) g
=  \Lambda(t) g := [ S(t) Q , \partial_{\xi} ] g .
\end{equation}
\end{cor}
\begin{proof}
The result follows by integrating
both sides of the identity
\sref{eq:split:comm:resolv}
over the contour
$\g_{r,\eta}$
and using
\sref{eq:IntDefS}
together with
\sref{eq:spl:defs:lambda:a:b}.
\end{proof}

\subsection{Semigroup block structure}
\label{sec:SemGrBlock}
For the nonlinear stability proof in \S\ref{sec:nls} we need to understand how the off-diagonal terms of
$S(t)$ act on a second order nonlinearity.
In order to do this,
we first write
$S_{\mathrm{d};I}(t)$ for the semigroup generated by
\begin{equation}
\mathcal{L}_{\mathrm{tw;d}} = \rho\p_{\xi\xi}v + c_0 v_{\xi},
\end{equation}
which contains only diagonal terms.
We also write
\begin{equation}
S_{\mathrm{od};I}(t) = S(t) - S_{\mathrm{d};I}(t)
\end{equation}
for the rest of the semigroup.
Note that $S_{\mathrm{od};I}(t)$ is not
strictly off-diagonal,
but it has the same off-diagonal elements
as $S_{\mathrm{od}}(t)$.

\begin{lem}\label{lem:DecayPropG}
Suppose that (HDt) and (HTw) are satisfied.
Then there exists a constant $C>0$
for which the short-term bound
\begin{align}
\begin{split}
\nrm{S_{\mathrm{od;I}}(t)}_{L^2\to H^2}&\leq C
\end{split}
\end{align}
holds for all $0 \le t \le 1$.
\end{lem}
\begin{proof}
Possibly decreasing the size of $\eta_+$,
we may assume that $\Omega_{\mathrm{tw}}$
is contained in the resolvent set of
$\mathcal{L}_{\mathrm{tw};\mathrm{d}}$. We may also assume
that the bound
\begin{equation}
\norm{R(\mathcal{L}_{\mathrm{tw};\mathrm{d}}, \lambda)}_{L^2 \to L^2}
 \le \frac{M}{\abs{\lambda}}
\end{equation}
holds for $\lambda \in \Omega_{\mathrm{tw}}$
by increasing the size of $M > 0$ if necessary.

For any $r > 0$ and $\eta \in (\frac{\pi}{2} , \eta_+)$
we have
\begin{align}\begin{split}
S_{\mathrm{od;I}}(t)&=
\frac{1}{2\pi i}\int_{\g_{r, \eta}} e^{\l t}[R(\L_{\mathrm{tw}},\l)-R(\L_{\mathrm{tw;d}},\l)] \, d\l
\\[0.2cm]
&=\frac{1}{2\pi i}\int_{\g_{r, \eta}}
e^{\l t}R(\L_{\mathrm{tw}},\l)
(\L_{\mathrm{tw}}-\L_{\mathrm{tw;d}})
R(\L_{\mathrm{tw;d}},\l) \,d\l
\\[0.2cm]
&=  \frac{1}{2\pi i}\int_{\g_{r, \eta}}
e^{\l t}R(\L_{\mathrm{tw}},\l)
Df(\Phi_0)
R(\L_{\mathrm{tw;d}},\l) \, d\l .
\end{split}
\end{align}
On account of the identity
\begin{equation}
\mathcal{L}_{\mathrm{tw}} R(\mathcal{L}_{\mathrm{tw}} , \lambda)
Df(\Phi_0)
R(\L_{\mathrm{tw;d}},\l)
= - Df(\Phi_0)
R(\L_{\mathrm{tw;d}},\l)
+ \lambda  R(\mathcal{L}_{\mathrm{tw}} , \lambda)
Df(\Phi_0)
R(\L_{\mathrm{tw;d}},\l)
\end{equation}
we have the bounds
\begin{equation}
\begin{array}{lcl}
\norm{ \mathcal{L}_{\mathrm{tw}}
R(\mathcal{L}_{\mathrm{tw}} , \lambda)
Df(\Phi_0)
R(\L_{\mathrm{tw;d}},\l) }_{L^2 \to L^2}
& \le & \norm{Df(\Phi_0)}_\infty \frac{M(M+1)}{\abs{\lambda}},
\\[0.2cm]
\norm{ R(\mathcal{L}_{\mathrm{tw}} , \lambda)
Df(\Phi_0)
R(\L_{\mathrm{tw;d}},\l) }_{L^2 \to L^2}
& \le & \norm{Df(\Phi_0)}_\infty
  \frac{M^2}{\abs{\lambda}^2 }.
\end{array}
\end{equation}
The desired estimate hence follows from
Lemma \ref{lem:spl:est:from:k}.
\end{proof}

\begin{proof}[Proof of Proposition \ref{prp:MainPropSplit}]
The statements concerning $\Lambda(t)$
follow directly from
Lemma \ref{lem:LexWellDef}
and Corollary \ref{cor:LambdaExt}.
The bound for $S_{\mathrm{od}}(t)$
follows from Lemma \ref{lem:DecayPropG}
since $S_{\mathrm{od;I}}(t)$
contains all the non-trivial elements of $S_{\mathrm{od}}(t)$.
\end{proof}


\section{Stochastic transformations}
\label{sec:STT}

In this section we set out to derive a mild
formulation for the SPDE satisfied by
the process
\begin{equation}
\label{eq:sps:def:V}
V(t) = T_{-\Gamma(t)} [ U(t)]
  - \Phi_\s ,
\end{equation}
which measures
the deviation from the traveling wave $\Phi_\s$ in the coordinate $\xi=x-\G(t)$.
After recalling several results from \cite{Hamster2017}
concerning the stochastic phaseshift,
we focus on the new extra second-order nonlinearity that appears in our setting.
We use the results from \S\ref{sec:SplitSem} to rewrite
this term in such a way that an effective mild integral equation
can be formulated that does not involve second derivatives.
We obtain estimates on all the nonlinear terms in \S\ref{sec:BoundsNl}
and rigorously verify that $V$ indeed satisfies this mild equation in \S\ref{sec:MildForm}.

We start by introducing the
nonlinearity \begin{equation}\label{eq:stt:DefR}
\begin{array}{lcl}
\mathcal{R}_\s
(v)
& = &
 \kappa_{\sigma}(\Phi_\s + v , \psi_{\mathrm{tw}} )
\rho \partial_{\xi\xi} [\Phi_\s + v]
\\[0.2cm]
& & \qquad
 + f(\Phi_\s + v)
+ \sigma^2  b(\Phi_\s + v, \psi_{\mathrm{tw}} ) \partial_\xi[ g(\Phi_\s + v) ]
\\[0.2cm]
& & \qquad
+ \Big[c_\s + a_{\sigma}\big(\Phi_\s + v , c_\s , \psi_{\mathrm{tw}}
\big) \Big] [\Phi_\s' + v'] ,
\end{array}
\end{equation}
together with
\begin{equation}
\begin{array}{lcl}
\mathcal{S}_{\s}(v)
& = &
     g( \Phi_\s + v )
     + b( \Phi_\s + v , \psi_{\mathrm{tw}}) [\Phi_\s' + v'].
\\[0.2cm]
\end{array}
\end{equation}
In \cite[{\S}5]{Hamster2017} we established that the shifted process $V$ can be interpreted as a
weak solution to the SPDE
\begin{equation}
\label{eq:WeakV}
\begin{array}{lcl}
d V & = &
\mathcal{R}_\s(V) \, dt
 + \sigma \mathcal{S}_\s(V) d\b_t .
\end{array}
\end{equation}
However, in our case here $\kappa_{\sigma}$ is a matrix rather than a
scalar. This means that we cannot transform \sref{eq:WeakV} into a semilinear problem
by a simple time transformation. But,
we can improve individual components of the system
by rescaling time with the diagonal elements $\k_{\s;i}$.

To this end, we follow \cite[Lem. 3.6]{Hamster2017}
to find a constant $K_{\kappa} > 0$
for which
\begin{equation}
\label{eq:prlm:kappa:glb:ests}
1 \leq  \kappa_{\sigma;i}(\Phi_\s + v, \psi_\mathrm{tw})
  \le K_{\kappa}
\end{equation}
holds for every
$\sigma \in (-\delta_{\sigma}, \delta_{\sigma})$,
every $v \in H^1$ and every $1 \le i \le n$.
Upon introducing the transformed time
\begin{equation}
\tau_{i}(t, \omega) = \int_0^t
  \kappa_{\sigma;i}\big( \Phi_\s + V(s, \omega) , \psi_{\mathrm{tw}} \big)  \, d s ,
\end{equation}
the bound \sref{eq:prlm:kappa:glb:ests}
allows us to conclude that
$t \mapsto \tau_i(t)$ is a continuous
strictly increasing $(\mathcal{F}_t)$-adapted
process that satisfies
\begin{equation}
\label{eq:mild:ineqs:tau:phi}
t \le \tau_i(t) \le K_{\kappa} t
\end{equation}
for $0 \le t \le T$.
In particular, we can define a map
\begin{equation}
\label{eq:mild:defn:t:phi:spaces}
t_i: [0, T] \times \Omega \to [0, T]
\end{equation}
for which
\begin{equation}
\label{eq:mild:defn:t:phi}
\tau_i(t_i(\tau, \omega), \omega) = \tau.
\end{equation}



This in turn allows us to introduce the
time-transformed map
\begin{equation}
\oV_i: [0,T] \times \Omega \to L^2
\end{equation}
that acts as
\begin{equation}
\label{eq:mld:def:ovl:v}
\begin{array}{lcl}
\overline{V_i}(\tau, \omega)
& = & V\big( t_i(\tau, \omega) , \omega \big) .
\\[0.2cm]
\end{array}
\end{equation}

Upon introducing
\begin{equation}
\label{eq:md:overline:r:sigma}
\begin{array}{lcl}
\overline{\mathcal{R}}_{\s;i}
(v)
& = &
  \kappa_{\sigma;i}(\Phi_\s + v , \psi_{\mathrm{tw}})^{-1}
    \mathcal{R}_{\s}(v)
  - \mathcal{L}_{\mathrm{tw}} v
\\[0.2cm]
\end{array}
\end{equation}
together with
\begin{equation}
\label{eq:md:overline:s:sigma}
\begin{array}{lcl}
\overline{\mathcal{S}}_{\s;i}(v)
& = &
\kappa_{\sigma;i}(\Phi_\s + v,\psi_{\mathrm{tw}})^{-1/2}
\mathcal{S}_{\s}(v),
\\[0.2cm]
\end{array}
\end{equation}
it is possible to follow \cite[Prop. 6.3]{Hamster2017} to show that $\overline{V}_i$ is a weak solution of
\begin{equation}
\label{eq:mild:weak:formul}
d \overline{V}_i
=
\big[\mathcal{L}_{\mathrm{tw}} \overline{V}_i +
  \overline{\mathcal{R}}_{\s;i}(\overline{V}_i) \big] \, d \tau
 + \sigma \overline{\mathcal{S}}_{\s;i}( \overline{V}_i) d\overline{\b}_{\tau;i}
\end{equation}
for every $1 \le i \le n$,
in which $(\overline{\b}_{\tau;i})_{\tau \ge 0}$
denotes the time-transformed
Brownian motion that is now adapted
to an appropriately transformed filtration
$(\overline{\mathcal{F}}_{\tau;i})_{\tau \ge 0}$;
see \cite[Lem. 6.2]{Hamster2017}.

The nonlinearity
$\overline{\mathcal{R}}_{\s;i}$
is less well-behaved than its counterpart
from \cite[Prop. 6.3]{Hamster2017}
since it still contains second order derivatives.
In order to isolate these terms,
we pick any $v \in H^1$ and
introduce the diagonal matrix
\begin{equation}\label{eq:DefPhi}
\phi_{\sigma;i}(v) =
  \big[\k_{\s;i}(\Phi_\s+v,\psi_{\mathrm{tw}}) \big]^{-1} \k_{\sigma}(\Phi_\s + v, \psi_{\mathrm{tw}} )
  - I
\end{equation}
together with the function
\begin{align}
\Upsilon_{\s;i}(v)&=\rho\phi_i(v)\p_{\xi} v .
\label{eq:DefUpsilon}
\end{align}

We note that $\p_{\xi} \Upsilon_{\s;i}$
can be considered as the error
caused by allowing unequal diffusion coefficients
in our main structural assumption (HDt).
Indeed,
upon defining our final nonlinearity
implicitly by
imposing the splitting
\begin{equation}
\label{eq:transf:def:w}
\overline{\mathcal{R}}_{\s;i}(v)=
\mathcal{W}_{\s;i}(v)+\p_{\xi} \Upsilon_{\s;i}(v),
\end{equation}
our first main result
states that $\mathcal{W}_{\s;i}$
is well-behaved in the sense
that it admits bounds that are similar
to those derived for the full nonlinearity
$\overline{\mathcal{R}}$ in \cite{Hamster2017}.
Indeed, it
depends at most quadratically on $\norm{v}_{H^1}$
but not on $\nrm{v}_{H^2}$.
Note furthermore that
$\Phi_{\sigma}$ was constructed in such a way
that $\overline{\mathcal{R}}(0) = 0$.

\begin{prop}
\label{prp:fnl:bnds}
Assume that (HDt), (HSt) and (HTw) all hold and fix $1 \le i \le n$.
Then there exist constants $K > 0$ and $\delta_v > 0$
so that for any
$0 \le \sigma \le \delta_{\sigma}$
and any
$v \in H^1$,
the following properties hold true.
\begin{itemize}
\item[(i)]{
We have the bound
\begin{equation}
\label{eq:fnl:rs:glb}
\begin{array}{lcl}
\norm{
\mathcal{W}_{\s;i}(v)
}_{L^2}
& \le &
 K \sigma^2 \norm{v}_{H^1}
+ K \norm{v}_{H^1}^2 \big[
  1
  + \norm{v}_{L^2}^2
  + \sigma^2 \norm{v}_{L^2}^3
\big] ,
\end{array}
\end{equation}
together with
\begin{equation}
\label{eq:mld:bnd:upsilon}
\norm{\Upsilon_{\s;i}(v)}_{L^2} \leq  K \s^2 \nrm{v}_{H^1}.
\end{equation}
}
\item[(ii)]{
We have the estimate
\begin{equation}
\norm{
\overline{\mathcal{S}}_{\s;i}(v)
}_{L^2}
\le K\big[ 1 +  \norm{v}_{H^1}
  \big].
\end{equation}
}
%
\item[(iii)]{
  If $\norm{v}_{L^2} \le \delta_v$, then we have
  the identities
  \begin{equation}\label{eq:MildNonProjZero}
    \langle \overline{\mathcal{R}}_{\s;i}(v) , \psi_{\mathrm{tw}} \rangle_{L^2}
    =
    \langle \overline{\mathcal{S}}_{\s;i}(v) , \psi_{\mathrm{tw}} \rangle_{L^2} = 0.
  \end{equation}
}
\end{itemize}
\end{prop}

The second main result of this section
formulates a mild representation for solutions to
\sref{eq:mild:weak:formul}.
Items (i)-(iv) are included for completeness and are analogous to the results in
\cite[Prop. 6.3]{Hamster2017}. However, item (v) is specific
to our situation because of the presence of the error term $\Upsilon_{\s;i}$.
Indeed, we shall need to exploit the techniques developed in {\S}\ref{sec:SplitSem}
to transfer the troublesome $\partial_{\xi}$ present in \sref{eq:transf:def:w} from the
$\Upsilon_{\s;i}$ term to the semigroup.
Nevertheless, the integral involving $\partial_{\xi} S$ is integrable in $H^{-1}$
but not necessarily in $L^2$.

\begin{prop}
\label{prp:MildTransSys}
Assume that (HDt), (HSt), (HTw) are all satisfied.
Then the map
\begin{equation}
\overline{V}_i: [0, T] \times \Omega \to L^2
\end{equation}
defined by the transformations
\sref{eq:sps:def:V}
and
\sref{eq:mld:def:ovl:v}
satisfies the following properties.
\begin{itemize}
\item[(i)]{
  For almost all $\omega \in \Omega$,
  the map $\tau \mapsto \overline{V_i}(\tau; \omega)$
  is of class $C\big([0,T]; L^2 \big)$.
}
\item[(ii)]{
  For all $\tau \in [0, T]$, the map
  $\omega \mapsto \overline{V_i}(\tau,\omega)$
  is $(\overline{\mathcal{F}}_{\tau;i})$-measurable.
}
\item[(iii)]{
  We have the inclusion
  \begin{equation}
     \begin{array}{lcl}
        \overline{V_i} \in \mathcal{N}^2\big([0,T];
          (\overline{\mathcal{F}})_{\tau;i} ; H^1 \big) ,
     \end{array}
  \end{equation}
  together with
  \begin{equation}
     \begin{array}{lcl}
        \overline{\mathcal{S}}_{\s;i}(\overline{V_i})
          \in \mathcal{N}^2\big([0,T];
          (\overline{\mathcal{F}})_{\tau;i} ; L^2 \big) .
     \end{array}
  \end{equation}
}
\item[(iv)]{For almost all $\omega \in \Omega$, we have the inclusion
  \begin{align}\label{eq:WIntoL2}
    \mathcal{W}_{\s;i}\big( \oV_i(\cdot, \omega) \big)
      &\in L^1([0,T]; L^2)
      \end{align}
 together with
 \begin{align}     \Upsilon_{\s;i}\big(\oV_i(\cdot, \omega)\big)&\in L^1([0,T]; L^2) .
  \end{align}}
\item[(v)]{
For almost all $\omega \in \Omega$,
the identity
\begin{align}\begin{split}\label{eq:MildEqTimeTrans}
\overline{V_i}(\tau)=&S(\tau)\overline{V_i}(0)
+\int_0^\tau S(\tau-\tau')\mathcal{W}_{\s;i}\big(\overline{V}_i(\tau')\big) \,d\tau'
 +\sigma \int_0^{\tau} S(\tau-\tau')\overline{\mathcal{S}}_{\s;i}
   \big(\overline{V}_i(\tau') \big)d\overline{\b}_{\tau';i}\\[0.2cm]
&+\int_0^{\tau} \p_{\xi} S(\tau-\tau') Q  \Upsilon_{\s;i}\big(\oV_i(\tau')\big) \, d\tau'
 +\int_0^{\tau} \Lambda(\tau-\tau')\Upsilon_{\s;i}\big(\oV_i(\tau') \big) \, d\tau'
\\[0.2cm]
& +\int_0^{\tau} S(\tau-\tau')  P_\xi \Upsilon_{\s;i}\big(\oV_i(\tau')\big) \, d\tau'
 \end{split}
\end{align}
holds for all $ \tau \in [0, T]$.
}
\end{itemize}
\end{prop}

\subsection{Bounds on nonlinearities}
\label{sec:BoundsNl}
In this section we set out to prove Proposition \ref{prp:fnl:bnds}.
In order to be able to write the nonlinearities
in a compact fashion,
we introduce the expression
\begin{equation}
\label{eq:mr:def:j}
\begin{array}{lcl}
\mathcal{J}_{\sigma}(u )
& = &
  \kappa_{\sigma}(u , \psi_{\mathrm{tw}})^{-1}
  \Big[
    f(u )
    +c_{\sigma} \partial_\xi u
    + \sigma^2 b( u, \psi_{\mathrm{tw}})
      \partial_\xi [g(u )]
  \Big]
\\[0.2cm]
\end{array}
\end{equation}
for any $u \in \mathcal{U}_{H^1}$.
This allows us to define
\begin{equation}
\mathcal{Q}_{\s}
     (v)
=
   \mathcal{J}_{\sigma}
   (\Phi_{\sigma} + v  )
   - \mathcal{J}_{\sigma}( \Phi_{\sigma}  )
     + [ \rho \partial_{\xi \xi } - \mathcal{L}_{\mathrm{tw}} ]v
\end{equation}
for any $v \in H^1$,
which is the residual upon linearizing $\mathcal{J}_{\sigma}(\Phi_\s+V)$ around $\Phi_{\sigma}$,
up to $O(\sigma^2)$ corrections. Indeed, we
can borrow the following bound from \cite{Hamster2017}.
\begin{cor}
\label{cor:swv:msigma:bnds}
Consider the setting of Proposition \ref{prp:fnl:bnds}.
There exists
$K > 0$ so that
for any
$0 \le \sigma \le \delta_{\sigma}$
and any
$v\in H^1 $
we have the estimate
\begin{equation}
\begin{array}{lcl}
\norm{\mathcal{Q}_{\s}(v)}_{L^2}
& \le &
K\big[\sigma^2+\nrm{v}_{L^2}\big] \norm{v}_{H^1}
\\[0.2cm]
& & \qquad
+  K\big[1 +(1+\s^2)\norm{v}_{L^2}+ \s^2\norm{v}_{L^2}^2\big]\norm{v}_{H^1}^2 ,
\end{array}
\end{equation}
together with
\begin{equation}
\begin{array}{lcl}
\abs{ \langle \mathcal{Q}_{\s}(v) , \psi_{\mathrm{tw}} \rangle_{L^2} }
& \le &
K \big[ 1 + \norm{v}_{H^1}  \big]
     \norm{v}_{L^2}
     \norm{v}_{L^2}
\\[0.2cm]
& & \qquad
+ K \big[
  \sigma^2 +
  \norm{v}_{L^2}
    \big]
     \norm{v}_{L^2}
\\[0.2cm]
& & \qquad
+ K \sigma^2
    \norm{v}_{H^1} \norm{v}_{L^2}^2
    \norm{v}_{L^2}
\\[0.2cm]
& & \qquad
  + K \sigma^2
    \norm{v}_{L^2}^2
    \norm{v}_{H^1}.
\\[0.2cm]
\end{array}
\end{equation}
\end{cor}
\begin{proof}
Recalling the function $\mathcal{M}$ that was defined in
\cite[Eq. (7.2)]{Hamster2017},
we observe that
\begin{equation}
\mathcal{Q}_{\s}(v) =
\mathcal{M}_{\s;\Phi_{\s},c_{\sigma}}
     (v, 0)  - \mathcal{M}_{\sigma;\Phi_{\sigma} , c_{\sigma}}(0, 0).
\end{equation}
In particular, the desired bounds follow directly
from \cite[Cor. 7.5]{Hamster2017}.
\end{proof}

We now introduce the function
\begin{align}
\mathcal{W}_{\sigma;I,i}(v)
=&\mathcal{Q}_{\s}(v)
 +\phi_{\sigma;i}(v)\Big[
	\mathcal{J}_{\sigma}(\Phi_{\sigma} + v)
	- \mathcal{J}_{\sigma}(\Phi_{\sigma})
	\Big]
\label{eq:Rline:Def}
\end{align}
together with
the notation
\begin{equation}
\begin{array}{lcl}
\mathcal{I}_{\sigma;I,i}(v) & = &
  \Big[\chi_{\mathrm{low}}\big(\langle\partial_\xi[\Phi_{\sigma} + v ] ,
    \psi_{\mathrm{tw}} \rangle_{L^2} \big)\Big]^{-1}
   \langle
     \mathcal{W}_{\sigma;I,i}(v),
     \psi_{\mathrm{tw}}
  \rangle_{L^2}
\\[0.2cm]
& & \qquad
   - \Big[\chi_{\mathrm{low}}\big(\langle\partial_\xi[\Phi_{\sigma} + v ] ,
    \psi_{\mathrm{tw}} \rangle_{L^2} \big)\Big]^{-1}
   \langle
     \Upsilon_{\sigma;i}(v),
     \partial_\xi \psi_{\mathrm{tw}}
  \rangle_{L^2} .
\end{array}
\end{equation}
The following result shows that
these two expressions allow us to split off
the $a_\s$-contribution to $\overline{\mathcal{R}}_{\s;i}$
that is visible in \sref{eq:stt:DefR}.

\begin{lem}
Consider the setting
of Proposition \ref{prp:fnl:bnds}.
Then for any $0 \le \sigma \le \delta_{\sigma}$
and $v \in H^1$,
we have the inclusion $\mathcal{W}_{\s;i}(v)\in L^2$
together with the identity
\begin{equation}
\begin{array}{lcl}
\label{eq:fnl:id:for:r:sigma}
\mathcal{W}_{\s;i}(v)
= \mathcal{W}_{\sigma;I,i}(v)
 - \mathcal{I}_{\sigma;I,i}(v)
  [\Phi_{\sigma}' + v'] .
\\[0.2cm]
\end{array}
\end{equation}
\end{lem}
\begin{proof}
For any $u \in \mathcal{U}_{H^2}$, the definition \sref{eq:mr:def:a}
implies that
\begin{equation}
a_{\sigma}(u , c_{\sigma} , \psi_{\mathrm{tw}})
=
-
    \Big[ \chi_{\mathrm{low}}\big( \langle \partial_\xi u, \psi_{\mathrm{tw}} \rangle_{L^2} \big) \Big]^{-1}
    \langle \kappa_{\sigma}(u, \psi_{\mathrm{tw}}) \big[ \rho \partial_{\xi \xi} u
       + \mathcal{J}_{\sigma}(u) \big] , \psi_{\mathrm{tw}} \rangle_{L^2}.
\end{equation}
The
implicit definition $a_{\sigma}(\Phi_{\sigma}, c_{\sigma}, \psi_{\mathrm{tw}}) = 0$
hence yields
\begin{equation}
\mathcal{J}_{\sigma}(\Phi_{\s} ) = - \rho \Phi_{\s}''.
\end{equation}
For any $v \in H^2$, this allows us to compute
\begin{equation}
\begin{array}{lcl}
\mathcal{Q}_{\s}
     (v)
& = & \mathcal{J}_{\sigma}
   (\Phi_{\sigma} + v )
       +  \rho  [ \Phi_{\sigma}'' + v'']
        - \mathcal{L}_{\mathrm{tw}} v ,
\end{array}
\end{equation}
which gives
\begin{equation}
\begin{array}{lcl}
\mathcal{W}_{\sigma;I,i}(v)
+ \partial_{\xi} \Upsilon_{\sigma;i}(v)
& = & [\kappa_{\sigma;i}(\Phi_{\sigma} + v , \psi_{\mathrm{tw}} ) ]^{-1}
    \kappa_{\sigma}(\Phi_{\sigma} + v , \psi_{\mathrm{tw}} )
  \big[ \rho [\Phi_{\sigma}'' + v''] + \mathcal{J}_{\sigma}
     (\Phi_{\sigma} + v )
  \big]
\\[0.2cm]
& & \qquad
   - \mathcal{L}_{\mathrm{tw}} v .
\end{array}
\end{equation}
Using the fact that
$\mathcal{L}^*_{\mathrm{tw}} \psi_{\mathrm{tw}} = 0$, we now
readily verify that for $v \in H^2$
we have
\begin{equation}
\mathcal{I}_{\sigma;I,i}(v)
=
[\kappa_{\sigma;i}( \Phi_{\sigma} + v , \psi_{\mathrm{tw}} )]^{-1}
 a_{\sigma}(\Phi_{\sigma} + v , \psi_{\mathrm{tw}}) .
\end{equation}
The result hence follows by
rewriting the definition \sref{eq:stt:DefR}
in the form
\begin{equation}
\begin{array}{lcl}
\mathcal{R}_{\sigma}(v)
& = & \kappa_{\sigma}( \Phi_{\sigma} + v , \psi_{\mathrm{tw}} )
  \Big[
     \rho \partial_{\xi \xi} [ \Phi_{\sigma} + v]
     + \mathcal{J}_{\sigma}( \Phi_{\sigma} + v )
  \Big]
\\[0.2cm]
& & \qquad
  + a_{\sigma}(\Phi_{\sigma} + v , c_{\sigma},  \psi_{\mathrm{tw}})
    [ \Phi_{\sigma}' + v']
\end{array}
\end{equation}
and substituting this into the definition \sref{eq:md:overline:r:sigma} of $\overline{\mathcal{R}}_{\s;i}$.
\end{proof}

In order to obtain the estimates in Proposition \ref{prp:fnl:bnds}
it hence suffices to obtain bounds for $\phi_i$,
$\mathcal{W}_{\s;I,i}$
and $\mathcal{I}_{\s;I,i}$. This can be done in a direct fashion.

\begin{lem}
\label{lem:phi:bound}
Assume that (HDt) and (HSt) are satisfied.
Then there exists a constant $K_{\phi} > 0$
so that
\begin{align}
|\phi_i(v)|\leq \s^2K_\phi
\end{align}
holds for any $v \in L^2$ and $0 \le \s \le \delta_{\s}$.
\end{lem}
\begin{proof}
For any $x,y\geq 0$ we have the inequality
\begin{align}
\left|\frac{1+\frac{1}{2\rho_j}x}{1+\frac{1}{2\rho_i}x}
   -\frac{1+\frac{1}{2\rho_i}y}{1+\frac{1}{2\rho_i}y}\right|=\frac{1}{4\rho_i\rho_j}\frac{|x-y|}{(1+\frac{1}{2\rho_i}x)(1+\frac{1}{2\rho_i}y)}
   \leq \frac{1}{4\rho_i\rho_j}|x-y| .
\end{align}
Applying these bounds with $y=0$, we obtain
\begin{align}
|\phi^j_i(v)|\leq \frac{\s^2}{4\rho_i\rho_j}|b(\Phi_\s+v)|^2\leq
  \frac{\s^2}{4\rho_{\mathrm{min}}^2}K_b^2,
\end{align}
where the last bound on $b$ follows from Lemma 3.6 in \cite{Hamster2017}.
The result now readily follows.
\end{proof}

\begin{lem}
\label{lem:fnl:rsi}
Consider the setting
of Proposition \ref{prp:fnl:bnds}.
Then there exists $K > 0$
so that for any $v \in H^1$
and $0 \le \sigma \le \delta_{\sigma}$
we have the bound
\begin{equation}
\label{eq:fnl:rsi:glb}
\begin{array}{lcl}
\norm{\mathcal{W}_{\s;I,i}(v)}_{L^2}
& \le &
K \sigma^2 \norm{v}_{H^1}
+ K \norm{v}_{H^1}^2
  \big[
    1 + \norm{v}_{L^2} + \sigma^2 \norm{v}_{L^2}^2
  \big] ,
\\[0.2cm]
\end{array}
\end{equation}
together with
\begin{equation}
\label{eq:fnl:rsi:glb:ip}
\begin{array}{lcl}
\abs{ \mathcal{I}_{\sigma;I,i}(v) }
& \le &
K  \norm{v}_{L^2}
  \big[ \sigma^2 + \norm{v}_{L^2} \big]
+ K \norm{v}_{H^1}
  \big[
     \norm{v}_{L^2}^2 + \sigma^2 \norm{v}_{L^2}^3
  \big]    .
\end{array}
\end{equation}
\end{lem}
\begin{proof}
Note first that we can write $\mathcal{W}_{\sigma;I,i}(v)$ as
\begin{align}
\mathcal{W}_{\sigma;I,i}(v)
=&\mathcal{Q}_{\s}(v)
 +\phi_{\sigma;i}(v)\Big[
	\mathcal{Q}_{\s}(v)+(\mathcal{L}_{\mathrm{tw}}-\rho\p_{\xi\xi}) v
	\Big]
\end{align}
and hence
\begin{align}
\nrm{\mathcal{W}_{\sigma;I,i}(v)}_{L^2}
\leq &\nrm{\mathcal{Q}_{\s}(v)}_{L^2}
 +|\phi_{\sigma;i}(v)|\Big[
	\nrm{\mathcal{Q}_{\s}(v)}_{L^2}+\nrm{(\mathcal{L}_{\mathrm{tw}}-\rho\p_{\xi\xi})v}_{L^2}
	\Big] .
\end{align}
The definition of $\mathcal{L}_{\mathrm{tw}}$ implies that there exists $C_1 > 0$
for which
\begin{equation}
\norm{[\L_{\mathrm{tw}}-\rho\p_{\xi\xi}]v}_{L^2}\leq  C_1\nrm{v}_{H^1}
\end{equation}
holds. The desired bound
hence follows from
Corollary \ref{cor:swv:msigma:bnds} and Lemma \ref{lem:phi:bound}.

Turning to the second estimate,
we note that there is a positive constant $C_2$ for which we have
\begin{align}
\abs{ \mathcal{I}_{\sigma;I,i} (v) } \leq
  C_2\big[ \norm{\mathcal{W}_{\s;I,i}(v)}_{L^2}+\norm{\Upsilon_{\s;i}(v)}_{L^2}\big] .
\end{align}
We can hence again apply Corollary \ref{cor:swv:msigma:bnds} and Lemma \ref{lem:phi:bound},
which yields expressions that can all be absorbed into \sref{eq:fnl:rsi:glb:ip}.
\end{proof}

\begin{proof}[Proof of Proposition \ref{prp:fnl:bnds}]
To obtain \sref{eq:fnl:rs:glb},
we use \sref{eq:fnl:id:for:r:sigma}
together with Lemma \ref{lem:fnl:rsi}
to compute
\begin{equation}
\begin{array}{lcl}
\norm{\mathcal{W}_{\s;i}(v)}_{L^2}
& \le &
\norm{\mathcal{W}_{\s;i}}_{L^2}
+ C_1 \abs{\mathcal{I}_{\s;I,i}(v)}
    \big[ 1 + \norm{v}_{H^1} \big]
\\[0.2cm]
& \le &
C_2 \sigma^2 \norm{v}_{H^1}
+ C_2 \norm{v}_{H^1}^2
  \big[
    1 + \norm{v}_{L^2} + \sigma^2 \norm{v}_{L^2}^2
  \big]
\\[0.2cm]
& & \qquad
+C_2 \norm{v}_{L^2} \big[ \sigma^2 + \norm{v}_{L^2} \big]
    \big[ 1 + \norm{v}_{H^1} \big]
\\[0.2cm]
& & \qquad
+ C_2 \norm{v}_{H^1}
  \big[
     \norm{v}_{L^2}^2 + \sigma^2 \norm{v}_{L^2}^3
  \big]
     \big[ 1 + \norm{v}_{H^1} \big]
\\[0.2cm]
\end{array}
\end{equation}
for some constants $C_1 > 0$ and $C_2 > 0$.
These terms can all be absorbed into
\sref{eq:fnl:rs:glb}.
The bound \sref{eq:mld:bnd:upsilon}
follows from Lemma \ref{lem:phi:bound} and (HDt),
while
 (ii) and (iii) follow directly
from \cite[Prop 8.1]{Hamster2017}.
\end{proof}

\subsection{Mild formulation}
\label{sec:MildForm}
In this section we establish Proposition \ref{prp:MildTransSys}.
We note that items (i)-(iv) follow directly from Propositions 5.1 and 6.3 in
\cite{Hamster2017}, so we focus
here on the integral identity \sref{eq:MildEqTimeTrans}.
We first obtain this identity in a weak sense,
bypassing the need to interpret the term
involving $\Upsilon_{\sigma;i}$ in a special fashion.
We note that
$S^*(t)$ is the adjoint operator of $S(t)$,
which coincides with the semigroup generated by $\L_{\mathrm{tw}}^*$.


\begin{lem}\label{lem:MildForm1} Consider the setting of
Proposition \ref{prp:MildTransSys}
and pick any $\eta \in H^3$.
Then for almost all $\omega \in \Omega$
the identity
\begin{align}\begin{split}
\ip{\oV_i(\tau),\eta}_{L^2}=&
\ip{S(\tau)\oV_i(0)
+\int_0^{\tau} S(\tau-\tau')\mathcal{W}_{\s;i}\big(\oV_i(\tau')\big)\, d\tau'
+ \sigma \int_0^{\tau} S(\tau-\tau')\overline{\mathcal{S}}_{\s;i}\big(\oV_i(\tau')\big)d\overline{\b}_{\tau';i},\eta}_{L^2}
\\[0.2cm]
&+\int_0^{\tau} \ip{ \p_{\xi} \Upsilon_{\s;i}\big(\oV_i(\tau')\big), S^*(\tau-\tau' )\eta}_{H^{-1} ; H^1} \, d \tau'
\end{split}
\end{align}
holds for any $\tau \in [0, T]$.
\end{lem}
\begin{proof}
Pick any $\tau \in [0, T]$.
Since $\overline{V_i}\in \mathcal{N}^2([0,T];(\overline{\mathcal{F}}_t);H^1)$
is a weak solution to \sref{eq:mild:weak:formul},
the identity
\begin{equation}
\begin{array}{lcl}
\oV_i(\tau) &=&
  \oV_i(0)
  +\int_0^{\tau} \big[ \mathcal{L}_{\mathrm{tw}} \overline{V}_i(\tau')
    + \overline{\mathcal{R}}_{\sigma;i}\big( \overline{V}_i(\tau') \big) \big]  \, d \tau'
\\[0.2cm]
& & \qquad
  + \sigma \int_0^\tau \overline{S}_{\sigma;i} \big( \oV_i(\tau') \big) \, d\overline{\b}_{\tau';i}
\end{array}
\end{equation}
holds in $H^{-1}$; see \cite[Prop. 6.3]{Hamster2017}.
We note that these integrals are well defined by items (i)-(iv) of Proposition \ref{prp:MildTransSys}.

Following the proof of \cite[Prop 2.10]{KatjaClaudia},
we pick $\eta \in H^3$ and define the function
\begin{align}
\zeta(\tau')=S^*(\tau-\tau')
  \eta
\end{align}
on the interval $[0,\tau]$.
Noting that $\zeta\in C^1([0,\tau], H^1 )$,
we may define the functional
$\phi: [0, \tau] \times H^{-1} \to \Real$
that acts as
\begin{equation}
\phi(\tau' , v) = \langle v , \zeta(\tau') \rangle_{H^{-1} ; H^1},
\end{equation}
which is $C^1$-smooth in the first variable and linear in the second variable.
Applying a standard It\^o formula such as
\cite[Thm. 1]{DaPratomild} (with $S = I$) yields
\begin{equation}
\begin{array}{lcl}
\phi\big(\tau, \overline{V}_i(\tau) \big)
 & = &
   \phi\big(0, \overline{V}_i(0) \big)
\\[0.2cm]
 & & \qquad
    + \int_0^{\tau} \langle  \overline{V}_i(\tau'), \zeta'(\tau') \rangle_{H^{-1}; H^1} \, d \tau'
   + \int_0^\tau \langle \mathcal{L}_{\mathrm{tw}} \overline{V}_i(\tau')   , \zeta(\tau')
     \rangle_{H^{-1} ; H^1}\,d \tau'
 \\[0.2cm]
 & & \qquad
   + \int_0^\tau \langle \overline{\mathcal{R}}_{\sigma;i}
     \big( \overline{V}_i(\tau')\big)   , \zeta(\tau')
       \rangle_{H^{-1} ; H^1} \, d \tau'
\\[0.2cm]
 & & \qquad
   + \sigma \int_0^\tau \langle \overline{\mathcal{S}}_{\sigma;i}
     \big( \overline{V}_i(\tau')\big)   , \zeta(\tau') \rangle_{L^2}
      \, d \overline{\beta}_{\tau';i} .
\end{array}
\end{equation}
Since $\zeta'(t) = - \mathcal{L}^*_{\mathrm{tw}} \zeta(t)$,
the second line in the expression above disappears.
Using the identities
\begin{equation}
\begin{array}{lcl}
\phi\big(\tau, \overline{V}_i(\tau) \big)
 & = & \langle \oV_i(\tau), 
    \eta \rangle_{L^2} ,
\\[0.2cm]
\phi\big(0, \overline{V}_i(0) \big)
 & = & \langle \overline{V}_i(0) , S^*(\tau) 
   \eta \rangle_{L^2}
\\[0.2cm]
 & = & \langle S(\tau) \overline{V}_i(0) ,
         \eta \rangle_{L^2}
\end{array}
\end{equation}
we hence obtain
\begin{equation}
\begin{array}{lcl}
\langle \oV_i(\tau), 
  \eta \rangle_{L^2}
& = &
\ip{S(\tau)\oV_i(0), 
  \eta}_{L^2}
\\[0.2cm]
& & \qquad
  +\int_0^{\tau} \ip{ S(\tau-\tau')\mathcal{W}_{\s;i}(\oV_i(\tau') \big) \, d \tau',
          \eta }_{L^2} \, d \tau'
\\[0.2cm]
& & \qquad
  +\int_0^{\tau} \ip{\p_{\xi} \Upsilon_{\s;i}\big(\oV_i(\tau') \big),
    S^*(t-s) 
    \eta}_{H^{-1}; H^1} \, d\tau'
\\[0.2cm]
& & \qquad
  +\sigma \int_0^{\tau} \ip{ S(\tau- \tau')\overline{S}_{\s;i}\big(\oV_i(\tau') \big)
       ,  
         \eta}_{L^2}  \,  d \overline{\b}_{\tau';i} ,
\end{array}
\end{equation}
as desired.
\end{proof}

\begin{lem}
\label{lem:mld:approxInH1}
Pick $v \in L^2$ together with $\eta \in H^1$ and $t > 0$.
Then we have the identity
\begin{equation}
\label{eq:mld:transfer:derivative}
\langle \partial_{\xi} v , S^*(t) \eta \rangle_{H^{-1} ; H^1}
 = \langle \partial_{\xi} S(t) Q v
      +  \Lambda(t) v
      +  S(t) P_\xi v , \eta \rangle_{L^2} .
\end{equation}
\end{lem}
\begin{proof}
For $v \in H^1$, this identity follows directly
from \sref{eq:spl:eqn:with:comm}. For fixed $\eta$ and $ t> 0$,
both sides
of \sref{eq:mld:transfer:derivative} can be interpreted
as bounded linear functions on $L^2$ by
Proposition \ref{prp:MainPropSplit}. In particular,
the result can be obtained by approximating $v \in L^2$ by $H^1$-functions.
\end{proof}

\begin{proof}[Proof of Proposition \ref{prp:MildTransSys}]
As mentioned above, items (i)-(iv) follow directly from
Propositions 5.1 and 6.3 in \cite{Hamster2017}.
Item (v) follows from
Lemmas \ref{lem:MildForm1} and \ref{lem:mld:approxInH1},
using the density of $H^3$ in $H^1$ and the fact that $H^{-1}$ is separable.
\end{proof}


\section{Nonlinear stability of mild solutions}
\label{sec:nls}
In this section we
prove Theorem \ref{thm:mr:orbital:stb}, which
provides an orbital  stability result
for the stochastic wave $(\Phi_{\sigma}, c_{\sigma})$.
In particular,
for any $\e > 0$, $T > 0$
and $\eta > 0$ we recall the
notation
\begin{equation}
\begin{array}{lcl}
 N_{\e}(t)=
 \nrm{V(t)}_{L^2}^2
 +\int_0^te^{-\e(t-s)}\nrm{V(s)}^2_{H^1}ds
\end{array}
\end{equation}
and introduce the
$(\mathcal{F}_t)$-stopping time
\begin{equation}
t_{\mathrm{st}}(T,\e,\eta)
 = \inf\Big\{0 \leq t < T:
     N_{\e}(t)
     > \eta
  \Big\} ,
\end{equation}
writing $t_{\mathrm{st}}(T,\e,\eta) = T$
if the set is empty.
We derive a number of technical regularity estimates
in \S\ref{sec:nls:reg:ests}
that allows us to
exploit the integral identity \sref{eq:MildEqTimeTrans}
to  bound the expectation
of $\sup_{0\leq t\leq t_{\mathrm{st}}(T,\e,\eta) }N_{\e}(t)$
in terms of itself,
the noise-strength $\sigma$
and the size of the initial condition $V(0)$.
This leads to the following
bound for this expectation.
\begin{prop}
\label{prp:nls:general}
Assume that (HDt),
(HSt) and (HTw) are satisfied.
Pick a constant $0 < \e < \beta$,
together with two sufficiently small
constants $\delta_{\eta} > 0$
and $\delta_{\sigma} > 0$.
Then there exists a constant
$K > 0$ so that for any $T > 0$,
any $0 < \eta \le \delta_{\eta}$
and any $0 \le \sigma \le \delta_{\sigma}T^{-1/2}$
we have the bound
\begin{equation}
\label{eq:nls:prp:general:estimate}
\begin{array}{lcl}
E[\sup_{0\leq t\leq t_{\mathrm{st}}(T,\e,\eta) }N_{\e}(t)]
&\leq &
   K \Big[ \nrm{V(0)}^2_{H^1}+ \sigma^2 T \Big] .
\end{array}
\end{equation}
\end{prop}

Exploiting the technique used in
Stannat \cite{Stannat},
this bound can be turned into an estimate
concerning the probability
\begin{equation}
p_{\e}(T,\eta) = P\Big(
 \sup_{0 \leq t \leq T} \big[N_{\e}(t)\big]
 > \eta
\Big).
\end{equation}
This allows our
main stability result to be established
in a straightforward fashion.

\begin{proof}[Proof of Theorem \ref{thm:mr:orbital:stb}]
Upon computing
\begin{equation}
\begin{array}{lcl}
\eta p_{\e}(T,\eta)
& = & \eta P\big( t_{\mathrm{st}}(T,\e,\eta) < T\big)
\\[0.2cm]
& = & E \Big[
      \mathbf{1}_{t_{\mathrm{st}}(T,\e,\eta) < T}
      N_{\e}\big(t_{\mathrm{st}}(T,\e,\eta)\big)
    \Big]
\\[0.2cm]
& \le & E [N_{\e}\big(t_{\mathrm{st}}(T,\e,\eta)  \big)]\\[0.2cm]
& \leq & E [\sup_{0\leq t\leq t_{\mathrm{st}}(T,\e,\eta) }N_{\e}(t)],
\end{array}
\end{equation}
the result follows from Proposition \ref{eq:nls:prp:general:estimate}.

\end{proof}

\subsection{Setup}
\label{sec:nls:reg:ests}

In this subsection we
establish Proposition \ref{prp:nls:general}
by estimating each of the terms
featuring in \sref{eq:MildEqTimeTrans}.
In contrast to the situation in \cite{Hamster2017}
we cannot estimate $N_\e(t)$ directly
because the integral involving $\partial_{\xi} S(t-s)$
applied to $\Upsilon_{\sigma;i}\big(\overline{V}_i(s) \big)$
presents short-time regularity issues.
Instead, we will obtain separate estimates for each of the components
$N_\e^{i}(t)$, which are given by
\begin{equation}
\begin{array}{lcl}
N_\e^{i}(t)
  & = & \nrm{ V^i(t) }_{L^2}^2 +\int_0^t e^{-\e(t-s)}
       \nrm{V^i(s)}_{H^1}^2 \, ds .
\end{array}
\end{equation}
Indeed, the definitions \sref{eq:DefPhi} and \sref{eq:DefUpsilon}
imply that the $i$-th component of $\Upsilon_{\s;i}$ vanishes,
which allows us to replace the problematic $\partial_{\xi} S(t - s)$ term
with its off-diagonal components $\partial_{\xi} S_{\mathrm{od}}(t - s)$.
More precisely, for $\tau' \ge \tau - 1$ when computing short time bounds,
we will use
\begin{equation}
\begin{array}{lcl}
\Big[\p_{\xi} S(\tau - \tau') Q \Upsilon_{\s;i}\big(\oV_i(\tau') \big) \Big]^i
 & = &
 \Big[\p_{\xi} S(\tau - \tau') (I - P) \Upsilon_{\s;i}\big(\oV_i(\tau') \big) \Big]^i
\\[0.2cm]
& = &
  \Big[\p_{\xi} S_{\mathrm{od}}(\tau - \tau')  \Upsilon_{\s;i}\big(\oV_i(\tau') \big)
   - \p_{\xi} S(\tau-  \tau')P \Upsilon_{\s;i}\big(\oV_i(\tau') \big)
   \Big]^i .
\end{array}
\end{equation}
This will allow us to bound $N_\e^{i}(t)$ in terms of $N_\e(t)$.


In order to streamline our computations,
we now introduce some notation that will help us to stay as close
as possible to the framework developed in \cite{Hamster2017}.
First of all, we impose the splittings
\begin{equation}
\label{eq:reg:defns:n:i:ii}
\begin{array}{lcl}
N_{\e,I}(t)
  & = & \nrm{ V(t) }_{L^2}^2 ,
\\[0.2cm]
N_{\e,II}(t)
 & = &
   \int_0^t e^{-\e(t-s)}
       \nrm{V(s)}_{H^1}^2 \, ds,
\end{array}
\end{equation}
together with
\begin{equation}
\label{eq:reg:defns:n:i:ii:comp}
\begin{array}{lcl}
N^i_{\e;I}(t)
  & = & \nrm{ V^i(t) }_{L^2}^2 \\[0.2cm]
  & = & \nrm{ \oV_i^i\big(\tau_{i}(t)\big) }_{L^2}^2 ,
\\[0.2cm]
N^i_{\e;II}(t)
 & = &
   \int_0^t e^{-\e(t-s)}
       \nrm{V^i(s)}_{H^1}^2 \, ds
   \\[0.2cm]
 & = &   \int_0^t e^{-\e(t-s)}
       \nrm{\oV_i^i\big(\tau_{i}(s)\big)}_{H^1}^2 \, ds.
\end{array}
\end{equation}
In addition, we split $\mathcal{W}_{\s;i}$ into a linear and nonlinear part
as
\begin{align}
\mathcal{W}_{\s;i}(v)=\s^2F_{\mathrm{lin}}(v)+F_{\mathrm{nl}}(v)
\end{align}
and we isolate the constant term in $\overline{\mathcal{S}}_{\s;i}$
by writing
\begin{align}
\overline{\mathcal{S}}_{\s;i}(v)=B_{\mathrm{cn}}+B_{\mathrm{lin}}(v) .
\end{align}
Proposition \ref{prp:fnl:bnds} implies 
that these functions satisfy the bounds
\begin{align}\begin{split}
    \nrm{F_{\mathrm{lin}}(v)}_{L^2}&\leq K_{\mathrm{F;lin}}\nrm{v}_{H^1},\\
    \nrm{F_{\mathrm{nl}}(v)}_{L^2}&\leq K_{\mathrm{F;nl}}\nrm{v}^2_{H^1}(1+\nrm{v}^3_{L^2}),\\
    \nrm{B_{\mathrm{cn}}}_{L^2}&< \infty,\\
    \nrm{B_{\mathrm{lin}}(v)}_{L^2}&\leq K_{\mathrm{B;lin}}\nrm{v}_{H^1}\end{split}
\end{align}
for appropriate constants $K_{\mathrm{F;lin}} > 0$, $K_{\mathrm{F;nl}} > 0$
and $K_{\mathrm{B;lin}} > 0$.
In particular, they satisfy assumption (hFB) in \cite{Hamster2017}, which gives us the
opportunity to apply some of the ideas in \cite[{\S}9]{Hamster2017}.

For convenience we will write from now on $t_{\mathrm{st}}$ for $t_{\mathrm{st}}(T,\e,\eta)$.
In order to understand $N^i_{\e;I}$,
we introduce the expression
\begin{equation}
\begin{array}{lcl}
\mathcal{E}_{0}(t) &  = & S\big(\tau_{i}(t)\big) Q V(0),
\\[0.2cm]
\end{array}
\end{equation}
together with the long-term
integrals
\begin{equation}
\label{eq:reg:LTI}
\begin{array}{lcl}
\mathcal{E}^{\mathrm{lt}}_{F;\mathrm{lin}}(t) &  = &
  \int_0^{\tau_{i}(t)-1} S(\tau_{i}(t)- \tau) Q F_{\mathrm{lin}}\big(\overline{V}_i(\tau)\big)\mathbf{1}_{\tau < \tau_i(t_{\mathrm{st}})}d\tau ,
\\[0.2cm]
\mathcal{E}^{\mathrm{lt}}_{F;\mathrm{nl}}(t) &  = &
  \int_0^{\tau_{i}(t)-1} S(\tau_{i}(t)  - \tau) Q F_{\mathrm{nl}}\big(\overline{V}_i(\tau)\big)\mathbf{1}_{\tau < \tau_i(t_{\mathrm{st}})}d\tau ,
\\[0.2cm]
\mathcal{E}^{\mathrm{lt}}_{B;\mathrm{lin}}(t)
  & = &
    \int_0^{\tau_{i}(t)-1} S(\tau_{i}(t) -\tau) Q B_{\mathrm{lin}}\big(\overline{V}_i(\tau) \big) \mathbf{1}_{\tau < \tau_i(t_{\mathrm{st}})} d \beta_\tau,
\\[0.2cm]
\mathcal{E}^{\mathrm{lt}}_{B;\mathrm{cn}}(t)
  & = &
    \int_0^{\tau_{i}(t)-1} S(\tau_{i}(t) -\tau)Q B_{\mathrm{cn}} \mathbf{1}_{\tau < \tau_i(t_{\mathrm{st}})} d \beta_\tau ,
\\[0.2cm]
\mathcal{E}_{\mathrm{so}}^{\mathrm{lt}}(t)
 &= &\int_0^{\tau_{i}(t) - 1}\p_{\xi}S(\tau_{i}(t)-\tau) Q \Upsilon_{\s;i}\big(\oV_i(\tau)\big)
+\Lambda
  ({\tau_{i}(t)}-\tau)\Upsilon_{\s;i}\big(\oV_i(\tau)\big) \mathbf{1}_{\tau < \tau_i(t_{\mathrm{st}})}d\tau ,
\end{array}
\end{equation}
the short-term integrals
\begin{equation}
\begin{array}{lcl}
\mathcal{E}^{\mathrm{sh}}_{F;\mathrm{lin}}(t) &  = &
  \int_{\tau_{i}(t)-1}^{\tau_{i}(t)} S(\tau_{i}(t) - \tau) Q F_{\mathrm{lin}}\big(\overline{V}_i(\tau)\big)\mathbf{1}_{\tau < \tau_i(t_{\mathrm{st}})}d\tau ,
\\[0.2cm]
\mathcal{E}^{\mathrm{sh}}_{F;\mathrm{nl}}(t) &  = &
  \int_{\tau_{i}(t)-1}^{\tau_{i}(t)} S(\tau_{i}(t)  - \tau) Q F_{\mathrm{nl}}\big(\overline{V}_i(\tau)\big)\mathbf{1}_{\tau < \tau_i(t_{\mathrm{st}})}d\tau ,
\\[0.2cm]
\mathcal{E}^{\mathrm{sh}}_{B;\mathrm{lin}}(t)
  & = &
    \int_{\tau_{i}(t)-1}^{\tau_{i}(t)} S(\tau_{i}(t) -\tau) Q B_{\mathrm{lin}}\big(\overline{V}_i(\tau) \big) \mathbf{1}_{\tau < \tau_i(t_{\mathrm{st}})} d \beta_\tau,
\\[0.2cm]
\mathcal{E}^{\mathrm{sh}}_{B;\mathrm{cn}}(t)
  & = &
    \int_{\tau_{i}(t)-1}^{\tau_{i}(t)} S(\tau_{i}(t) -\tau)Q B_{\mathrm{cn}} \mathbf{1}_{\tau < \tau_i(t_{\mathrm{st}})} d \beta_\tau ,
\\[0.2cm]
\end{array}
\end{equation}
and finally the split second-order integrals
\begin{equation}
\begin{array}{lcl}
\mathcal{E}^{\mathrm{sh}}_{\mathrm{so};A}(t)
 & = &- \int_{\tau_i(t) - 1}^{\tau_{i}(t)}
   \p_{\xi} S(\tau_{i}(t)-\tau) P \Upsilon_{\s;i}\big(\oV_i(\tau)\big)
 \mathbf{1}_{\tau < \tau_i(t_{\mathrm{st}})}d\tau,
\\[0.2cm]
  \mathcal{E}^{\mathrm{sh}}_{\mathrm{so};B}(t)
 & = &\int_{\tau_i(t) - 1}^{\tau_{i}(t)}
\Lambda
  ({\tau_{i}(t)}-\tau)\Upsilon_{\s;i}\big(\oV_i(\tau)\big)
  \mathbf{1}_{\tau < \tau_i(t_{\mathrm{st}})}d\tau,
\\[0.2cm]
\mathcal{E}^{\mathrm{sh}}_{\mathrm{so};C}(t)
 & = &\int_{\tau_i(t) - 1}^{\tau_{i}(t)}\p_{\xi} S_{\mathrm{od}}(\tau_{i}(t)-\tau)
    \Upsilon_{\s;i}\big(\oV_i(\tau)\big) \, \mathbf{1}_{\tau < \tau_i(t_{\mathrm{st}})} d\tau.
\end{array}
\end{equation}
%
%
Here we use the convention that integrands are set to zero for $\tau < 0$.
Note that  integration variables in the original time are represented by $s$,
while integration variables in the rescaled time are denoted by $\tau$.
For $\eta > 0$ sufficiently small, our stopping time ensures that
the identities \sref{eq:MildNonProjZero} hold. This implies that
we may assume
\begin{equation}
P_\xi \Upsilon_{\s;i} \big(\oV_i(\tau)\big) + P \mathcal{W}_{\s;i}\big(\oV_i(\tau)\big)  = 0.
\end{equation}
This explains why there is a $Q$ in the first two lines of \sref{eq:reg:LTI}, as their $P$-counterparts are canceled against the
$S(\tau_i(t)-\tau)P_\xi$ term that is present in \sref{eq:MildEqTimeTrans} but
absent from \sref{eq:reg:LTI}.

For convenience, we also write
\begin{equation}
\mathcal{E}_{F;\#}(t) = \mathcal{E}^{\mathrm{lt}}_{F;\#}(t) + \mathcal{E}^{\mathrm{sh}}_{F;\#}(t)
\end{equation}
for $\# \in \{\mathrm{lin},\mathrm{nl}\}$,
together with
\begin{equation}
\mathcal{E}_{B;\#}(t) = \mathcal{E}^{\mathrm{lt}}_{B;\#}(t) + \mathcal{E}^{\mathrm{sh}}_{B;\#}(t)
\end{equation}
for $\# \in \{\mathrm{lin},\mathrm{cn}\}$
and finally
\begin{equation}
\mathcal{E}^{\mathrm{sh}}_{\mathrm{so}}(t) =
 \mathcal{E}_{\mathrm{so};A}^{\mathrm{sh}}(t)
   +  \mathcal{E}_{\mathrm{so};B}^{\mathrm{sh}}(t)
   +  \mathcal{E}_{\mathrm{so};C}^{\mathrm{sh}}(t)
\end{equation}
for the short-term second-order terms.

Turning to the terms
that are relevant
for evaluating $N^i_{\e;II}$,
we introduce the expression
\begin{equation}
\begin{array}{lcl}
\mathcal{I}_{\e,\delta;0}(t) &  = &
  \int_0^t e^{-\e ( t- s) }
    \norm{ S(\delta) \mathcal{E}_0(s) }_{H^1}^2 \, ds ,
\\[0.2cm]
\end{array}
\end{equation}
together with
\begin{equation}
\begin{array}{lcl}
\mathcal{I}^{\#}_{\e,\delta;F;\mathrm{lin}}(t) &  = &
  \int_0^t e^{-\e ( t- s) }
    \norm{ S(\delta) \mathcal{E}^{\#}_{F;\mathrm{lin}}(s) }_{H^1}^2 \, ds ,
\\[0.2cm]
\mathcal{I}^{\#}_{\e,\delta;F;\mathrm{nl} }(t) &  = &
  \int_0^t e^{-\e ( t- s) }
    \norm{ S(\delta) \mathcal{E}^{\#}_{F;\mathrm{nl}}(s) }_{H^1}^2 \, ds ,
\\[0.2cm]
\mathcal{I}^{\#}_{\e,\delta;B;\mathrm{lin}}(t)
  & = &
    \int_0^t e^{-\e ( t- s) }
    \norm{ S(\delta) \mathcal{E}^{\#}_{B;\mathrm{lin}}(s) }_{H^1}^2 \, ds  ,
\\[0.2cm]
 \mathcal{I}^{\#}_{\e,\delta;B;\mathrm{cn}}(t)
  & = &
    \int_0^t e^{-\e ( t- s) }
    \norm{ S(\delta) \mathcal{E}^{\#}_{B;\mathrm{cn}}(s) }_{H^1}^2 \, ds ,
\\[0.2cm]
 \mathcal{I}^{\#}_{\e,\delta;\mathrm{so}}(t)
  & = &
    \int_0^t e^{-\e ( t- s) }
    \norm{ S(\delta) \mathcal{E}^{\#}_{\mathrm{so}}(s) }_{H^1}^2 \, ds
\\[0.2cm]
\end{array}
\end{equation}
for $\# \in \{\mathrm{lt}, \mathrm{sh}\}$. The extra $S(\delta)$ factor will be used
to ensure that all the integrals we encounter are well-defined.
We emphasize that all our estimates are uniform in $0 < \delta < 1$,
allowing us to take $\delta \downarrow 0$.
The estimates concerning $\mathcal{I}^{\mathrm{sh}}_{\e,\delta;F;\mathrm{nl}}$
and $\mathcal{I}^{\mathrm{sh}}_{\e,\delta;B;\mathrm{lin}}$
in Lemmas \ref{lem:nls:f:nl:st} and \ref{lem:nls:b:lin:st}
are particularly delicate in this respect, as a direct
application of the bounds in Lemma \ref{lem:nls:sem:group:decay}
would result in expressions that diverge as $\delta \downarrow 0$.

The main difference between the approach here and
the computations in \cite[{\S}9]{Hamster2017} is that we need to keep track
of several time transforms simultaneously, which forces us
to use the original time $t$ in the definitions
\sref{eq:reg:defns:n:i:ii}-\sref{eq:reg:defns:n:i:ii:comp}.
The following result plays a key role in this respect, as it shows that
decay rates in the $\tau$-variable are stronger than decay rates in the original time.

\begin{lem}\label{lem:TauEst}
Assume that (HDt), (HSt) and (HTw) are satisfied
and pick $0 \le \sigma \le \delta_{\sigma}$.
Then for any pair $t>s \ge 0$ we have the inequality
\begin{align}
\label{eq:nls:TauEst:tau:est:i}
\tau_i(t)-\tau_i(s)\geq t-s,
\end{align}
while for any $s \ge t_i(1)$
we have
\begin{align}
t_i(\tau_i(s)-1)\geq s-1 .
\end{align}
\end{lem}
\begin{proof}
The first inequality can be verified
by using \sref{eq:prlm:kappa:glb:ests}
to compute
\begin{align}
\begin{split}
\tau_i(t)-\tau_i(s)&=\int_s^t\k_{\s;i}(\Phi_\s+V(s'),\psi_{\mathrm{tw}})ds'\\
&\geq (t-s)\min_{s\leq s'\leq t}\k_{\s;i}(\Phi_\s+V(s'),\psi_{\mathrm{tw}})\\
&\geq t-s .
\end{split}
\end{align}
To obtain the second inequality, we write $\tilde{s} = t_i(1) \le 1$
and compute
\begin{equation}
\tau_i(s) -1 = \tau_i(s) - \tau_i(\tilde{s})
\ge s - \tilde{s} \ge s - 1.
\end{equation}
\end{proof}

We now set out to bound all the terms appearing in $N_\e^{i}(t)$.
Following \cite{Hamster2017}, we first study the deterministic integrals and afterwards use $H^\infty$-calculus to bound the stochastic integrals.

\subsection{Deterministic Regularity Estimates}
First, we collect some results from \cite[\S 9.2]{Hamster2017} that are easily adapted to the present situation.

\begin{lem}\label{lem:ColPrevRes}
Fix $T > 0$, assume that (HDt), (HSt) and (HTw) all hold
and pick a constant $0 < \e < \beta$.
Then for any $\eta>0$,
any $0 \leq \delta < 1$
and any $0 \le t \le t_{\mathrm{st}}$,
we have the bounds
\begin{equation}
\begin{array}{lcl}
\norm{\mathcal{E}_0(t)}_{L^2}^2
 & \le &  M^2 e^{-2\b t} \norm{V(0)}^2_{L^2},
 \\[0.2cm]
  \norm{\mathcal{E}_{F;\mathrm{lin}}(t)}_{L^2}^2
 & \le &  K^2_\k K_{F;\mathrm{lin}}^2 \frac{M^2}{2 \beta - \e} N_{\e;II}(t),
 \\[0.2cm]
 \norm{\mathcal{E}_{F;\mathrm{nl}}(t)}_{L^2}^2
 & \le &  \eta K_\k^2K_{F;\mathrm{nl}}^2 M^2( 1 + \eta^3)^2
  N_{\e;II}(t),
\end{array}
\end{equation}
together with
\begin{equation}
\begin{array}{lcl}
\mathcal{I}_{\e,\delta;0}(t)
 & \le &  \frac{M^2}{2 \beta - \e} e^{-\e t} \norm{V(0)}^2_{H^1},
 \\[0.2cm]
  \mathcal{I}^{\mathrm{lt}}_{\e,\delta;F;\mathrm{lin}}(t)
 & \le &
  K^2_\k K_{F;\mathrm{lin}}^2 \frac{M^2}{2(\beta - \e)
    \e} N_{\e;II}(t),
    \\[0.2cm]
    \mathcal{I}^{\mathrm{sh}}_{\e,\delta;F;\mathrm{lin}}(t)
 & \le & 4 e^{\e} M^2 K_\k K_{F;\mathrm{lin}}^2
 N_{\e;II}(t) ,
 \\[0.2cm]
 \mathcal{I}^{\mathrm{lt}}_{\e,\delta;F;\mathrm{nl}}(t)
 & \le &    \eta K_\k^2K_{F;\mathrm{nl}}^2 ( 1 + \eta^3)^2
   \frac{M^2}{ \beta  - \e }
      N_{\e;II}(t) .
\end{array}
\end{equation}
\end{lem}
\begin{proof}
Observe first that
\begin{equation}
\begin{array}{lcl}
\nrm{\mathcal{E}_{F;\mathrm{lin}}(t)}_{L^2}^2
& \le &
   K_{F;\mathrm{lin}}^2 M^2  \left(
     \int_0^{\tau_{i}(t)} e^{-\beta(\tau_{i}(t) - \tau) } \norm{\overline{V}_i(\tau)}_{H^1} \, d\tau \right)^2 .
\end{array}
\end{equation}
Substituting $s = t_i(\tau)$ we find
\begin{equation}
\begin{array}{lcl}
\nrm{\mathcal{E}_{F;\mathrm{lin}}(t)}_{L^2}^2
  & \le &
  K_{F;\mathrm{lin}}^2 M^2  \left(
     \int_0^{t} e^{- (\beta - \frac{\e}{2} )(\tau_{i}(t) - \tau_{i}(s)) } e^{-\frac{\e}{2} ( \tau_{i}(t) - \tau_{i}(s)) }
     \norm{V(s)}_{H^1}\tau'_i(s) \, ds \right)^2 .
\\[0.2cm]
\end{array}
\end{equation}
Applying \sref{eq:nls:TauEst:tau:est:i}
and using
\sref{eq:prlm:kappa:glb:ests} to bound the extra integration factor
$\tau_i'(s)$ by $K_{\kappa}$,
we obtain
\begin{equation}
\begin{array}{lcl}
\nrm{\mathcal{E}_{F;\mathrm{lin}}(t)}_{L^2}^2
  & \le &
  K_\k^2K_{F;\mathrm{lin}}^2 M^2 \left(
     \int_0^{t} e^{- (\beta - \frac{\e}{2} )(t-s) } e^{-\frac{\e}{2} ( t-s) } \norm{V(s)}_{H^1} \, ds \right)^2 .
\\[0.2cm]
\end{array}
\end{equation}
Cauchy-Schwartz now yields the desired bound
\begin{equation}
\begin{array}{lcl}
\nrm{\mathcal{E}_{F;\mathrm{lin}}(t)}_{L^2}^2
& \le &
  K_\k^2K_{F;\mathrm{lin}}^2 \frac{M^2}{2 \beta - \e}
     \int_0^{t} e^{- \e ( t - s) } \norm{V(s)}_{H^1}^2 \, ds
\\[0.2cm]
& = &
  K_\k^2K_{F;\mathrm{lin}}^2 \frac{M^2}{2 \beta - \e} N_{\e;II}(t) .
\end{array}
\end{equation}
The remaining estimates follow in an analogous fashion
by making similar small adjustments to the proofs of Lemmas 9.9-9.11 in \cite{Hamster2017}.
\end{proof}

Our next result discusses the novel second-order terms.
The crucial ingredient here is that we no longer
have to consider the
dangerous $\p_{\xi} S(t_i(\tau) - \tau) Q \Upsilon_{\sigma;i}\big(\overline{V}(\tau) \big)$
term for $\tau \ge t_i(\tau) - 1$. Indeed,
this term need not be integrable even in $L^2$
because of the divergent $(\tau_i(t) - \tau)^{-1/2}$
behavior of $\partial_{\xi} S$ and the fact that we only have square-integrable
control of the $H^1$-norm of $\overline{V}_i(\tau)$.

\begin{lem}\label{lem:nls:so:st}
Fix $T > 0$ and assume that (HDt), (HSt) and (HTw) all hold.
Pick a constant $0< \e<2\b$.
Then for any
$0 \leq \delta < 1$
and any $0 \le t \le t_{\mathrm{st}}$,
we have the bounds
\begin{equation}
\begin{array}{lcl}
\norm{ \mathcal{E}^{\mathrm{sh}}_{\mathrm{so}}(t)}_{L^2}^2
  &\leq &
 9 \s^4 e^{2 \beta} K^2 K_\k M^2 N_{\e;II}(t),
\\[0.2cm]
\norm{ \mathcal{E}^{\mathrm{lt}}_{\mathrm{so}}(t)}_{L^2}^2
  &\leq &
 4 \s^4  K^2 K_\k \frac{M^2}{2\beta - \e} N_{\e;II}(t),
\end{array}
\end{equation}
together with
\begin{equation}
\begin{array}{lcl}
\mathcal{I}^{\mathrm{sh}}_{\e,\delta;\mathrm{so}}(t)
&\leq &
  9 \s^4 e^{2 \beta} K^2K_\k M^2 N_{\e;II}(t)
\\[0.2cm]
\mathcal{I}^{\mathrm{lt}}_{\e,\delta;\mathrm{so}}(t)
&\leq &
  4 \s^4K^2K_\k \frac{M^2}{2(\beta - \e) \e} N_{\e;II}(t) .
\end{array}
\end{equation}
\end{lem}
\begin{proof}
For $\tau \ge \tau_i(t) - 1$
we may use Lemma \ref{lem:nls:sem:group:decay}
together with Proposition
\ref{prp:fnl:bnds}
to obtain the estimate
\begin{equation}
\begin{array}{lcl}
\norm{\partial_{\xi} S(\tau_i(t) - \tau) P
  \Upsilon_{\sigma;i}\big( \overline{V}_i(\tau) \big)
}_{H^1} & \le &
   \sigma^2 K M \norm{\overline{V}_i(\tau)}_{H^1}
\\[0.2cm]
& \le &
   e^{\beta} \sigma^2 K M
     e^{-\beta(\tau_i(t) - \tau) } \norm{\overline{V}_i(\tau)}_{H^1}.
\end{array}
\end{equation}
In the same fashion we obtain
\begin{equation}
\begin{array}{lcl}
\norm{\Lambda(\tau_i(t) - \tau)
  \Upsilon_{\sigma;i}\big( \overline{V}_i(\tau) \big)
}_{H^1} & \le & e^{\beta} \sigma^2 K M
    e^{-\beta(\tau_i(t) - \tau) } \norm{\overline{V}_i(\tau)}_{H^1} ,
\\[0.2cm]
\norm{\partial_{\xi} S_{\mathrm{od}}(\tau_i(t) - \tau)
  \Upsilon_{\sigma;i}\big( \overline{V}_i(\tau) \big)
}_{H^1}  & \le &
e^{\beta} \sigma^2 K M
    e^{-\beta(\tau_i(t) - \tau) } \norm{\overline{V}_i(\tau)}_{H^1} .
\end{array}
\end{equation}
In addition,
for $\tau \le \tau_i(t) - 1$
we obtain
\begin{equation}
\nrm{
\big[ \p_{\xi}S(\tau_{i}(t)-\tau)Q
  +  \Lambda
    ({\tau_{i}(t)}-\tau) \big]
  \Upsilon_{\s;i}\big(\oV_{i}(\tau)\big)
  }_{H^1}
  \le 2 K M \sigma^2 \norm{\overline{V}_i(\tau)}_{H^1} e^{-\beta(\tau_i(t) - \tau) } .
\end{equation}
The desired estimates can hence be obtained
in the same fashion as the bounds
for $\mathcal{E}_{F;\mathrm{lin}}(t)$
and $\mathcal{I}_{\e,\delta;F;\mathrm{lin}}^{\mathrm{lt}}(t)$
in Lemma \ref{lem:ColPrevRes} .
\end{proof}

The following results at times
do require the computations in \cite{Hamster2017} to be modified
in a subtle non-trivial fashion. We therefore provide full proofs here,
noting, however, that the main ideas remain unchanged.
\begin{lem}
\label{lem:nls:f:nl:st}
Fix $T > 0$ and assume that (HDt), (HSt) and (HTw) all hold.
Pick a constant $\e > 0$.
Then for any
$\eta > 0$,
any $0 \leq \delta < 1$
and any $0 \le t \le t_{\mathrm{st}}$,
we have the bound
\begin{equation}
\begin{array}{lcl}
\mathcal{I}^{\mathrm{sh}}_{\e,\delta;F,\mathrm{nl}}(t) & \le &
 \eta M^2e^{3\e} K_\k^2K_{F;\mathrm{nl}}^2 (1 + \eta^3)^2 (1 + \rho_{\mathrm{min}}^{-1})(3K_\k+2)
   N_{\e;II}(t) .
\end{array}
\end{equation}
\end{lem}
\begin{proof}
We first introduce the inner product
\begin{equation}
\label{eq:nls:def:inn:prod:h:1:rho}
\langle v , w \rangle_{H^1_{\rho} } = \langle v , w \rangle_{L^2} + \langle \sqrt{\rho} \partial_{\xi} v , \sqrt{\rho} \partial_{\xi} w \rangle_{L^2}
\end{equation}
and note that
\begin{equation}
\norm{v}_{H^1}^2 \leq \norm{v}_{L^2}^2 + \rho_{\mathrm{min}}^{-1} \norm{\sqrt\rho \p_{\xi} v}_{L^2}^2
\le \big(1 + \rho_{\mathrm{min}}^{-1}  \big) \langle v , v \rangle_{H^1_\rho} .
\end{equation}
For $\# \in \{L^2, H^1_{\rho} \}$
we introduce the expression
\begin{equation}\label{eq:nls:DefE}
\mathcal{E}_{\tau,\tau',\tau'';\#}
= \Big\langle S(\tau+\d-\tau')QF_{\mathrm{nl}}\big(\oV_{i}(\tau')\big),S(\tau + \d-\tau'')QF_{\mathrm{nl}}\big(\oV_{i}(\tau'')\big)\Big\rangle_{\#},
\end{equation}
which allows us to obtain the estimate
\begin{equation}
\label{eq:nls:bnd:i:sh:extr:term}
\begin{array}{lcl}
\mathcal{I}^{\mathrm{sh}}_{\e,\delta;F,\mathrm{nl}}(t)
& \le &  \big(1 + \rho_{\mathrm{min}}^{-1}  \big)
\int_0^te^{\e(t-s)}
\int_{\tau_{i}(s)-1}^{\tau_{i}(s)}\int_{\tau_{i}(s)-1}^{\tau_{i}(s)}
\mathcal{E}_{\tau_i(s),\tau',\tau'';H^1_\rho} d\tau''d\tau'ds
\\[0.2cm]
& \le &
 \big(1 + \rho_{\mathrm{min}}^{-1}  \big)
\int_0^te^{\e(t-s)} \big[ t_i'\big(\tau_i(s) \big) \big]^{-1}
\int_{\tau_{i}(s)-1}^{\tau_{i}(s)}\int_{\tau_{i}(s)-1}^{\tau_{i}(s)}
\mathcal{E}_{\tau_i(s),\tau',\tau'';H^1_\rho} d\tau''d\tau'ds .
\end{array}
\end{equation}
The extra term involving the function $t_i'$, which takes values in $[K_{\kappa}^{-1}, 1]$,
was included for technical reasons
that will become clear in wat follows.

For any $v,w \in L^2$, $\vartheta > 0$, $\vartheta_A \ge 0$
and $\vartheta_B \ge 0$,
we have
\begin{equation}
\label{eq:nls:deriv:semigroup}
\begin{array}{lcl}
\frac{d}{d \vartheta} \langle S(\vartheta + \vartheta_A) v,
   S(\vartheta + \vartheta_B) w \rangle_{L^2}
& = &
  \langle \mathcal{L}_\mathrm{tw} S(\vartheta + \vartheta_A) v ,
     S(\vartheta +\vartheta_B) w \rangle_{L^2}
\\[0.2cm]
& & \qquad
+ \langle  S(\vartheta + \vartheta_A) v ,
    \mathcal{L}_\mathrm{tw}S(\vartheta + \vartheta_B) w \rangle_{L^2}
\\[0.2cm]
& = &
  \langle  S(\vartheta + \vartheta_A) v ,
      \mathcal{L}_\mathrm{tw}^* S(\vartheta + \vartheta_B) w \rangle_{L^2}
\\[0.2cm]
& & \qquad
+ \langle  S(\vartheta + \vartheta_A) v ,
   \mathcal{L}_\mathrm{tw}  S(\vartheta + \vartheta_B) w \rangle_{L^2}
\\[0.2cm]
& = &
  \langle  S(\vartheta + \vartheta_A) v ,
     \big[ \mathcal{L}_\mathrm{tw}^* - \rho\p_{\xi\xi}\big]
       S(\vartheta + \vartheta_B) w \rangle_{L^2}
\\[0.2cm]
& & \qquad
+ \langle  S(\vartheta + \vartheta_A) v ,  \
  \big[ \mathcal{L}_\mathrm{tw} - \rho\p_{\xi\xi}\big]  S(\vartheta + \vartheta_B) w
     \rangle_{L^2}
\\[0.2cm]
& & \qquad
  - 2 \langle \sqrt\rho\p_{\xi} S(\vartheta + \vartheta_A) v,
    \sqrt\rho\p_{\xi} S(\vartheta + \vartheta_B) w \rangle_{L^2} .
\end{array}
\end{equation}
Upon taking $\delta>0$ for the moment and choosing
$v=QF_{\mathrm{nl}}\big(\oV_{i}(\tau')\big)$, $w=QF_{\mathrm{nl}}(\oV_{i}(\tau''))$,
$\vartheta=\tau_i(s)+\delta$, $\vartheta_A=\tau'$ and $\vartheta_B=\tau''$,
we may rearrange \sref{eq:nls:deriv:semigroup} to obtain the estimate
\begin{equation}\label{eq:nls:BoundE}
\begin{array}{lcl}
\mathcal{E}_{\tau_i(s),\tau',\tau'';H_\rho^1}
& \le & M^2 K_{F;\mathrm{nl}}^2 (1 + \eta^3)^2 \norm{\oV_{i}(\tau')}_{H_\rho^1}^2
  \norm{\oV_{i}(\tau'')}_{H^1}^2
\\[0.2cm]
& & \qquad
 +  M^2 K_{F;\mathrm{nl}}^2 (1 + \eta^3)^2
 \frac{1}{\sqrt{\tau_{i}(s) + \delta - \tau''} }
    \norm{\oV_{i}(\tau')}_{H^1}^2 \norm{\oV_{i}(\tau'')}_{H^1}^2
\\[0.2cm]
& & \qquad
  - \frac{1}{2}
  \partial_1
    \mathcal{E}_{\tau_i(s),\tau',\tau'';L^2}
\end{array}
\end{equation}
for the values of $(s,\tau',\tau'')$ that are relevant.

Upon introducing the integrals
\begin{equation}
\label{eq:nls:def:i:i:ii}
\begin{array}{lcl}
\mathcal{I}_{I}
 & = &
     \int_0^{t}
   e^{-\e(t-s) } \big[ t_i'\big(\tau_i(s) \big) \big]^{-1} \int_{\tau_{i}(s)-1}^{\tau_{i}(s)} \int_{\tau_{i}(s)-1}^{\tau_{i}(s)}\\[0.2cm]
       &&\qquad \qquad \big[ 1 + \frac{1}{\sqrt{\tau_{i}(s) + \delta - \tau''} } ]
          \norm{\oV_{i}(\tau')}_{H^1}^2 \norm{\oV_{i}(\tau'')}_{H^1}^2
             \, d\tau'' \, d\tau' \, ds ,
\\[0.2cm]
\mathcal{I}_{II}
& = &
 \int_0^{t}
   e^{-\e(t-s) } \big[ t_i'\big(\tau_i(s) \big) \big]^{-1} \int_{\tau_{i}(s)-1}^{\tau_{i}(s)} \int_{\tau_{i}(s)-1}^{\tau_{i}(s)}
    \partial_1 \mathcal{E}_{\tau_i(s),\tau',\tau'';L^2} \, d\tau'' \, d\tau' \, ds ,
\\[0.2cm]
\end{array}
\end{equation}
we hence readily obtain the estimate
\begin{equation}
\begin{array}{lcl}
\mathcal{I}^{\mathrm{sh}}_{\e,\delta;B;\mathrm{nl}}(t)
& \le &
  (1 + \rho_{\mathrm{min}}^{-1}) M^2 K^2_{F;\mathrm{nl}}(1 + \eta^3)^2 \mathcal{I}_{I}
  - \frac{1}{2} (1 + \rho_{\mathrm{min}}^{-1}) \mathcal{I}_{II}.
\end{array}
\end{equation}

Using Lemma \ref{lem:TauEst} we see that
\begin{equation}
\begin{array}{lcl}
\mathcal{I}_{I}
 & \leq &
   K_\k^3  \int_0^{t}
   e^{-\e(t-s) } \int_{s-1}^{s} \int_{s-1}^{s}
       \big[ 1 + \frac{1}{\sqrt{s + \delta - s''} } ]
          \norm{V(s')}_{H^1}^2 \norm{V(s'')}_{H^1}^2
             \, ds'' \, ds' \, ds ,
\end{array}
\end{equation}
which allows us to repeat the computation \cite[(9.68)]{Hamster2017}
and conclude
\begin{align}
\mathcal{I}_{I}
 \leq
   3\eta e^{3\e} K_\k^3    N_{\e;II}(t).
\end{align}

To understand $\mathcal{I}_{II}$ it is essential to change the order of integration and integrate with respect to $s$
before switching $\tau'$ and $\tau''$ back to the original time.
Rearranging the integrals in \sref{eq:nls:def:i:i:ii}, we find
\begin{equation}
\begin{array}{lcl}
 \mathcal{I}_{II} & =&
 \int_0^{\tau_i(t)} e^{-\e t}
  \int_{\max\{0, \tau'-1\}}^{\min\{\tau_i(t), \tau'+1 \} }
 \Big[\int_{\max\{t_i(\tau'),t_i(\tau'')\}}^{\min\{ t,t_i(\tau'+1), t_i(\tau''+1) \} }
                      \frac{e^{\e s}}{t_i'\big(\tau_i(s) \big) }
                        \partial_1 \mathcal{E}_{\tau_i(s),\tau',\tau'';L^2}   \, ds
                   \Big]        \, d \tau '' \, d \tau'.
\end{array}
\end{equation}
Introducing the notation
\begin{equation}
\tau^+(\tau', \tau'') = \min\{ \tau_i(t), \tau' + 1 , \tau'' + 1 \},
\qquad
\tau^-(\tau', \tau'') = \max\{ \tau', \tau'' \},
\end{equation}
the substitution $\tau = \tau_i(s)$ yields
\begin{equation}
\begin{array}{lcl}
 \mathcal{I}_{II} & =&
 \int_0^{\tau_i(t)} e^{-\e t}
  \int_{\max\{0, \tau'-1\}}^{\min\{\tau_i(t) , \tau'+1 \} }
 \Big[\int_{\tau^-(\tau', \tau'')}^{\tau^+(\tau', \tau'') }
                      e^{\e t_i(\tau)} 
                       \partial_1 \mathcal{E}_{\tau,\tau',\tau'';L^2}  \, d \tau
                   \Big]        \, d\tau '' \, d\tau' .
 \\[0.2cm]
\end{array}
\end{equation}
We emphasize here that the integration factor associated to this substitution cancels out against
the additional term introduced in \sref{eq:nls:bnd:i:sh:extr:term}.
%
Integrating by parts, we find
\begin{equation}
\begin{array}{lcl}
 \mathcal{I}_{II}
 &=&\mathcal{I}_{II;A}+\mathcal{I}_{II;B}+\mathcal{I}_{II;C}
\end{array}
\end{equation}
in which
 we have introduced
\begin{equation}
\begin{array}{lcl}
\mathcal{I}_{II;A}
 & = &
     \int_0^{\tau_i(t)} e^{-\e t} \int_{\max\{0, \tau'-1\}}^{\min\{\tau_i(t), \tau'+1\} }
     e^{\e t_i(\tau)}  
       \mathcal{E}_{\tau, \tau',\tau'';L^2}\big|_{\tau=\tau^+(\tau', \tau'')} \, d\tau'' \, d\tau' ,
 \\[0.2cm]
\mathcal{I}_{II;B}
 & = &
   -\int_0^{\tau_i(t)} e^{-\e t} \int_{\max\{0,\tau'-1\}}^{\min\{\tau_i(t), \tau'+1\} }
     e^{\e t_i(\tau) }   
       \mathcal{E}_{\tau,\tau',\tau'';L^2}\big|_{\tau=\tau^-(\tau', \tau'')} d\tau'' \, d\tau' ,
\\[0.2cm]
\mathcal{I}_{II;C}
 & = &
   -  \int_0^{\tau_i(t)} e^{-\e t}
  \int_{\max\{0, \tau'-1\}}^{\min\{\tau_i(t) , \tau'+1 \} }
 \Big[\int_{\tau^-(\tau', \tau'')}^{\tau^+(\tau', \tau'') }
                      \left(\frac{d}{d\tau} e^{\e t_i(\tau)} \right)
    \mathcal{E}_{\tau,\tau',\tau'';L^2}    \, d\tau
                   \Big]        \, d\tau '' \, d\tau'  .
 \\[0.2cm]
\end{array}
\end{equation}
Note here that $\mathcal{I}_{II;B}$ is well defined because $\delta>0$.

Using
the substitutions
\begin{equation}
s' = t_i(\tau'),
\qquad
s'' = t_i(\tau'')
\end{equation}
together with the bound
\begin{equation}
\begin{array}{lcl}
t_i\big( \tau_-(\tau', \tau'') \big)
& \le & t_i\big( \tau^+(\tau', \tau'') \big)
\\[0.2cm]
& \le &
\min\{ t , t_i(\tau' + 1), t_i(\tau'' + 1) \}
\\[0.2cm]
& \le & \min \{ t , t_i(\tau') + 1 , t_i(\tau'') + 1 \},
\\[0.2cm]
& \le & \min\{  s' , s''\} + 1
\\[0.2cm]
& \le & s' + 1,
\end{array}
\end{equation}
we find
\begin{equation}
\int_{\tau^-(\tau', \tau'')}^{\tau^+(\tau', \tau'') }
                      \abs{ \frac{d}{d\tau} e^{\e t_i(\tau)}  }d\tau
=
\int_{\tau^-(\tau', \tau'')}^{\tau^+(\tau', \tau'') }
                     \frac{d}{d\tau} e^{\e t_i(\tau)}d\tau
=  e^{\e t_i(\tau) } \big|_{\tau^-(\tau', \tau'')}^{\tau^+(\tau', \tau'')}
\le 2 e^{\e} e^{ \e s'} .
\end{equation}
Applying  Cauchy-Schwartz to the inner product $\mathcal{E}$,
we hence obtain
\begin{equation}
\begin{array}{lcl}
\abs{\mathcal{I}_{II} }
 & \leq & 4 e^{\e}M^2K_\k^2K_{F;\mathrm{nl}}^2(1+\eta^3)^2
     \int_0^{t} e^{-\e (t - s')}   \nrm{V(s')}_{H^1}^2
     \mathcal{J}(s') \, d s' ,
 \\[0.2cm]
\end{array}
\end{equation}
in which we have introduced the function
\begin{equation}
\mathcal{J}(s')
=    \int_{\max\{0,t_i(\tau_i(s')-1)\}}^{\min\{t, t_i(\tau_i(s')+1)\} }
                 \nrm{V(s'')}_{H^1}^2 \, ds'' .
\end{equation}
Exploiting Lemma \ref{lem:TauEst} again,
we can bound
\begin{equation}
\begin{array}{lcl}
\mathcal{J}(s')
& \le &    \int_{\max\{0, s' - 1 \} }^{\min\{t, s' + 1 \}}
                 \nrm{V(s'')}_{H^1}^2 \, ds''
\\[0.2cm]
& \le &
  \int_{\max\{0, s' - 1 \} }^{\min\{t, s' + 1 \}}
                 e^{2 \e} e^{ - \e( \min\{t , s'+ 1\} - s'' ) }\nrm{V(s'')}_{H^1}^2 \, ds''
\\[0.2cm]
& \le & e^{2 \e} \eta,
\end{array}
\end{equation}
which hence gives
\begin{equation}
\begin{array}{lcl}
\abs{\mathcal{I}_{II} }
 & \leq & 4 \eta e^{3 \e}M^2K_\k^2K_{F;\mathrm{nl}}^2(1+\eta^3)^2
        N_{\e;II}(t),
 \\[0.2cm]
\end{array}
\end{equation}
as desired. It hence remains to consider the case $\delta = 0$.
We may apply Fatou's lemma
to conclude
\begin{equation}
\begin{array}{lcl}
\mathcal{I}^{\mathrm{sh}}_{\e,0;F;\mathrm{nl}} (t)
& = &
  \int_0^{t}
e^{\e(t-s)}
(\lim_{\d\to0}\nrm{S(\d)
  \mathcal{E}^{\mathrm{sh}}_{B;\mathrm{lin}}(s) }_{H^1})^2
    \mathbf{1}_{s < t_{\mathrm{st}}} \, ds
\\[0.2cm]
&\leq & \liminf_{\d\to0}
   \mathcal{I}^{\mathrm{sh}}_{\e,\delta;F;\mathrm{nl}} (t).
\\[0.2cm]
\end{array}
\end{equation}
The result now follows from the fact that
the bounds obtained above do not depend on  $\delta$.
\end{proof}

\subsection{Stochastic Regularity Estimates}
We are now ready to discuss the stochastic integrals.
These require special care because they cannot be bounded
in a pathwise fashion, unlike the deterministic integrals above.
Expectations of suprema are particularly delicate in this respect.
Indeed, the powerful Burkholder-Davis-Gundy inequalities cannot
be directly applied to the stochastic convolutions that arise in our
mild formulation. However, as was shown in Lemma 9.7 in \cite{Hamster2017}, we can obtain an $H^\infty$-calculus for our linear operator $\mathcal{L}_{\mathrm{tw}}$ which
allows us to use the following mild version,
which is the source of the extra $T$ factors
that appear in our estimates.

\begin{lem}
\label{lem:nls:mild:regularity}
Fix $T > 0$ and assume that (HDt), (HSt) and (HTw)
all hold.
There exists a constant $K_{\mathrm{cnv}} > 0$ so that for any
$W \in \mathcal{N}^2([0,T];(\mathcal{F})_{t} ; L^2)$
we have
\begin{equation}
E \sup_{0 \le t \le T}
  \norm{ \int_0^t S( t - s) Q W(s) \, d \beta_s }_{L^2}^2
  \le K_{\mathrm{cnv}} E \int_0^T \norm{W(s)}_{L^2}^2 \, d s.\\
\end{equation}
\end{lem}
\begin{proof}
This is a direct result of the computations in \cite[{\S}9.1]{Hamster2017}, which are based on the main theorem of \cite{veraar2011note}.
\end{proof}

\begin{lem}
\label{lem:nls:b:lin:sup:e}
Fix $T > 0$ and assume that (HDt), (HSt), and (HTw) all hold.
Then for any
$\e > 0$ we have the bound
\begin{equation}
\begin{array}{lcl}
E \sup_{0 \le t \le t_\mathrm{st}} \norm{\mathcal{E}_{B;\mathrm{lin}}(t )}_{L^2}^2
 & \le &  (T + 1)  K_{\mathrm{cnv}} K_{B;\mathrm{lin}}^2 e^{\e}
   E \sup_{0 \le t \le t_\mathrm{st}}   N^i_{\e;II}(t) .
\end{array}
\end{equation}
\end{lem}
\begin{proof}
Using Lemma \ref{lem:nls:mild:regularity}
we compute
\begin{equation}
\label{eq:nls:bnd:on:e:b:lin:dir:from:mild:fmr}
\begin{array}{lcl}
E \sup_{0 \le t \le t_\mathrm{st}} \norm{\mathcal{E}_{B;\mathrm{lin}}(t )}_{L^2}^2
  & \le   &
  E \sup_{0 \le t \le T} \norm{\mathcal{E}_{B;\mathrm{lin}}(t )}_{L^2}^2
\\[0.2cm]
& = &
    E \sup_{0 \le \tau \le \tau_i(T)}
      \norm{ \int_0^{\tau} S(\tau-\tau') Q B_{\mathrm{lin}}\big(\oV_i(\tau')\big)
        \mathbf{1}_{\tau' < \tau_i(t_{\mathrm{st}})} \, d \beta_{\tau'}
       }_{L^2}^2
\\[0.2cm]
& \le &
  K_{\mathrm{cnv}} E
      \int_{0}^{\tau_i(T)}
        \norm{
           B_{\mathrm{lin}}\big(\oV_i(\tau)\big)
            \mathbf{1}_{\tau < \tau_i(t_{\mathrm{st}})}
         }_{L^2}^2 \, d\tau
\\[0.2cm]
& \le &
  K_\k K_{\mathrm{cnv}} K_{B;\mathrm{lin}}^2
  E \int_0^{t_{\mathrm{st}}} \norm{V(s)}_{H^1}^2 \, ds .
\end{array}
\end{equation}
By dividing up the integral, we obtain
\begin{equation}
\begin{array}{lcl}
\int_0^{t_{\mathrm{st}}}
      \norm{V(s)}_{H^1}^2 \, ds
& \le &
    e^{\e} \int_0^{1 }
     e^{ -\e(1 - s) }
       \norm{V(s)}_{H^1}^2 \mathbf{1}_{s < t_{\mathrm{st}}}  \, ds
\\[0.2cm]
& & \qquad
  +
     e^{\e} \int_1^{2 }
        e^{ -\e(2 - s) }
        \norm{V(s)}_{H^1}^2 \mathbf{1}_{s < t_{\mathrm{st}}}  \, ds
\\[0.2cm]
& & \qquad
  + \ldots
+
   e^{\e} \int_{\lfloor T \rfloor}^{\lfloor T \rfloor + 1}
       e^{ -\e(\lfloor T \rfloor +1 - s) }
       \norm{V(s)}_{H^1}^2 \mathbf{1}_{s < t_{\mathrm{st}}}  \, ds
\\[0.2cm]
& \le &
  (T + 1)   e^{\e}
     \sup_{0 \le t \le T + 1}
      \int_0^t e^{-\e(t - s)}
        \norm{V(s)}_{H^1}^2 \mathbf{1}_{s < t_{\mathrm{st}}} \, ds
\\[0.2cm]
& \le &
  (T + 1)   e^{\e}
     \sup_{0 \le t \le t_{\mathrm{st}}}
      \int_0^t e^{-\e(t - s)}
        \norm{V(s)}_{H^1}^2  \, ds
\\[0.2cm]
& = &
  (T + 1)   e^{\e}
     \sup_{0 \le t \le t_{\mathrm{st}}}
      N_{\e;II}(t) ,
\end{array}
\end{equation}
which yields the desired bound upon taking expectations.
\end{proof}

\begin{lem}
\label{lem:nls:b:cn:sup:e}
Fix $T > 0$ and assume that (HDt), (HSt) and (HTw) all hold.
Then we have the bound
\begin{equation}
\begin{array}{lcl}
E \sup_{0 \le t \le t_{\mathrm{st}}} \norm{\mathcal{E}_{B;\mathrm{cn}}(t )}_{L^2}^2
 & \le &  T K_{\mathrm{cnv}} K_{B;\mathrm{cn}}^2 .
\end{array}
\end{equation}
\end{lem}
\begin{proof}
This bound follows directly from \sref{eq:nls:bnd:on:e:b:lin:dir:from:mild:fmr}
by making the substitutions
\begin{equation}
  K_{B;\mathrm{lin}} \mapsto K_{B;\mathrm{cn}},
  \qquad
  \qquad
    \nrm{V(s)}_{H^1}^2 \mapsto 1 .
\end{equation}
\end{proof}

We now set out to bound the expectation of the suprema of the remaining double integrals
$\mathcal{I}^{\#}_{\e,\delta;B;\mathrm{lin}}(t)$
and $\mathcal{I}^{\#}_{\e,\delta;B;\mathrm{cn}}(t)$ with $\# \in \{ \mathrm{lt} , \mathrm{sh} \}$.
This is performed in Lemma \ref{lem:nls:st:sup:on:b:lin:i},
but we first compute several time independent bounds for the expectation of the integrals themselves.

\begin{lem}
Fix $T > 0$ and assume that (HDt), (HSt) and (HTw) all hold.
Pick a constant $\e > 0$.
Then for any $0 \le \delta < 1$ and $0 \le t \le T$,
we have the identities
\begin{equation}
\label{eq:nls:ito:isometry:i:lt}
\begin{array}{lcl}
E \, \mathcal{I}^{\mathrm{lt}}_{\e,\delta;B;\mathrm{lin}}(t)
& = &
 E  \int_0^{t}
   e^{-\e(t-s) } \int_{0}^{\tau_{i}(s)-1}
          \norm{  S(\tau_{i}(s)+\delta-\tau') Q B_{\mathrm{lin}}\big(\oV_{i}(\tau')\big) }_{L^2}^2 \mathbf{1}_{\tau' < \tau_i(t_{\mathrm{st}}) }  d\tau' \, ds ,
\\[0.2cm]
E \, \mathcal{I}^{\mathrm{lt}}_{\e,\delta;B;\mathrm{cn}}(t)
& = &
 E \int_0^{t}
   e^{-\e (t-s) } \int_{0}^{\tau_{i}(s)-1}
          \norm{  S(\tau_{i}(s)+\delta-\tau') Q B_{\mathrm{cn}} }_{L^2}^2 \mathbf{1}_{\tau' < \tau_i(t_{\mathrm{st}}) }  d\tau' \, ds
\end{array}
\end{equation}
and their short-time counterparts
\begin{equation}
\label{eq:nls: :isometry:i:sh}
\begin{array}{lcl}
E \, \mathcal{I}^{\mathrm{sh}}_{\e,\delta;B;\mathrm{lin}}(t)
& = &
 E  \int_0^{t}
   e^{-\e(t-s) } \int_{\tau_{i}(s)-1}^{\tau_{i}(s)}
          \norm{  S(\tau_{i}(s)+\delta-\tau') Q B_{\mathrm{lin}}\big(\oV_{i}(\tau')\big) }_{L^2}^2 \mathbf{1}_{\tau' < \tau_i(t_{\mathrm{st}}) } d\tau' \, ds ,
\\[0.2cm]
E \, \mathcal{I}^{\mathrm{sh}}_{\e,\delta;B;\mathrm{cn}}(t)
& = &
 E \int_0^{t}
   e^{-\e(t-s) } \int_{\tau_{i}(s)-1}^{\tau_{i}(s)}
          \norm{  S(\tau_{i}(s)+\delta-\tau') Q B_{\mathrm{cn}} }_{L^2}^2 \mathbf{1}_{\tau' < \tau_i(t_{\mathrm{st}}) } d\tau' \, ds .
\end{array}
\end{equation}
\end{lem}
\begin{proof}
This follows directly from the It\^o Isometry, see also Lemma 9.16 in \cite{Hamster2017}. \end{proof}

\begin{lem}
\label{lem:nls:b:lin}
Fix $T > 0$, assume that (HDt), (HSt) and
(HTw) all hold and
pick a constant $0 < \e < 2 \beta$.
Then for
any $0 \leq  \delta < 1$ 
and any $0 \le t \le T$,
we have the bound
\begin{equation}
\begin{array}{lcl}
E \, \mathcal{I}^{\mathrm{lt}}_{\e,\delta;B;\mathrm{lin}}(t)
 & \le &    \frac{M^2}{2 \beta - \e} K_\k K_{B;\mathrm{lin}}^2
  E  N_{\e;II}(t\wedge t_\mathrm{st}).
\end{array}
\end{equation}
\end{lem}
\begin{proof}
Using \sref{eq:nls:ito:isometry:i:lt}
and switching the integration order,
we obtain
\begin{equation}
\begin{array}{lcl}
E \, \mathcal{I}^{\mathrm{lt}}_{\e,\delta;B;\mathrm{lin}}(t)
  &\leq  &
  M^2 K_{B;\mathrm{lin}}^2 E   \int_0^{t}
  e^{ - \e(t - s) }
    \int_0^{\tau_{i}(s)\wedge\tau_i(t_{\mathrm{st}})}  e^{-2 \beta(\tau_{i}(s)-\tau')}\norm{\oV_{i}(\tau')}_{H^1}^2 \, d\tau' \, ds
\\[0.2cm]
  &\leq  &
  M^2 K_\k K_{B;\mathrm{lin}}^2 E \int_0^{t}
  e^{ - \e(t - s) }
    \int_0^{s\wedge t_{\mathrm{st}}}  e^{-2 \beta(s-s')}\norm{V(s')}_{H^1}^2 \, ds' \, ds
\\[0.2cm]
& = &
  M^2 K_\k K_{B;\mathrm{lin}}^2 E   \int_0^{t\wedge t_{\mathrm{st}}}
   e^{-\e t}
   \Big[\int_{s'}^{t}
      e^{  -(2 \beta -\e )s } \, d s
   \Big]
      e^{2 \beta s'}\norm{V(s')}_{H^1}^2  \, ds'
\\[0.2cm]
& \le &
  \frac{M^2}{2 \beta - \e} K_\k K_{B;\mathrm{lin}}^2
  E   \int_0^{t\wedge t_{\mathrm{st}}}
   e^{-\e t} e^{-(2 \beta - \e)s'}
       e^{2 \beta s'}\norm{V(s')}_{H^1}^2  \, ds'
\\[0.2cm]
& \leq &
  \frac{M^2}{2 \beta - \e} K_\k K_{B;\mathrm{lin}}^2
  E   \int_0^{t\wedge t_{\mathrm{st}}}
   e^{-\e(t\wedge t_{\mathrm{st}} - s' )}  \norm{V(s')}_{H^1}^2  \, ds'
\\[0.2cm]
& = &
   \frac{M^2}{2 \beta - \e} K_\k K_{B;\mathrm{lin}}^2
  E  N_{\e;II}(t\wedge t_{\mathrm{st}}).
\end{array}
\end{equation}
\end{proof}

\begin{lem}
\label{lem:nls:b:lin:st}
Fix $T > 0$ and assume that (HDt), (HSt) and
(HTw), all hold.
Pick a constant $\e > 0$.
Then for
any $0 \leq \delta < 1$ ,
and any $0 \le t \le T$,
we have the bound
\begin{equation}
\begin{array}{lcl}
E \, \mathcal{I}^{\mathrm{sh}}_{\e,\delta;B;\mathrm{lin}}(t)
 & \le &    K_\k K_{B;\mathrm{lin}}^2 M^2 (1 + \rho_\mathrm{min}^{-1})e^\e( 3 K_\k  + 2 )
  E  N_{\e;II}(t\wedge t_{\mathrm{st}}).
\end{array}
\end{equation}
\end{lem}
\begin{proof} We only consider the case $\delta>0$ here, noting that the limit $\delta\downarrow0$ can be handled as in the proof of Lemma \ref{lem:nls:f:nl:st}. Applying the identity \sref{eq:nls:deriv:semigroup}
with $w=v$ and $\vartheta_A = \vartheta_B$,
we obtain
\begin{equation}
\label{eq:nls:deriv:semigroup2}
\begin{array}{lcl}
\frac{d}{d \vartheta} \nrm{ S(\vartheta + \vartheta_A) v}^2_{L^2}
& = &
  \langle  S(\vartheta + \vartheta_A) v ,
     \big[ \mathcal{L}_\mathrm{tw}^* - \rho\p_{\xi\xi}\big]
       S(\vartheta + \vartheta_A) v \rangle_{L^2}
\\[0.2cm]
& & \qquad
+ \langle  S(\vartheta + \vartheta_A) v ,  \
  \big[ \mathcal{L}_\mathrm{tw} - \rho\p_{\xi\xi}\big]  S(\vartheta + \vartheta_A) v
     \rangle_{L^2}
\\[0.2cm]
& & \qquad
  - 2 \nrm{ \sqrt\rho\p_{\xi} S(\vartheta + \vartheta_A) v}^2_{L^2} .
\end{array}
\end{equation}
Recalling the inner product \sref{eq:nls:def:inn:prod:h:1:rho}
and introducing the expression
\begin{equation}
\mathcal{E}_{\tau , \tau'; \#}
= \norm{S(\tau+\delta -\tau')Q B_{\mathrm{lin}}\big(\oV_{i}(\tau')\big) }_{\#}^2
\end{equation}
for $\# \in \{ L^2, H^1_\rho \}$,
we obtain the bound
\begin{equation}
\begin{array}{lcl}
\mathcal{E}_{\tau, \tau'; H^1_\rho}
& \le & M^2 K_{B;\mathrm{lin}}^2 \norm{\oV_{i}(\tau')}_{H^1}^2
 +  M^2 K_{B;\mathrm{lin}}^2 \frac{1}{\sqrt{\tau_{i}(s) + \delta - \tau'} }
    \norm{\oV_{i}(\tau')}_{H^1}^2
\\[0.2cm]
& & \qquad
  -  \frac{1}{2}
  \partial_1 \mathcal{E}_{\tau, \tau'; L^2 }
\end{array}
\end{equation}
for the values of $(s,\tau')$ that are relevant below.
Upon writing
\begin{align}\begin{split}
\mathcal{I}_{I}
 & =
    E \int_0^{t}
   e^{-\e(t-s) } \big[ t_i'\big(\tau_i(s) \big) \big]^{-1}\int_{\tau_{i}(s)-1}^{\tau_{i}(s)}
       \big[ 1 + \frac{1}{\sqrt{\tau_{i}(s) + \delta - \tau'} } ]
          \norm{\oV_{i}(\tau')}_{H^1}^2
          \mathbf{1}_{\tau'<\tau_i(t_\mathrm{st})} \, d\tau' \, ds ,
\\[0.2cm]
\mathcal{I}_{II}
& =
E  \int_0^{t}
   e^{-\e(t-s) }\big[ t_i'\big(\tau_i(s) \big) \big]^{-1} \int_{\tau_{i}(s)-1}^{\tau_{i}(s)}
    \partial_1 \mathcal{E}_{\tau_i(s) , \tau' ; L^2}
    \mathbf{1}_{\tau'<\tau_i(t_\mathrm{st})} \, d\tau' \, ds ,
    \end{split}
\end{align}
we obtain the estimate
\begin{equation}
\begin{array}{lcl}
E \, \mathcal{I}^{\mathrm{sh}}_{\nu,\delta;B;\mathrm{lin}}(t)
& \le &
  (1 + \rho_\mathrm{min}^{-1}) M^2 K^2_{B;\mathrm{lin}} \mathcal{I}_{I}
  - \frac{1}{2}(1 + \rho_\mathrm{min}^{-1})\mathcal{I}_{II} .
\end{array}
\end{equation}
Changing the integration order, we obtain
\begin{equation}
\begin{array}{lcl}
\mathcal{I}_{I}
 & =&
    E \int_0^{\tau_i(t\wedge t_{\mathrm{st}})} e^{-\e t}\Big[
 \int_{t_i(\tau')}^{\min\{t\wedge t_{\mathrm{st}},t_i(\tau'+1) \} }
                      \frac{e^{\e s}}{t_i'\big(\tau_i(s) \big)}
                      \big[
                         1 + \frac{1}{\sqrt{\tau_i(s)+\d-\tau'}} \big]  \, ds
                   \Big]   \nrm{\oV_{i}(\tau')}_{H^1}^2       \, d\tau ' ,
\\[0.2cm]
\mathcal{I}_{II}
& = &
E  \int_0^{\tau_i(t\wedge t_{\mathrm{st}})} e^{-\e t}
\int_{t_i(\tau')}^{\min\{ t\wedge t_{\mathrm{st}},t_i(\tau'+1) \} }
          \frac{e^{\e s}}{t_i'\big(\tau_i(s) \big)}
    \partial_1 \mathcal{E}_{\tau_i(s), \tau' ; L^2}
    \, ds
                         \, d\tau '.
\end{array}
\end{equation}
The substitution $s' = t_i(\tau')$
together with Lemma \ref{lem:TauEst}
 now yields
\begin{equation}
\begin{array}{lcl}
 \mathcal{I}_{I}
&\leq&
  K_\k^2E \int_0^{t\wedge t_{\mathrm{st}}} e^{-\e (t\wedge t_{\mathrm{st}})}\Big[
 \int_{s'}^{\min\{t\wedge t_{\mathrm{st}},t_i(\tau_i(s')+1) \} }
                      e^{\e s}
                      \big[
                         1 + \frac{1}{\sqrt{\tau_i(s)+\d-\tau(s')}} \big]  \, ds
                   \Big]   \nrm{V(s')}_{H^1}^2       \, ds '
\\[0.2cm]
& \leq &
  K_\k^2E \int_0^{t\wedge t_{\mathrm{st}}} e^{-\e (t\wedge t_{\mathrm{st}})}\Big[
 \int_{s'}^{\min\{t\wedge t_{\mathrm{st}},s' + 1 \} }
                      e^{\e s}
                      \big[
                         1 + \frac{1}{\sqrt{s+\d-s'}} \big]  \, ds
                   \Big]   \nrm{V(s')}_{H^1}^2       \, ds '
\\[0.2cm]
& \le &
3 e^{\e} K^2_\k
E \int_0^{t\wedge t_{\mathrm{st}}} e^{-\e (t\wedge t_{\mathrm{st}}-s')}  \nrm{V(s')}_{H^1}^2    \, ds'
\\[0.2cm]
& = &
3 e^{\e} K^2_\k
 E N_{\e;II}(t\wedge t_{\mathrm{st}}).
\\[0.2cm]
\end{array}
\end{equation}
For convenience, we introduce the notation
\begin{equation}
\tau^+(\tau') = \min\{ \tau_i(t\wedge t_{\mathrm{st}}),\tau'+1 \} .
\end{equation}
Substituting $\tau=\tau_i(s)$ and integrating by parts,
we may compute
\begin{equation}
\begin{array}{lcl}
 \mathcal{I}_{II} & = &
 E  \int_0^{\tau_i(t\wedge t_{\mathrm{st}})} e^{-\e t}
\int_{\tau'}^{  \tau^+(\tau')}
                      e^{\e t_i(\tau)}
                       \partial_1 \mathcal{E}_{\tau, \tau';L^2}
      \, d\tau
                       \, d\tau '
 \\[0.2cm]
 & = &
   \mathcal{I}_{II;A}
   + \mathcal{I}_{II;B}
   + \mathcal{I}_{II;C},
\end{array}
\end{equation}
in which we have introduced the expressions
\begin{equation}
\begin{array}{lcl}
\mathcal{I}_{II;A}
 & = &
   E \int_0^{\tau_i(t\wedge t_{\mathrm{st}})} e^{-\e t}
     e^{\e t_i(\tau^+(\tau')) } \mathcal{E}_{\tau^+(\tau'), \tau' ; L^2}
         \, d\tau' ,
 \\[0.2cm]
\mathcal{I}_{II;B}
 & = &
   - E  \int_0^{\tau_i(t\wedge t_{\mathrm{st}})} e^{-\e t}
     e^{\e t_i(\tau') } \mathcal{E}_{\tau', \tau' ; L^2}
     %
       \, d\tau' ,
\\[0.2cm]
\mathcal{I}_{II;C}
 & = &
   - E  \int_0^{\tau_i(t\wedge t_{\mathrm{st}})} e^{-\e t}
 \int_{\tau'}^{\tau^+(\tau')  }
                     \left(\frac{d}{d\tau} e^{\e t_i(\tau)}\right)
               \mathcal{E}_{\tau, \tau';L^2}
     \, d\tau    \, d\tau ' .
 \\[0.2cm]
\end{array}
\end{equation}

Upon computing
\begin{equation}
\int_{\tau'}^{  \tau^+(\tau') }
                      \abs{ \frac{d}{d\tau} e^{\e t_i(\tau)}  } \, d\tau
=  e^{\e t_i(\tau) } \big|_{\tau'}^{\tau^+(\tau')}
\le 2 e^{\e} e^{ \e t_i(\tau')},
\end{equation}
we can make the substitution $s' = t_i(\tau')$
and obtain the final estimate
\begin{equation}
\begin{array}{lcl}
\abs{\mathcal{I}_{II} }
& \leq &
  4 e^\e K_\k
  M^2 K_{B;\mathrm{lin}}^2 E\int_0^{t\wedge t_{\mathrm{st}}}
    e^{- \e(t\wedge t_{\mathrm{st}}-s')} \nrm{V(s')}^2_{H^1}ds' \\[0.2cm]
  &\leq& 4 e^\e K_\k
  M^2 K_{B;\mathrm{lin}}^2 E N_{\e;II}(t\wedge t_{\mathrm{st}}).
\end{array}
\end{equation}
\end{proof}

\begin{lem}
\label{lem:nls:b:cn}
Fix $T > 0$ and assume that (HDt), (HSt) and (HTw) all hold.
Pick a constant $0 < \e <  \beta$.
Then for
any $0 \leq \delta < 1$,
any $(\mathcal{F}_t)$-stopping time
$t_{\mathrm{st}}$ and any $0 \le t \le T$,
we have the bounds
\begin{equation}
\begin{array}{lcl}
E \, \mathcal{I}^{\mathrm{lt}}_{\e,\delta;B;\mathrm{cn}}(t)
 & \le & \frac{M^2}{(2 \beta-\e)\e} K_{B;\mathrm{cn}}^2 ,
\\[0.2cm]
E \, \mathcal{I}^{\mathrm{sh}}_{\e,\delta;B;\mathrm{cn}}(t)
 & \le &
 \frac{1}{\e} K_\k K_{B;\mathrm{lin}}^2 M^2 (1 + \rho_\mathrm{min}^{-1})e^\e( 3 K_\k  + 2 ).
\end{array}
\end{equation}
\end{lem}
\begin{proof}
These results follows by repeating Lemmas \ref{lem:nls:b:lin} and \ref{lem:nls:b:lin:st}.
Since
\begin{equation}
\int_0^{t}
  e^{- \e(t - s) } \,ds
\le \frac{1}{\e},
\end{equation}
we can obtain the bounds
by making the substitution
\begin{equation}
  K_{B;\mathrm{lin}} \mapsto K_{B;\mathrm{cn}},
  \qquad
  \qquad
  E N_{\e;II}(t\wedge t_{\mathrm{st}}) \mapsto \frac{1}{\e}.
\end{equation}
\end{proof}

\begin{lem}
\label{lem:nls:st:sup:on:b:lin:i}
Fix $T > 0$ and assume that (HDt), (HSt) and (HTw) all hold.
Pick a constant $0< \e <2\beta$, then for any $0\leq \delta <1$ we have the bounds
\begin{equation}
\begin{array}{lcl}
E \sup_{0 \le t \le t_{\mathrm{st}}}
    \mathcal{I}^{\mathrm{lt}}_{\e,\delta;B; \mathrm{lin}}(t)
& \le &
  e^{\e} (T + 1) \frac{M^2}{2 \beta - \e} K_\k K_{B;\mathrm{lin}}^2
  E  \sup_{0 \le t \le t_{\mathrm{st}}} N_{\e;II}(t) ,
\\[0.2cm]
E \sup_{0 \le t \le t_{\mathrm{st}}}
    \mathcal{I}^{\mathrm{sh}}_{\e,\delta;B; \mathrm{lin}}(t)
& \le &
  e^{\e} (T + 1)
   K_{B;\mathrm{lin}}^2 M^2 (1 + \rho^{-1}) e^{\e}( 3K_\k  + 2 )
  E  \sup_{0 \le t \le t_{\mathrm{st}}} N_{\e;II}(t) ,
\end{array}
\end{equation}
and
\begin{equation}
\begin{array}{lcl}
E \sup_{0 \le t \le t_{\mathrm{st}}}
    \mathcal{I}^{\mathrm{lt}}_{\e,\delta;B; \mathrm{cn}}(t)
& \le &
  e^{\e} (T + 1)
\frac{M^2}{(2 \beta - \e)\e} K_\k K_{B;\mathrm{cn}}^2 ,
\\[0.2cm]
E \sup_{0 \le t \le t_{\mathrm{st}}}
    \mathcal{I}^{\mathrm{sh}}_{\e,\delta;B; \mathrm{cn}}(t)
& \le &
  e^{\e} (T + 1)
K_\k K_{B;\mathrm{cn}}^2  \frac{M^2}{\e} (1 + \rho^{-1}) e^{\e}
   ( 3K_\k  + 2 ) .
\end{array}
\end{equation}

\end{lem}
\begin{proof}
This follows directly from Lemmas 9.20 and 9.21 in \cite{Hamster2017}.
\end{proof}

\begin{proof}[Proof of Proposition \ref{prp:nls:general}]
Pick $T >0$ and $0 < \eta < \eta_0$
and write $t_{\mathrm{st}} = t_{\mathrm{st}}(T,\e, \eta)$.
Since the identities \sref{eq:MildNonProjZero}
with $v = V(t\wedge t_{\mathrm{st}})$ hold for all $0 \le t \le T$,
we may compute
\begin{equation}
\begin{array}{lcl}
E\sup_{0\leq t\leq t_{\mathrm{st}}}[N^i_{\e;I}(t)]
& \le &
7 E \sup_{0\leq t\leq t_{\mathrm{st}}}
\Big[
\norm{\mathcal{E}_0(t)}_{L^2}^2
+ \sigma^4 \norm{\mathcal{E}_{F;\mathrm{lin}}(t)}_{L^2}^2
+ \norm{\mathcal{E}_{F;\mathrm{nl}}(t)}_{L^2}^2
\\[0.2cm]
& & \qquad
+ \sigma^2 \norm{\mathcal{E}_{B;\mathrm{lin}}(t)}_{L^2}^2
+ \sigma^2 \norm{\mathcal{E}_{B;\mathrm{cn}}(t)}_{L^2}^2
\\[0.2cm]
& & \qquad
+ \norm{\mathcal{E}^{\mathrm{lt}}_{\mathrm{so}}(t)}_{L^2}^2
+ \norm{\mathcal{E}^{\mathrm{st}}_{\mathrm{so}}(t)}_{L^2}^2
\Big]
\end{array}
\end{equation}
by applying Young's inequality.
The inequalities in
Lemmas \ref{lem:ColPrevRes}-\ref{lem:nls:st:sup:on:b:lin:i} now imply that
\begin{equation}
\begin{array}{lcl}
E\sup_{0\leq t\leq t_{\mathrm{st}}}[N^i_{\e;I}(t)]
& \le &
C_1 \big[ \norm{V(0)}_{H^1}^2
  + (\eta+\sigma^2T + \sigma^4) \sup_{0\leq t\leq t_{\mathrm{st}}}N_{\e;II}(t)
\big].
\end{array}
\end{equation}
In addition,
we note that
\begin{equation}
\begin{array}{lcl}
E\sup_{0\leq t\leq t_{\mathrm{st}}} N^i_{\e,0;II}(t)
& \le &
11 E \sup_{0\leq t\leq t_{\mathrm{st}}}
\Big[
  \mathcal{I}_{\e,0;0}(t)
+ \sigma^4 \mathcal{I}^{\mathrm{lt}}_{\e,0;F;\mathrm{lin}}(t)
+ \sigma^4 \mathcal{I}^{\mathrm{sh}}_{\e,0;F;\mathrm{lin}}(t)
\\[0.2cm]
& & \qquad \qquad
+ \mathcal{I}^{\mathrm{lt}}_{\e,0;F;\mathrm{nl}}(t)
+ \mathcal{I}^{\mathrm{sh}}_{\e,0;F;\mathrm{nl}}(t)
\\[0.2cm]
& & \qquad \qquad
+ \sigma^2 \mathcal{I}^{\mathrm{lt}}_{\e,0;B;\mathrm{lin}}(t)
+ \sigma^2 \mathcal{I}^{\mathrm{sh}}_{\e,0;B;\mathrm{lin}}(t)
\\[0.2cm]
& & \qquad \qquad
+\sigma^2 \mathcal{I}^{\mathrm{lt}}_{\e,0;B;\mathrm{cn}}(t)
+ \sigma^2 \mathcal{I}^{\mathrm{sh}}_{\e,0;B;\mathrm{cn}}(t)
\\[0.2cm]
& & \qquad \qquad
+ \mathcal{I}^{\mathrm{lt}}_{\e,0;\mathrm{so}}(t)
+ \mathcal{I}^{\mathrm{sh}}_{\e,0;\mathrm{so}}(t)
\Big] .
\end{array}
\end{equation}
The inequalities in
Lemmas \ref{lem:ColPrevRes}-\ref{lem:nls:b:cn}
now imply that
\begin{equation}
\begin{array}{lcl}
E \sup_{0\leq t\leq t_{\mathrm{st}}}N^i_{\e,0;II}(t)
& \le &
C_2 \big[ \norm{V(0)}_{H^1}^2
  + \sigma^2 T  + (\eta + \sigma^2 T + \sigma^4)\sup_{0\leq t\leq t_{\mathrm{st}}} N_{\e;II}(t)
\big].
\end{array}
\end{equation}
In particular, we see that
\begin{equation}
E \sup_{0\leq t\leq t_{\mathrm{st}}} N_\e^{i}(t)
\le C_3 \big[
  \norm{V(0)}_{H^1}^2
  + \sigma^2T + (\eta + \sigma^2 T + \sigma^4)
   E\sup_{0\leq t\leq t_{\mathrm{st}}} N_{\e}(t)
\big] .
\end{equation}
The desired bound hence follows by summing over $i$ and appropriately restricting the size of
$\eta + \sigma^2T + \sigma^4$.
\end{proof}

\bibliographystyle{klunumHJ}
\bibliography{ref}

\end{document}

%% file: csco.pdf_tex
\begingroup%
  \makeatletter%
  \providecommand\color[2][]{%
    \errmessage{(Inkscape) Color is used for the text in Inkscape, but the package 'color.sty' is not loaded}%
    \renewcommand\color[2][]{}%
  }%
  \providecommand\transparent[1]{%
    \errmessage{(Inkscape) Transparency is used (non-zero) for the text in Inkscape, but the package 'transparent.sty' is not loaded}%
    \renewcommand\transparent[1]{}%
  }%
  \providecommand\rotatebox[2]{#2}%
  \newcommand*\fsize{\dimexpr\f@size pt\relax}%
  \newcommand*\lineheight[1]{\fontsize{\fsize}{#1\fsize}\selectfont}%
  \ifx\svgwidth\undefined%
    \setlength{\unitlength}{1211.25bp}%
    \ifx\svgscale\undefined%
      \relax%
    \else%
      \setlength{\unitlength}{\unitlength * \real{\svgscale}}%
    \fi%
  \else%
    \setlength{\unitlength}{\svgwidth}%
  \fi%
  \global\let\svgwidth\undefined%
  \global\let\svgscale\undefined%
  \makeatother%
  \begin{picture}(1,0.60123839)%
    \lineheight{1}%
    \setlength\tabcolsep{0pt}%
    \put(0,0){\includegraphics[width=\unitlength,page=1]{csco.pdf}}%
  \end{picture}%
\endgroup%

%% file: Phi.pdf_tex
\begingroup%
  \makeatletter%
  \providecommand\color[2][]{%
    \errmessage{(Inkscape) Color is used for the text in Inkscape, but the package 'color.sty' is not loaded}%
    \renewcommand\color[2][]{}%
  }%
  \providecommand\transparent[1]{%
    \errmessage{(Inkscape) Transparency is used (non-zero) for the text in Inkscape, but the package 'transparent.sty' is not loaded}%
    \renewcommand\transparent[1]{}%
  }%
  \providecommand\rotatebox[2]{#2}%
  \newcommand*\fsize{\dimexpr\f@size pt\relax}%
  \newcommand*\lineheight[1]{\fontsize{\fsize}{#1\fsize}\selectfont}%
  \ifx\svgwidth\undefined%
    \setlength{\unitlength}{1211.25bp}%
    \ifx\svgscale\undefined%
      \relax%
    \else%
      \setlength{\unitlength}{\unitlength * \real{\svgscale}}%
    \fi%
  \else%
    \setlength{\unitlength}{\svgwidth}%
  \fi%
  \global\let\svgwidth\undefined%
  \global\let\svgscale\undefined%
  \makeatother%
  \begin{picture}(1,0.60123839)%
    \lineheight{1}%
    \setlength\tabcolsep{0pt}%
    \put(0,0){\includegraphics[width=\unitlength,page=1]{Phi.pdf}}%
  \end{picture}%
\endgroup%

%% file: Uc0.pdf_tex
\begingroup%
  \makeatletter%
  \providecommand\color[2][]{%
    \errmessage{(Inkscape) Color is used for the text in Inkscape, but the package 'color.sty' is not loaded}%
    \renewcommand\color[2][]{}%
  }%
  \providecommand\transparent[1]{%
    \errmessage{(Inkscape) Transparency is used (non-zero) for the text in Inkscape, but the package 'transparent.sty' is not loaded}%
    \renewcommand\transparent[1]{}%
  }%
  \providecommand\rotatebox[2]{#2}%
  \newcommand*\fsize{\dimexpr\f@size pt\relax}%
  \newcommand*\lineheight[1]{\fontsize{\fsize}{#1\fsize}\selectfont}%
  \ifx\svgwidth\undefined%
    \setlength{\unitlength}{1211.25bp}%
    \ifx\svgscale\undefined%
      \relax%
    \else%
      \setlength{\unitlength}{\unitlength * \real{\svgscale}}%
    \fi%
  \else%
    \setlength{\unitlength}{\svgwidth}%
  \fi%
  \global\let\svgwidth\undefined%
  \global\let\svgscale\undefined%
  \makeatother%
  \begin{picture}(1,0.60123839)%
    \lineheight{1}%
    \setlength\tabcolsep{0pt}%
    \put(0,0){\includegraphics[width=\unitlength,page=1]{Uc0.pdf}}%
  \end{picture}%
\endgroup%

%% file: Ucs.pdf_tex
\begingroup%
  \makeatletter%
  \providecommand\color[2][]{%
    \errmessage{(Inkscape) Color is used for the text in Inkscape, but the package 'color.sty' is not loaded}%
    \renewcommand\color[2][]{}%
  }%
  \providecommand\transparent[1]{%
    \errmessage{(Inkscape) Transparency is used (non-zero) for the text in Inkscape, but the package 'transparent.sty' is not loaded}%
    \renewcommand\transparent[1]{}%
  }%
  \providecommand\rotatebox[2]{#2}%
  \newcommand*\fsize{\dimexpr\f@size pt\relax}%
  \newcommand*\lineheight[1]{\fontsize{\fsize}{#1\fsize}\selectfont}%
  \ifx\svgwidth\undefined%
    \setlength{\unitlength}{1211.25bp}%
    \ifx\svgscale\undefined%
      \relax%
    \else%
      \setlength{\unitlength}{\unitlength * \real{\svgscale}}%
    \fi%
  \else%
    \setlength{\unitlength}{\svgwidth}%
  \fi%
  \global\let\svgwidth\undefined%
  \global\let\svgscale\undefined%
  \makeatother%
  \begin{picture}(1,0.60123839)%
    \lineheight{1}%
    \setlength\tabcolsep{0pt}%
    \put(0,0){\includegraphics[width=\unitlength,page=1]{Ucs.pdf}}%
  \end{picture}%
\endgroup%

%% file: UGamma.pdf_tex
\begingroup%
  \makeatletter%
  \providecommand\color[2][]{%
    \errmessage{(Inkscape) Color is used for the text in Inkscape, but the package 'color.sty' is not loaded}%
    \renewcommand\color[2][]{}%
  }%
  \providecommand\transparent[1]{%
    \errmessage{(Inkscape) Transparency is used (non-zero) for the text in Inkscape, but the package 'transparent.sty' is not loaded}%
    \renewcommand\transparent[1]{}%
  }%
  \providecommand\rotatebox[2]{#2}%
  \newcommand*\fsize{\dimexpr\f@size pt\relax}%
  \newcommand*\lineheight[1]{\fontsize{\fsize}{#1\fsize}\selectfont}%
  \ifx\svgwidth\undefined%
    \setlength{\unitlength}{1211.25bp}%
    \ifx\svgscale\undefined%
      \relax%
    \else%
      \setlength{\unitlength}{\unitlength * \real{\svgscale}}%
    \fi%
  \else%
    \setlength{\unitlength}{\svgwidth}%
  \fi%
  \global\let\svgwidth\undefined%
  \global\let\svgscale\undefined%
  \makeatother%
  \begin{picture}(1,0.60123839)%
    \lineheight{1}%
    \setlength\tabcolsep{0pt}%
    \put(0,0){\includegraphics[width=\unitlength,page=1]{UGamma.pdf}}%
  \end{picture}%
\endgroup%

%% file: rico2.pdf_tex
\begingroup%
  \makeatletter%
  \providecommand\color[2][]{%
    \errmessage{(Inkscape) Color is used for the text in Inkscape, but the package 'color.sty' is not loaded}%
    \renewcommand\color[2][]{}%
  }%
  \providecommand\transparent[1]{%
    \errmessage{(Inkscape) Transparency is used (non-zero) for the text in Inkscape, but the package 'transparent.sty' is not loaded}%
    \renewcommand\transparent[1]{}%
  }%
  \providecommand\rotatebox[2]{#2}%
  \newcommand*\fsize{\dimexpr\f@size pt\relax}%
  \newcommand*\lineheight[1]{\fontsize{\fsize}{#1\fsize}\selectfont}%
  \ifx\svgwidth\undefined%
    \setlength{\unitlength}{1211.25bp}%
    \ifx\svgscale\undefined%
      \relax%
    \else%
      \setlength{\unitlength}{\unitlength * \real{\svgscale}}%
    \fi%
  \else%
    \setlength{\unitlength}{\svgwidth}%
  \fi%
  \global\let\svgwidth\undefined%
  \global\let\svgscale\undefined%
  \makeatother%
  \begin{picture}(1,0.60123839)%
    \lineheight{1}%
    \setlength\tabcolsep{0pt}%
    \put(0,0){\includegraphics[width=\unitlength,page=1]{rico2.pdf}}%
  \end{picture}%
\endgroup%

%% file: AvRico2.pdf_tex
\begingroup%
  \makeatletter%
  \providecommand\color[2][]{%
    \errmessage{(Inkscape) Color is used for the text in Inkscape, but the package 'color.sty' is not loaded}%
    \renewcommand\color[2][]{}%
  }%
  \providecommand\transparent[1]{%
    \errmessage{(Inkscape) Transparency is used (non-zero) for the text in Inkscape, but the package 'transparent.sty' is not loaded}%
    \renewcommand\transparent[1]{}%
  }%
  \providecommand\rotatebox[2]{#2}%
  \newcommand*\fsize{\dimexpr\f@size pt\relax}%
  \newcommand*\lineheight[1]{\fontsize{\fsize}{#1\fsize}\selectfont}%
  \ifx\svgwidth\undefined%
    \setlength{\unitlength}{1211.25bp}%
    \ifx\svgscale\undefined%
      \relax%
    \else%
      \setlength{\unitlength}{\unitlength * \real{\svgscale}}%
    \fi%
  \else%
    \setlength{\unitlength}{\svgwidth}%
  \fi%
  \global\let\svgwidth\undefined%
  \global\let\svgscale\undefined%
  \makeatother%
  \begin{picture}(1,0.60123839)%
    \lineheight{1}%
    \setlength\tabcolsep{0pt}%
    \put(0,0){\includegraphics[width=\unitlength,page=1]{AvRico2.pdf}}%
  \end{picture}%
\endgroup%

%% file: FHNpaper.bbl
\begin{thebibliography}{000}

\bibitem{alexander1990topological}
J. Alexander, R. Gardner and C. Jones (1990), A topological invariant arising
  in the stability analysis of travelling waves.
\newblock {\em J. reine angew. Math} {\bf 410}(167-212), 143.

\bibitem{HJHNLS}
M. Beck, H.~J. Hupkes, B. Sandstede and K. Zumbrun (2010), Nonlinear
  {S}tability of {S}emidiscrete {S}hocks for {T}wo-{S}ided {S}chemes.
\newblock {\em SIAM J. Math. Anal.} {\bf 42}, 857--903.

\bibitem{beck2010nonlinear}
M. Beck, B. Sandstede and K. Zumbrun (2010), Nonlinear stability of
  time-periodic viscous shocks.
\newblock {\em Archive for rational mechanics and analysis} {\bf 196}(3),
  1011--1076.

\bibitem{BHM}
H. Berestycki, F. Hamel and H. Matano (2009), Bistable traveling waves around
  an obstacle.
\newblock {\em Comm. Pure Appl. Math.} {\bf 62}(6), 729--788.

\bibitem{Bloemker}
L.~A. Bianchi, D. Bl\"omker and P. Wacker (2017), Pattern size in Gaussian
  fields from spinodal decomposition.
\newblock {\em SIAM Journal on Applied Mathematics} {\bf 77}(4), 1292--1319.

\bibitem{brassesco1995brownian}
S. Brassesco, A. De~Masi and E. Presutti (1995), Brownian fluctuations of the
  interface in the D=1 Ginzburg-Landau equation with noise.
\newblock {\em Ann. Inst. H. Poincar{\'e} Probab. Statist} {\bf 31}(1),
  81--118.

\bibitem{bressloff2015nonlinear}
P.~C. Bressloff and Z.~P. Kilpatrick (2015), Nonlinear Langevin equations for
  wandering patterns in stochastic neural fields.
\newblock {\em SIAM Journal on Applied Dynamical Systems} {\bf 14}(1),
  305--334.

\bibitem{Bressloff}
P.~C. Bressloff and M.~A. Webber (2012), Front propagation in stochastic neural
  fields.
\newblock {\em SIAM Journal on Applied Dynamical Systems} {\bf 11}(2),
  708--740.

\bibitem{cartwright2019}
M. Cartwright and G.~A. Gottwald (2019), A collective coordinate framework to
  study the dynamics of travelling waves in stochastic partial differential
  equations.
\newblock {\em Physica D: Nonlinear Phenomena}.

\bibitem{chen2015traveling}
C.-N. Chen and Y. Choi (2015), Traveling pulse solutions to FitzHugh--Nagumo
  equations.
\newblock {\em Calculus of Variations and Partial Differential Equations} {\bf
  54}(1), 1--45.

\bibitem{cornwell2017opening}
P. Cornwell (2017), Opening the Maslov Box for Traveling Waves in Skew-Gradient
  Systems.
\newblock {\em arXiv preprint arXiv:1709.01908}.

\bibitem{cornwell2017existence}
P. Cornwell and C.~K. Jones (2017), On the Existence and Stability of Fast
  Traveling Waves in a Doubly-Diffusive FitzHugh-Nagumo System.
\newblock {\em arXiv preprint arXiv:1709.09132}.

\bibitem{DaPratomild}
G. Da~Prato, A. Jentzen and M. R{\"o}ckner (2010), A mild {I}t{\^o} formula for
  {SPDE}s.
\newblock {\em arXiv preprint arXiv:1009.3526}.

\bibitem{NunnoAdvMathFinance2011}
G. di~Nunno and B.~O. (editors) (2011), {\em Advanced Mathematical Methods for
  Finance}.
\newblock Springer.

\bibitem{Climate}
C.~L.~E. Franzke, T.~J. O'Kane, J. Berner, P.~D. Williams and V. Lucarini
  (2015), Stochastic climate theory and modeling.
\newblock {\em Wiley Interdisciplinary Reviews: Climate Change} {\bf 6}(1),
  63--78.

\bibitem{funaki1995scaling}
T. Funaki (1995), The scaling limit for a stochastic PDE and the separation of
  phases.
\newblock {\em Probability Theory and Related Fields} {\bf 102}(2), 221--288.

\bibitem{Garcia2001}
J. Garc{\'\i}a-Ojalvo, F. Sagu{\'e}s, J.~M. Sancho and L. Schimansky-Geier
  (2001), Noise-enhanced excitability in bistable activator-inhibitor media.
\newblock {\em Physical Review E} {\bf 65}(1), 011105.

\bibitem{gowda2015early}
K. Gowda and C. Kuehn (2015), Early-warning signs for pattern-formation in
  stochastic partial differential equations.
\newblock {\em Communications in Nonlinear Science and Numerical Simulation}
  {\bf 22}(1), 55--69.

\bibitem{Hamster2017}
C.~H.~S. Hamster and H.~J. Hupkes (2019), Stability of Traveling Waves for
  Reaction-Diffusion Equations with Multiplicative Noise.
\newblock {\em SIAM Journal on Applied Dynamical Systems} {\bf 18}(1),
  205--278.

\bibitem{Hamster2020}
C.~H.~S. Hamster and H.~J. Hupkes (2020), Travelling waves for
  reaction--diffusion equations forced by translation invariant noise.
\newblock {\em Physica D: Nonlinear Phenomena} {\bf 401}, 132233.

\bibitem{HJHSTB2D}
A. Hoffman, H. Hupkes and E. Van~Vleck (2015), Multi-dimensional Stability of
  Waves Travelling through Rectangular Lattices in Rational Directions.
\newblock {\em Transactions of the American Mathematical Society} {\bf
  367}(12), 8757--8808.

\bibitem{HJHOBST2D}
A. Hoffman, H. Hupkes and E. Van~Vleck (2017), {\em Entire {S}olutions for
  {B}istable {L}attice {D}ifferential {E}quations with {O}bstacles}.
\newblock American Mathematical Society.

\bibitem{Inglis}
J. Inglis and J. MacLaurin (2016), A general framework for stochastic traveling
  waves and patterns, with application to neural field equations.
\newblock {\em SIAM Journal on Applied Dynamical Systems} {\bf 15}(1),
  195--234.

\bibitem{KAP1997}
T. Kapitula (1997), Multidimensional {S}tability of {P}lanar {T}ravelling
  {W}aves.
\newblock {\em Trans. Amer. Math. Soc.} {\bf 349}, 257--269.

\bibitem{KatjaClaudia}
C. Knoche and K. Frieler (2001), Solutions of stochastic differential equations
  in infinite dimensional Hilbert spaces and their dependence on initial data.
\newblock Diplomarbeit, {B}i{B}o{S}-{P}reprint {E}02-04-083, Bielefeld
  University.

\bibitem{Kuske2017}
R. Kuske, C. Lee and V. Rottsch{\"a}fer (2017), Patterns and coherence
  resonance in the stochastic Swift-Hohenberg equation with Pyragas control:
  The Turing bifurcation case.
\newblock {\em Physica D: Nonlinear Phenomena} pp.~--.

\bibitem{Lang}
E. Lang (2016), A multiscale analysis of traveling waves in stochastic neural
  fields.
\newblock {\em SIAM Journal on Applied Dynamical Systems} {\bf 15}(3),
  1581--1614.

\bibitem{leadbetter2012extremes}
M.~R. Leadbetter, G. Lindgren and H. Rootz{\'e}n (2012), {\em Extremes and
  related properties of random sequences and processes}.
\newblock Springer Science \& Business Media.

\bibitem{LiuRockner}
W. Liu and M. R{\"o}ckner (2010), {SPDE} in {H}ilbert space with locally
  monotone coefficients.
\newblock {\em Journal of Functional Analysis} {\bf 259}(11), 2902--2922.

\bibitem{Lord2012}
G. Lord and V. Th{\"u}mmler (2012), Computing stochastic traveling waves.
\newblock {\em SIAM Journal on Scientific Computing} {\bf 34}(1), B24--B43.

\bibitem{lorenzi2004analytic}
L. Lorenzi, A. Lunardi, G. Metafune and D. Pallara (2004), Analytic semigroups
  and reaction-diffusion problems.
\newblock In: {\em Internet Seminar}, Vol. 2005.
\newblock p. 127.

\bibitem{MasciaZumbrun02}
C. Mascia and K. Zumbrun (2002), Pointwise {G}reen's function bounds and
  stability of relaxation shocks.
\newblock {\em Indiana Univ. Math. J.} {\bf 51}(4), 773--904.

\bibitem{Shardlow}
T. Shardlow (2005), Numerical simulation of stochastic PDEs for excitable
  media.
\newblock {\em Journal of computational and applied mathematics} {\bf 175}(2),
  429--446.

\bibitem{Stannat}
W. Stannat (2013), Stability of travelling waves in stochastic {N}agumo
  equations.
\newblock {\em arXiv preprint arXiv:1301.6378}.

\bibitem{stannat2014stability}
W. Stannat (2014), Stability of travelling waves in stochastic bistable
  reaction-diffusion equations.
\newblock {\em arXiv preprint arXiv:1404.3853}.

\bibitem{veraar2011note}
M. Veraar and L. Weis (2011), A note on maximal estimates for stochastic
  convolutions.
\newblock {\em Czechoslovak mathematical journal} {\bf 61}(3), 743.

\bibitem{vinals1991numerical}
J. Vi{\~n}als, E. Hern{\'a}ndez-Garc{\'\i}a, M. San~Miguel and R. Toral (1991),
  Numerical study of the dynamical aspects of pattern selection in the
  stochastic Swift-Hohenberg equation in one dimension.
\newblock {\em Physical Review A} {\bf 44}(2), 1123.

\bibitem{Zhang}
J. Zhang, A. Holden, O. Monfredi, M. Boyett and H. Zhang (2009), Stochastic
  vagal modulation of cardiac pacemaking may lead to erroneous identification
  of cardiac â€œchaosâ€.
\newblock {\em Chaos: An Interdisciplinary Journal of Nonlinear Science} {\bf
  19}(2), 028509.

\bibitem{Zumbrun2009}
K. Zumbrun (2011), Instantaneous {S}hock {L}ocation and {O}ne-{D}imensional
  {N}onlinear {S}tability of {V}iscous {S}hock {W}aves.
\newblock {\em Quarterly of applied mathematics} {\bf 69}(1), 177--202.

\end{thebibliography}
